\documentclass[aos,preprint,reqno]{imsart}
	\usepackage[utf8]{inputenc}
	\usepackage{amsmath}
	\usepackage{amsfonts,amssymb}
	\usepackage{dsfont}
	\usepackage{graphicx}
	\usepackage[dvipsnames,svgnames,x11names]{xcolor}
	\usepackage{mathrsfs,booktabs}
	\usepackage{nicefrac}
	\usepackage{algorithm}
	\usepackage{algorithmic}

	\RequirePackage[authoryear,sort]{natbib} 
	\RequirePackage[colorlinks,citecolor=blue,urlcolor=blue]{hyperref}

	\usepackage[listings,skins,theorems,breakable,most]{tcolorbox}
    \startlocaldefs
	\definecolor{mycolor}{rgb}{0.122, 0.435, 0.698}
	\usepackage{cleveref}
	\usepackage{autonum}

	\newtcbox{\mybox}{
		nobeforeafter,
		colframe=mycolor,
		colback=mycolor!10!white,
		boxrule=0.5pt,
		arc=4pt,
		boxsep=0pt,
		left=6pt,
		right=6pt,
		top=6pt,
		bottom=6pt,
		box align=center,
			halign=center,	
		tcbox raise base
		}

		\newtcbtheorem[auto counter,
			crefname={fact}{Fact},		
			Crefname={fact}{Fact} ]{fact}{Fact}{%
		fonttitle = \bfseries,
		colframe = NavyBlue,%
		colback = blue!5%
		}{fact}
		
		\newtcbtheorem[auto counter,
			crefname={alg}{Algorithm},		
			Crefname={alg}{Algorithm} ]{alg}{Algorithm}{%
		fonttitle = \bfseries,
		}{algo}
		
	\usepackage{tikz}
	\usetikzlibrary{arrows}
	
	\graphicspath{
    {pics/exp1/}
    {pics/exp2/}
    {pics/exp3/}
    {pics/exp4/}
    {pics/exp5/}
}

\usetikzlibrary{calc,trees,positioning,arrows,chains,shapes.geometric,%
    decorations.pathreplacing,decorations.pathmorphing,shapes,%
    matrix,shapes.symbols}

\tikzset{
>=stealth',
  punktchain/.style={
    rectangle, 
    rounded corners, 
    fill=green!30,
    draw=blue, very thick,
    text width=10em, 
    minimum height=3em, 
    text centered, 
    on chain},
  line/.style={blue,draw, thick, <-},
  element/.style={
    tape,
    top color=white,
    bottom color=blue!50!black!60!,
    minimum width=8em,
    draw=blue!40!black!90, very thick,
    text width=10em, 
    minimum height=3.5em, 
    text centered, 
    on chain},
  every join/.style={blue,->, thick,shorten >=1pt},
  decoration={brace},
  tuborg/.style={decorate},
  tubnode/.style={midway, right=2pt},
}

	\newcommand\bs{\boldsymbol}
	\newcommand\bmu{\boldsymbol\mu}
	\newcommand\bxi{\boldsymbol\xi}
	\newcommand\bzeta{\boldsymbol\zeta}
	\newcommand\bSigma{\boldsymbol\Sigma}
	\newcommand\bDelta{\boldsymbol\Delta}

	\newcommand\bfE{\mathbf E}
	\newcommand\bfP{\mathbf P}
	\newcommand\bfQ{\mathbf Q}
	\newcommand\bfU{\mathbf U}
	
	\newcommand\bfSigma{\mathbf\Sigma}

	\newcommand\bfM{\mathbf M}
	\newcommand\bfA{\mathbf A}
	\newcommand\bfB{\mathbf B}

	\newcommand\bX{\boldsymbol X}
	
	\newcommand\bY{\boldsymbol Y}
	
	\newcommand\bv{\boldsymbol v}
	\newcommand\bu{\boldsymbol u}
	\newcommand\bw{\boldsymbol w}

	\newcommand\RR{\mathbb R}

	\newcommand\bfI{\mathbf I}

	\def\GMmu{\hat\bmu_n^{\textup{GM}}}
	\def\tr{\mathop{\textup{Tr}}}
	\newcommand\rSigma{\textbf{\textup{r}}_{\scriptscriptstyle\bSigma}}

	\newcommand\btheta{\boldsymbol\theta}

	\def\tilde{\widetilde}
	\def\hat{\widehat}
	\def\ds{\displaystyle}
	
	\usepackage{amsthm}
	\theoremstyle{plain}
	\newtheorem{theorem}{Theorem}

	\Crefname{fact}{Fact}{Facts}
	\newtheorem{proposition}{Proposition}
	\newtheorem{lemma}{Lemma}
	\theoremstyle{remark}
	\newtheorem{definition}{Definition}

\endlocaldefs

\begin{document}

	\begin{frontmatter}

	\title{All-In-One Robust Estimator of the Gaussian Mean }

	\runtitle{Robust estimation of a Gaussian mean}

	\begin{aug}
	\author[A]{\fnms{Arnak S.} \snm{Dalalyan}\ead[label=e1]{arnak.dalalyan@ensae.fr}}
	\and
	\author[B]{\fnms{Arshak} \snm{Minasyan}\ead[label=e2]{minasyan@yerevann.com}}
	\runauthor{Dalalyan and Minasyan}

    \address[A]{ENSAE-CREST, \printead{e1}}
    \address[B]{Yerevan State University, YerevaNN, \printead{e2}}

	\end{aug}


	\begin{abstract}
	The goal of this paper is to show that a single robust estimator of 
	the mean of a multivariate Gaussian distribution can enjoy five
	desirable properties. First, it is computationally
	tractable in the sense that it can be computed in a time which 
	is at most	polynomial in dimension, sample size and the logarithm 
	of the inverse of the contamination rate. Second, it is equivariant 
	by translations, uniform scaling and orthogonal transformations. 
	Third, it has a high breakdown point equal to $0.5$, and a
	nearly-minimax-rate-breakdown point approximately equal to $0.28$.
	Fourth, it is minimax rate optimal, up to a logarithmic factor, 
	when data consists of independent observations corrupted by
	adversarially chosen outliers. Fifth, it is asymptotically efficient
	when the rate of contamination tends to zero. The estimator is
	obtained by an 
	iterative reweighting approach. Each sample point is assigned a 
	weight that is iteratively updated by solving a convex optimization 
	problem. We also establish a dimension-free non-asymptotic risk 
	bound for the expected error of the proposed estimator. It is 
	the first result of this kind in the literature and involves 
	only the effective rank of the covariance matrix. Finally, we
	show that the obtained results can be extended to sub-Gaussian
	distributions, as well as to the cases of unknown rate of contamination
	 or unknown covariance matrix. 
	\end{abstract}



	\end{frontmatter}

	\maketitle

	\itemsep=0pt
	\setcounter{tocdepth}{1}
	\tableofcontents

	\section{Introduction}\label{sec:intro}

    Robust estimation is one of the most fundamental 
    problems in statistics. Its goal is to design efficient 
    methods capable of processing data sets contaminated 
    by outliers, so that these outliers have little influence 
    on the final result. The notion of an outlier is hard to 
    define for a single data point. It is also hard, inefficient 
    and often impossible to clean data by removing the outliers. 
    Instead, one can build methods that take as input the 
    contaminated data set and provide as output an estimate 
    which is not very sensitive to the contamination. Recent 
    advances in data acquisition and computational power 
    provoked a revival of interest in robust estimation and 
    learning, with a focus on finite sample results and 
    computationally tractable procedures. This was in contrast 
    to the more traditional studies analyzing asymptotic 
    properties of such statistical methods. 
    
    This paper builds on recent advances made in robust 
    estimation and suggests a method that has attractive 
    properties both from asymptotic and finite-sample points 
    of view. Furthermore, it is computationally tractable 
    and its statistical complexity depends optimally on 
    the dimension. As a matter of fact, we even show that 
    what really matters is the intrinsic dimension, defined 
    in the Gaussian model as the effective rank of the 
    covariance matrix. 
    
    Note that in the framework of robust estimation, the 
    high-dimensional setting is qualitatively different 
    from the one dimensional setting. This qualitative
    difference can be shown at two levels. First, from a 
    computational point of view, the running time of several 
    robust methods scales poorly with dimension. Second, from 
    a statistical point of view, while a simple ``remove then 
    average'' strategy might be successful in low-dimensional 
    settings, it can easily be seen to fail in the high 
    dimensional case. Indeed, assume that for some $\varepsilon
    \in(0,1/2)$, $p\in\mathbb N$, and $n\in\mathbb N$, the 
    data $\bX_1,\ldots,\bX_n$ consist of $n(1-\varepsilon)$ points (inliers) drawn 
    from a $p$-dimensional Gaussian distribution $\mathcal N_p
    (0,\bfI_p)$ (where $\bfI_p$ is the $p\times p$ identity 
    matrix) and $\varepsilon n$ points (outliers) equal to a 
    given vector $\bu$. Consider an idealized setting in which, for a given threshold $r>0$, an oracle tells the user whether or not $\bX_i$ is within a distance $r$ of the true mean $0$. A simple strategy for robust mean 
    estimation consists of removing all the points of Euclidean 
    norm larger than $2\sqrt{p}$ and averaging all the remaining 
    points. If the norm of $\bu$ is equal to $\sqrt{p}$, one can 
    check that the distance between this estimator and the true 
    mean $\bmu = 0$ is of order 
    $\sqrt{p/n} + \varepsilon\|\bu\|_2 = \sqrt{p/n} + 
    \varepsilon\sqrt{p}$. This error rate is provably optimal
    in the small dimensional setting $p=O(1)$, but suboptimal
    as compared to the optimal rate $\sqrt{p/n} + \varepsilon$
    when the dimension $p$ is not constant. 
    The reason of this suboptimality is that the individually 
    harmless outliers, lying close to the bulk of the point cloud, have
    a strong joint impact on the quality of estimation. 
    
    We postpone a review of the relevant prior work 
    to~\Cref{sec:discuss} in order to ease comparison with
    our results, and proceed here with a summary of our 
    contributions. In the 
    context of a data set subject to a fully adversarial 
    corruption, we introduce a new estimator of the Gaussian 
    mean that enjoys the following properties (the precise 
    meaning of these properties is given in \Cref{sec:setting}):
    \vspace{-5pt}
	\begin{itemize}
	\item it is computable in polynomial time,
	\item it is equivariant with respect to similarity transformations 
	(translations, uniform scaling and orthogonal transformations),
	\item it has a high (minimax) breakdown point: $\varepsilon^* 
	= (5-\sqrt{5})/10\approx 0.28$,
	\item it is minimax-rate-optimal, up to a logarithmic factor,
	\item it is asymptotically efficient when the rate of 
	contamination tends to zero,
	\item for inhomogeneous covariance matrices, it achieves a 
	better sample complexity than all the other previously 
	studied methods.
	\end{itemize}
    
    In order to keep the presentation simple, all the aforementioned
    results are established in the case where the inliers are
    drawn from the Gaussian distribution. We then show that the
    extension to a sub-Gaussian distribution can be carried out
    along the same lines. Furthermore, we prove that using Lepski's
    method, one can get rid of the knowledge of the contamination
    rate. More precisely, we establish that the rate $\sqrt{p/n} + 
    \varepsilon\sqrt{\log(1/\varepsilon)}$ can be achieved without
    any information on $\varepsilon$ other than $\varepsilon < 
    (5-\sqrt{5})/10\approx 0.28$. Finally, we prove that the same
    order of magnitude of the estimation error is achieved when the 
    covariance matrix $\bSigma$ is unknown but isotropic (\textit{i.e.}, 
    proportional to the identity matrix). When the covariance matrix
    is an arbitrary unknown matrix with bounded operator norm, our 
    estimator has an error of order $\sqrt{p/n} + \sqrt{\varepsilon}$, 
    which is the best known rate of estimation by a computationally 
    tractable procedure in the case of unknown covariance matrices.

    The rest of this paper is organized as follows. We complete 
    this introduction by presenting the notation used throughout the 
    paper. \Cref{sec:setting} describes the problem setting and 
    provides the definitions of the properties of robust estimators 
    such as rate optimality or breakdown point. The iteratively 
    reweighted mean estimator is introduced in \Cref{sec:method}. 
    This section also contains the main facts characterizing the 
    iteratively reweighted mean estimator along with their 
    high-level proofs. A detailed discussion of relation to prior 
    work is included in \Cref{sec:discuss}. \Cref{sec:formal}
    is devoted to a formal statement of the main building blocks 
    of the proofs. Extensions to the cases of sub-Gaussian 
    distributions, unknown $\varepsilon$ and $\bSigma$ are 
    examined in \Cref{sec:extensions}. Some empirical results 
    illustrating our theoretical claims are reported in
    \Cref{sec:empiric}. Postponed proofs are gathered in  
    \Cref{sec:proofs} and in the appendix.
    
    For any vector $\bv$, we use the norm notations 
    $\|\bv\|_2$ for the standard Euclidean norm, $\|\bv\|_1$ for 
    the sum of absolute values of entries and $\|\bv\|_\infty$ 
    for the largest in absolute value entry of $\bv$. The tensor 
    product of $\bv$ by itself is denoted by $\bv^{\otimes 2} = 
    \bv\bv^\top$. We denote by 
    $\bDelta^{n-1}$ and by $\mathbb S^{n-1}$, respectively, the 
    probability simplex and the unit sphere in $\mathbb R^n$. 
    For any symmetric matrix $\bfM$, $\lambda_{\max}(\bfM)$ is 
    the largest eigenvalue of $\bfM$, while $\lambda_{\max,+}(\bfM)$ 
    is its positive part. The operator norm of $\bfM$ is denoted 
    by $\|\bfM\|_{\textup{op}}$. We will often use the effective 
    rank $\mathbf r_{\bfM}$ defined as $\tr(\bfM)/\|\bfM\|_{
    \textup{op}}$, where $\tr(\bfM)$ is the trace of matrix $\bfM$. 
    {For symmetric matrices $\bfA$ and $\bfB$ of the same size 
    we write $\bfA \succeq \bfB$, if the matrix $\bfA - \bfB$ 
    is positive semidefinite.} 
    For a rectangular $p\times n$ matrix $\bfA$, we let 
    $s_{\min}(\bfA)$ and $s_{\max}(\bfA)$ be the smallest and the 
    largest singular values of $\bfA$ defined respectively as 
    $s_{\min}(\bfA) = \inf_{\bv\in\mathbb S^{n-1}} \|\bfA\bv\|_2$ 
    and $s_{\max}(\bfA) = \sup_{\bv\in\mathbb S^{n-1}} \|\bfA\bv\|_2$. 
    The set of all $p\times p$ positive semidefinite matrices is 
    denoted by $\mathcal S_{+}^p$. 
    
	\section{Desirable properties of a robust estimator}\label{sec:setting}

	We consider the setting in which the sample points are corrupted versions
	of independent and identically distributed random vectors drawn from
	a $p$-variate Gaussian distribution with mean $\bmu^*$ and covariance 
	matrix $\bSigma$. In what follows, we will assume that the rate of 
	contamination and the covariance matrix are known and, therefore, can
	be used for constructing an estimator of $\bmu^*$. We present in 
	\Cref{sec:extensions} some additional results which are valid under 
	relaxations of this assumption.

    \def\bfPn{\mathbf P\!_n}
	\begin{definition}\label{def:1}
	We say that the distribution $\bfPn$ of data $\bX_1,\ldots,\bX_n$ 
	is Gaussian with adversarial contamination, denoted by 
	$\bfPn\in \textup{GAC}(\bmu^*,\bSigma,\varepsilon)$ with 
	$\varepsilon\in(0,1/2)$ and $\bSigma\succeq 0$,
	if there is a set of $n$ independent and identically distributed random 
	vectors $\bY_1,\ldots,\bY_n$ drawn from $\mathcal N_p(\bmu^*,\bSigma)$ 
	satisfying
	\begin{align}
	\big|\{i:\bX_i\neq \bY_i \}\big|\le \varepsilon n.
	\end{align}
	\end{definition}

	In what follows, the sample points $\bX_i$ with indices in the set 
	$\mathcal O = \{i:\bX_i\neq \bY_i\}$ are called outliers, while all the
	other sample points are called inliers. We define $\mathcal I = \{1,
	\ldots,n\}\setminus\mathcal O$, the set of inliers. Assumption GAC allows 
	both the
	set of outliers $\mathcal O$ and the outliers themselves to be random and 
	to depend arbitrarily on the values of $\bY_1,\ldots,\bY_n$. In 
	particular, we can consider a game in which an adversary has access to
	the clean observations $\bY_1,\ldots,\bY_n$ and is allowed to modify 
	an $\varepsilon$ fraction of them before unveiling to the Statistician.  
	The Statistician aims at estimating $\bmu^*$ as accurately as 
	possible, the accuracy being measured by the expected estimation error:
	\begin{align}
		R_{\bfPn}(\hat\bmu_n, \bmu^*) = \|\hat\bmu_n - 
		\bmu^*\|_{\mathbb L_2({\bfPn})} = \bigg(\sum_{j=1}^p \bfE_{\bfPn}
		[(\hat{\bmu}_n-\bmu^*)_j^2]\bigg)^{1/2}.
	\end{align}
	Thus, the goal of the adversary is to apply a contamination that makes 
	the task of estimation the hardest possible. The goal of the Statistician
	is to find an estimator $\hat{\bmu}_n$ that minimizes the worst-case risk
	\begin{align}
		R_{\max}(\hat\bmu_n,\bSigma,\varepsilon) = \sup_{\bmu^*\in\RR^p}
		\sup_{\bfPn\in\textup{GAC}(\bmu^*,\bSigma,\varepsilon)} 
		R_{\bfPn}(\hat\bmu_n, \bmu^*).
	\end{align}
	Let $\rSigma = \tr(\bSigma)/\|\bSigma\|_{\textup{op}}$ be the effective
	rank of $\bSigma$. The theory developed by \citep{chen2016,chen2018}, 
	in conjunction with \citep[Prop.~1]{bateni2020minimax}, 
	implies that 
	\begin{align}\label{lower}
		\inf_{\hat{\bmu}_n} R_{\max}(\hat\bmu_n,\bSigma,\varepsilon) \ge c
		\|\bSigma\|^{1/2}_{\textup{op}} \bigg(\sqrt{\frac{\rSigma}{n}} + 
		\varepsilon\bigg)	
	\end{align}
	for some constant $c>0$, where the infimum is over all measurable functions
	of $(\bX_1,\ldots,\bX_n)$. 
	A detailed proof of this claim is presented in \Cref{App:D}. This lower bound naturally leads to the following definition.

	\begin{definition}\label{def:2}
	We say that the estimator $\hat{\bmu}_n$ is minimax rate optimal (in 
	expectation), if there are universal constants $c_1$, $c_2$ and $C$ 
	such that
	\begin{align}
		R_{\max}(\hat\bmu_n,\bSigma,\varepsilon) \le C\|\bSigma\|^{1/2}_{\textup{op}} 
		\bigg(\sqrt{\frac{\rSigma}{n}} + \varepsilon\bigg)	
	\end{align}
	for every $(n,\bSigma,\varepsilon)$ satisfying $\rSigma\le c_1 n$ and 
	$\varepsilon\le c_2$. 
	\end{definition}

	The iteratively reweighted mean estimator, introduced in the 
	next section, is not minimax rate optimal but is very close 
	to being so. Indeed, we will prove that it is minimax rate 
	optimal up to a  $\sqrt{\log(1/\varepsilon)}$ factor in the second term ($\varepsilon$ is replaced by $\varepsilon\sqrt{\log(1/\varepsilon)}$). It should 
	be stressed here that, to the best of our knowledge, none of the 
	results on robust estimation of the Gaussian mean provides 
	rate-optimality in expectation in the high-dimensional setting. 
	Indeed, all those results provide risk bounds that hold with 
	high probability, and either (a) do not say anything about the 
	magnitude of the error on a set of small but strictly positive 
	probability or (b) use the confidence parameter in the 
	construction of the estimator. Both of these shortcomings prevent 
	from extracting bounds for expected loss from high-probability 
	bounds. This being said, it should be noted that most prior work 
	has focused on the Huber contamination, in which case no meaningful 
	(\textit{i.e.}, different from $+\infty$, see Section~2.6 
	in \citep{bateni2020minimax}) upper bound on the minimax risk 
	in expectation can be obtained. 
	
	\begin{definition}\label{def:3}
	We say that $\hat{\bmu}_n$ is an asymptotically efficient 
	estimator of $\bmu^*$, if when $\varepsilon = \varepsilon_n$ 
	tends to zero sufficiently fast, as $n$ tends to infinity, 
	we have
	\begin{align}
		R_{\max}(\hat\bmu_n,\bSigma,\varepsilon) \le 
		\|\bSigma\|^{1/2}_{\textup{op}} 
		\sqrt{\frac{\rSigma}{n}}\,\big( 1 + o_n(1)\big).	
	\end{align}	
	\end{definition}
	
	If we compare inequalities in \Cref{def:2} and \Cref{def:3}, we can see that the constant $C$ present in the former disappeared in the latter. This reflects the fact that asymptotic efficiency implies not only rate-optimality, but also the optimality of the constant factor. We recall here that in the outlier-free situation, the sample mean is asymptotically efficient and its worst-case risk is equal to  $\|\bSigma\|^{1/2}_{\textup{op}} 
	\sqrt{{\rSigma}/{n}}$.
	
	One can infer from \eqref{lower} that a necessary condition 
	for the existence of asymptotically efficient estimator is 
	$\varepsilon_n^2 = o_n(\rSigma/n)$. We show in the next 
	section that this condition is almost sufficient, by proving 
	that the iteratively reweighted mean estimator is asymptotically
	efficient provided that	$\varepsilon_n^2\log(1/\varepsilon_n) 
	= o_n(\rSigma/n)$.
	
	The last notion that we introduce in this section is the 
	breakdown point, the term being coined by \cite{Hampel}, 
	see also \citep{DonHuber}. Roughly speaking, the breakdown 
	point of a given estimator is the largest proportion of 
	outliers that the estimator can support without becoming 
	infinitely large. The definition we provide below slightly differs from the original one. This difference is motivated by our goal to focus on 
	studying the expected risk of robust estimators.

	\begin{definition}\label{def:4}
		We say that $\varepsilon_n^*\in[0,1/2]$ is the (finite-sample) breakdown 
		point of the estimator $\hat\bmu_n$, if 	
		\begin{align}
			R_{\max}(\hat\bmu_n,\bSigma,\varepsilon) < +\infty,\qquad \forall
			\varepsilon < \varepsilon^*_n			
		\end{align}	
		and $R_{\max}(\hat\bmu_n,\bSigma,\varepsilon) = +\infty$, for every  
		$\varepsilon>\varepsilon^*_n$.
	\end{definition}
	
	One can check that the breakdown points of the componentwise median 
	and the geometric median (see the definition of $\GMmu$ in \eqref{GM:def} 
	below) equal $1/2$. 
	Unfortunately, the minimax rate of these methods is strongly suboptimal,
	see \citep[Prop.\ 2.1]{chen2018} and \citep[Prop.\ 2.1]{LaiRV16}. 
	Among all rate-optimal (up to a polylogarithmic factor) robust estimators, 
	Tukey's median is\footnote{Recent results in \citep{zhu2020tukey} suggest that an estimator based on the TV-projection may achieve the optimal breakdown point of $1/2$.} the one with highest known breakdown point equal to $1/3$ \citep{DonGasko}. It should be
	noted that this paper deal with the original 
	definition of the breakdown point which, as already 
	mentioned, is slightly different from that of \Cref{def:4}. 
	
	The notion of breakdown point given in \Cref{def:4}, well adapted to estimators that do not
	rely on the knowledge of $\varepsilon$, becomes less relevant in the
	context of known $\varepsilon$. Indeed, if a given estimator
	$\hat\bmu_n(\varepsilon)$ is proved to have a breakdown point equal to 
	$0.1$, one can consider instead the estimator $\tilde\bmu_n(\varepsilon) 
	= \hat\bmu_n(\varepsilon)\mathds 1(\varepsilon< 0.1) + \GMmu \mathds 1(
	\varepsilon\ge 0.1)$, which	will have a breakdown point equal to $0.5$. 
	For this reason, it appears more appealing to consider a different notion
	that we call rate-breakdown point, and which is of the same flavor
	as the $\delta$-breakdown point defined in \citep{chen2016}. 
	
	\begin{definition}
    	We say that $\varepsilon_{r}^*\in[0,1/2]$ is the  
    	$r(n,\bSigma,\varepsilon)$-breakdown point of the estimator 
    	$\hat\bmu_n$ for a given function $r:\mathbb N\times\mathcal S_{+}^p
    	\times [0,1/2)$, if for every $\varepsilon<\varepsilon_{r}^*$, 
		\begin{align}
			\sup_{n,p} 
			\frac{R_{\max}(\hat\bmu_n(\varepsilon),\bSigma,\varepsilon)}{r(n,
			\bSigma,\varepsilon)} < +\infty,
	\end{align}	
	and $\varepsilon^*_r$ is the largest value satisfying this property.
	\end{definition}
	
	In the context of Gaussian mean estimation, if the previous definition is 
	applied with $r(n,\bSigma,\varepsilon) = \|\bSigma\|_{\textup{op}} 
	\big(\sqrt{\rSigma/n} + \varepsilon)$, we call the corresponding value 
	the minimax-rate-breakdown point. Similarly, if  $r(n,\bSigma,\varepsilon) =
	\|\bSigma\|_{\textup{op}}\big(\sqrt{\rSigma/n} + \varepsilon
	\sqrt{\log(1/\varepsilon)})$, we call the corresponding value the
	nearly-minimax-rate-breakdown point. It should be mentioned here that
	the extra $\sqrt{\log(1/\varepsilon)}$ factor cannot be avoided by any
	Statistical Query (SQ) polynomial-time algorithm, as shown by the 
	SQ lower bound established in \citep{DiakonikolasKS17}.  
	
	\section{Iterative reweighting approach}\label{sec:method}

	In this section, we define the iterative reweighting estimator that will
	be later proved to enjoy all the desirable properties. To this end, we set
	\begin{align}\label{notations1}
			\bar{\bs X}_{\bw} = \sum_{i=1}^n w_i \bs X_i,\qquad
			G(\bw,\bmu) = \lambda_{\max,+}\bigg(\sum_{i=1}^n  w_i(\bs X_i-\bmu)
			^{\otimes 2} - \bSigma\bigg)
	\end{align}
	for any pair of vectors  $\bw\in[0,1]^n$ and $\bmu\in\RR^p$. The main idea 
	of the proposed methods is to find a weight vector $\hat{\bw}_n$ belonging 
	to the probability simplex
	\begin{align}
		\bDelta^{n-1} = \Big\{\bw\in [0,1]^n : w_1+\ldots +w_n = 1\Big\}
	\end{align}
	that mimics the ideal weight vector $\bw^*$ defined by 
	$w^*_j = \mathds 1(j\in\mathcal I)/|\mathcal I|$, so that 
	the weighted average $\bar{\bX}_{\hat{\bw}_n}$ is nearly as 
	close to $\bmu^*$ as the average of the inliers. Note that, for any 
	weight vector $\bw\in \bDelta^{n-1}$ and any vector $\bmu\in \mathbb{R}^p$,
	we have
	\begin{align}
	    \sum_{i=1}^n  w_i(\bs X_i-\bmu)^{\otimes 2} = \sum_{i=1}^n  w_i(\bs X_i-
	    \bar\bX_{\bw})^{\otimes 2} + (\bar\bX_{\bw} - \bmu)^{\otimes 2}
	    \succeq \sum_{i=1}^n  w_i(\bs X_i-\bar\bX_{\bw})^{\otimes 2}.
	\end{align}
    This readily yields that $G(\bw,\bmu)\ge G(\bw,\bar\bX_{\bw})$ 
    , and, therefore 
    \begin{align}\label{eq:minG}
        G(\bw,\bar\bX_{\bw}) = \min_{\bmu\in\RR^n} G(\bw,\bmu),\qquad \forall
        \bw\in\bDelta^{n-1}.
    \end{align}%
\begin{figure*}
	\begin{center}
	\begin{minipage}{0.8\linewidth}
			\begin{alg}{Iteratively reweighted mean estimator (known 
			$\bs\varepsilon$ and $\bSigma$)}{a1}		
				\begin{algorithmic}
				\small
				\STATE {\bfseries Input:} data $\bX_1,\ldots,\bX_n\in\RR^p$,  
				contamination rate $\varepsilon$ and $\bSigma$
				\STATE {\bfseries Output:} parameter estimate $\hat\bmu_n^{\textup{IR}}$
				\STATE {\bfseries Initialize:} compute $\hat\bmu^0$ as a minimizer of 
				$\sum_{i=1}^n \|\bX_i-\bmu\|_2$\\[7pt]
				\STATE  Set $K = 0\vee\Big\lceil \frac{\log(4\rSigma) - 2\log(\varepsilon({1-2\varepsilon}))}
				{2\log(1-2\varepsilon) - \log \varepsilon -\log(1-\varepsilon)}
				\Big\rceil$.\\[7pt]
				\STATE {{\bf For} $k=1:K$}
				\STATE \qquad {Compute current weights:}\\
				\qquad\quad $\ds	\bw \in \mathop{\textup{arg\,min}}_{(n-n\varepsilon)\|\bw\|_\infty\le 1}
					\lambda_{\max}\bigg(\sum_{i=1}^n  w_i (\bs X_i-\hat\bmu^{k-1})
					^{\otimes 2}- \bSigma\bigg)\vee 0.		
				$
				\STATE \qquad{Update the estimator: $\hat{\bmu}^k = \sum_{i=1}^n w_i\bX_i$}.
				\STATE {\bfseries EndFor}
				\STATE {{\bf Return} $\hat\bmu^{K}$}.
				\end{algorithmic}		
			\end{alg}
		\end{minipage}
	\end{center}
\end{figure*}%
	The precise definition of the proposed estimator is as follows. 
	We start from an initial estimator $\hat\bmu^0$ of $\bmu^*$. To 
	give a concrete example, and also in order to guarantee 
	equivariance by similarity transformations, 
	we assume that $\hat\bmu^0$ is the geometric median:
	\begin{align}\label{GM:def}
	 \hat\bmu^0 = 
	\GMmu \in\textup{arg}\min_{\bmu\in\RR^p} \sum_{i=1}^n\|\bX_i-\bmu\|_2.
	\end{align}   

	\begin{definition}\label{IR:def}
	We call iteratively reweighted mean estimator, denoted by 
	$\hat{\bmu}_n^{\textup{IR}}$, the $K$-th element of the sequence 
	$\{\hat\bmu^k; k=0,1,\ldots\}$ starting from $\hat\bmu^0$ in \eqref{GM:def}
	and defined by the recursion
	\begin{align}\label{iter:1}
		\hat\bw^k \in \textup{arg}\min_{(n-n\varepsilon)\|\bw\|_\infty\le 1}
		G(\bw,\hat\bmu^{k-1}),\qquad  \hat{\bmu}^k = \bar{\bX}_{\hat{\bw}^{k}},
	\end{align}
	where the minimum is over all weight vectors $\bw\in\bDelta^{n-1}$
	satisfying $\max_j w_j\le 1/(n-n\varepsilon)$ and the number of iteration 
	is
	\begin{align}\label{K}
		K = 0\bigvee  \bigg\lceil \frac{\log(4\rSigma) - 
		2\log(\varepsilon({1-2\varepsilon}))}{2\log(1-2\varepsilon) - 
		\log \varepsilon -\log(1-\varepsilon)}\bigg\rceil. 
	\end{align}
	\end{definition}

    The idea of computing a weighted mean, with weights measuring the 
    outlyingness of the observations goes back at least to 
    \citep{Donoho1,Stahel}. Perhaps the first idea similar to that of 
    minimizing the largest eigenvalue of the covariance matrix 
    was that of minimizing the determinant of the sample
    covariance matrix over all subsamples of a given cardinality 
    \citep{Rousseeuw84,Rousseeuw85}. It was also observed in \citep{Lopuha} 
    that one can improve the estimator by iteratively updating the weights. 
    An overview of these results can be found in \citep{Rousseeuw}. 
    
    \begin{figure}
        \centering
        \includegraphics[origin=c, width=0.6\textwidth]{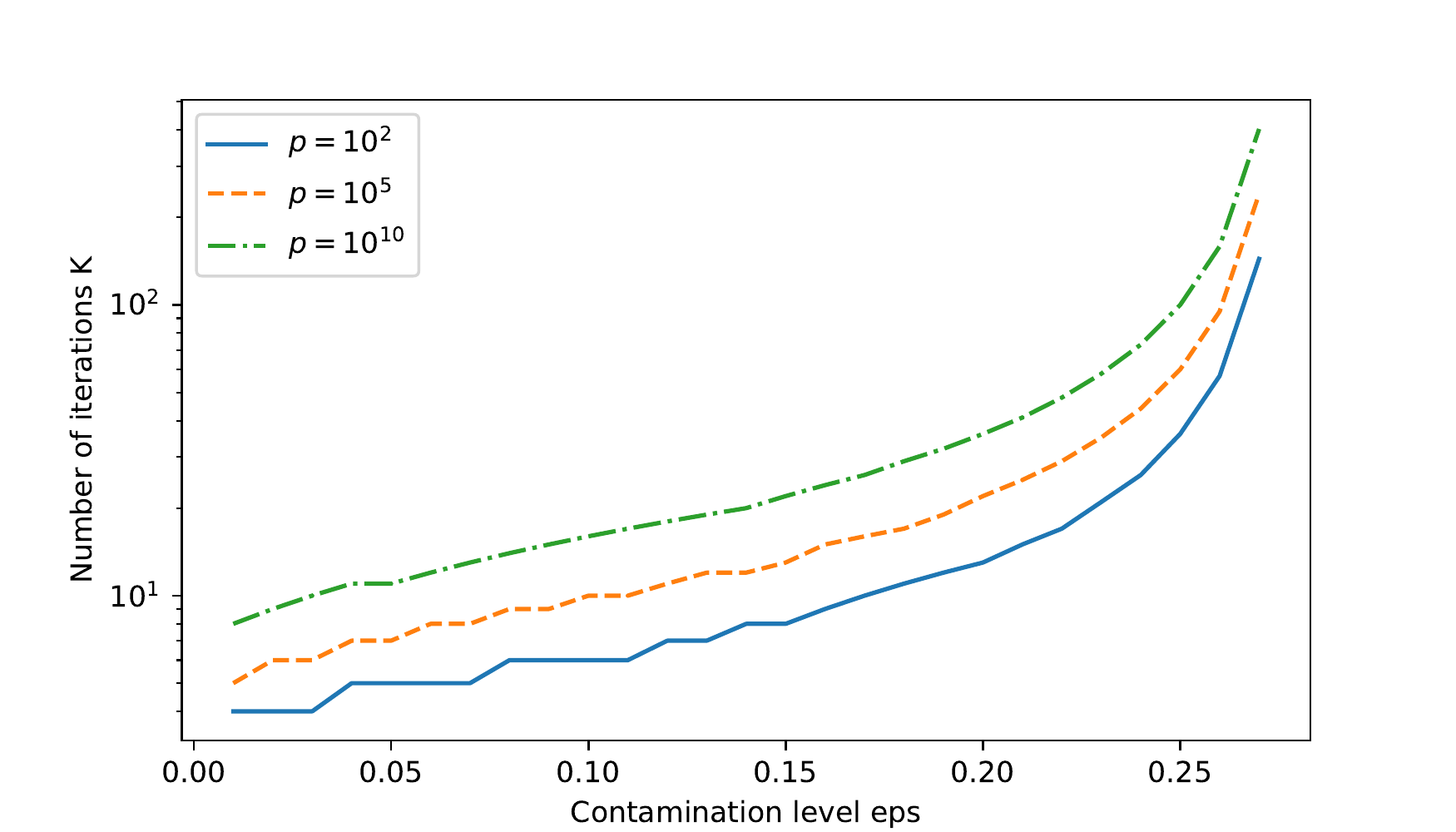}
    
    \caption{The behavior of the number of iterations $K = K_\varepsilon$, given by \eqref{K}, as a function of the contamination rate $\varepsilon$ for different values of the dimension $p = \rSigma$.}
    \label{fig:k-eps}
    \end{figure}

	Note that the value of $K$ provided above is tailored to 
	the case where the initial estimator is the geometric 
	median. Clearly, $K$ depends only logarithmically on the 
	dimension and $K = K_\varepsilon$ tends to 2 when $\varepsilon$ 
	goes to zero, see \Cref{fig:k-eps}. Note also that the choice of $K$ in \eqref{K} is 
	derived from the condition $\big(\sqrt{\varepsilon(1-\varepsilon)
	}/(1-2\varepsilon	)\big)^K\|\hat\bmu^0-\bmu^*\|_{\mathbf L_2}
	\le \varepsilon$, see \eqref{ineq:GMR} below. We can use any other 
	initial estimator of $\bmu^*$ instead of the geometric median, 
	provided that $K$ is large enough to satisfy the last inequality.
	
	We have to emphasize that $\hat{\bmu}_n^{
	\textup{IR}}$ relies on the knowledge of both $\varepsilon$ and 
	$\bSigma$ (the dependence on $\bSigma$ is through the effective
	rank, which coincides with the dimension for non degenerate covariance 
	matrices). Indeed, the number of iterations $K$ depends on both $\varepsilon$ and $\bSigma$. Additionally, $\bSigma$ is used in 
	the cost function $G(\bw, \bmu)$ and $\varepsilon$ is used for 
	specifying the set of feasible weights in optimization problem \eqref{iter:1}. We present some extensions to the case of unknown
	$\varepsilon$ and $\bSigma$ in \Cref{sec:extensions}. 
	
	The rest of this section is devoted to showing that the 
	iteratively reweighted estimator enjoys all the desirable 
	properties announced in the introduction. An estimator is called
	computationally tractable, if its computational complexity 
	is at most polynomial in $n$, $p$ and $1/\varepsilon$. 

	\begin{center}
	\begin{minipage}{0.90\linewidth}
			\begin{fact}{}{f1}
			The estimator $\hat{\bmu}_n^{\textup{IR}}$ is computationally tractable. 
			\end{fact}
		\end{minipage}
	\end{center}

	In order to check computational tractability, it suffices to prove that
	each iteration of the algorithm can be performed in polynomial time. Since
	the number of iterations depends logarithmically on $\rSigma\le p$, 
	this will suffice. 
	Note now that the optimization problem in \eqref{iter:1} is convex and 
	can be cast into a semi-definite program. Indeed, it is equivalent to 
	minimizing a real value $t$ over all the pairs $(t,\bw)$ satisfying the 
	constraints 
	\begin{align}
	t\ge 0,\quad \bw\in\bDelta^{n-1},\quad 
	\|\bw\|_\infty\le \frac1{n(1-\varepsilon)},\quad
	\sum_{i=1}^n  w_i(\bs X_i-\hat{\bmu}^{k-1})^{\otimes 2} 
	\preceq \bSigma + t\bfI_p.
	\end{align} 
	The first two constraints can be rewritten as a set of linear inequalities, 
	while the third constraint is a linear matrix inequality. Given the special
	form of the cost function and the constraints, it is possible to design
	specific optimization routines which will find an approximate solution to
	the problem in a faster way than the out-of-shelf SDP-solvers. However, 
	we will not pursue this line of research in this work.

	\begin{center}
		\begin{minipage}{0.90\linewidth}
			\begin{fact}{}{f2}
				The estimator $\hat{\bmu}_n^{\textup{IR}}$ is 
				translation, uniform scaling and orthogonal 
				transformation equivariant. 
			\end{fact}	
		\end{minipage}
	\end{center}

	The equivariance mentioned in this statement should be 
	understood as follows. If we denote by $\hat{\bmu}_{n,X}^{
	\textup{IR}}$ the estimator computed for data $\bX_1,\ldots,
	\bX_n$, and by $\hat{\bmu}_{n,X'}^{\textup{IR}}$ the one 
	computed for data $\bX'_1,\ldots,\bX'_n$, with $\bX_i' = 
	\bs a + \lambda\bfU \bX_i$, where $\bs a\in\RR^p$, 
	$\lambda>0$ and $\bfU$ is a $p\times p$ orthogonal matrix, 
	then $\hat{\bmu}_{n,X'}^{\textup{IR}} = \bs a + \lambda\bfU
	\hat{\bmu}_{n,X}^{\textup{IR}}$. 
	To prove this property, we first note that 
	\begin{align}
	 \min_{\bmu\in\RR^p} \sum_{i=1}^n\|\bX'_i-\bmu\|_2 = 
	 \lambda\min_{\bmu\in\RR^p}\sum_{i=1}^n\|\bX_i-
	 \lambda^{-1}\bfU^\top(\bmu-\bs a)\|_2.
	\end{align}   
	This implies that $\hat{\bmu}_{n,X}^{\textup{GM}} = 
	\lambda^{-1}\bfU^\top(\hat{\bmu}_{n,X'}^{\textup{GM}} -\bs a)$, 
	which is equivalent to $\hat{\bmu}_{n,X'}^{\textup{GM}} 
	= \bs a + \lambda\bfU\hat{\bmu}_{n,X}^{\textup{GM}}$. 
	Therefore, the initial value of the recursion is equivariant. 
	If we add to this the fact that\footnote{We use here the 
	notation $G_X(\bw,\bmu)$ to make clear the dependence of 
	$G$ in \eqref{notations1} on $\bX_i$s. We also stress that 
	when the estimator is computed for the transformed data 
	$\bX_i'$, the matrix $\bSigma$ is naturally replaced by
	$\lambda^2\bfU\bSigma\bfU^\top$.} $G_X(\bw,\bmu) =\lambda^2 
	G_{X'}(\bw,\bs a + \lambda\bfU\bmu)$ for every $(\bw,\bmu)$, 
	we get the equivariance of $\hat{\bmu}_n^{\textup{IR}}$.

	\begin{center}
		\begin{minipage}{0.90\linewidth}
			\begin{fact}{}{f3}
					The breakdown point $\varepsilon^*_n$ and the nearly-minimax-rate-breakdown 
					point $\varepsilon^*_r$ of $\hat{\bmu}_n^{\textup{IR}}$ 
					satisfy, respectively, $\varepsilon^*_n=0.5$ and $\varepsilon^*_r\ge 
					(5-\sqrt{5})/10\approx 0.28$.
			\end{fact}
		\end{minipage}
	\end{center}

	We prove later in this paper (see \eqref{eq:fund3}) that  if $\bX_1,\ldots,\bX_n$ satisfy 
	{GAC}$(\bmu^*,\bSigma,\varepsilon)$, there is a random variable $\Xi$
	depending only on $\bzeta_i = \bY_i-\bmu^*$, $i=1,\ldots,n$, such that
	\begin{align}\label{main:ineq}
			\|\bar{\bs X}_{\bw} - {\bmu}^*\|_2 
			&\le \frac{\sqrt{\varepsilon(1-\varepsilon)}}{1-2\varepsilon}\,
				G(\bw,\bmu)^{1/2} + \Xi,\quad\forall \bmu\in\RR^p,
	\end{align}
	for every $\bw\in\bDelta^{n-1}$ such that $n(1-\varepsilon)\|\bw\|_\infty 
	\le 1$. Inequality \eqref{main:ineq} is one of the main building blocks of 
	the proof of \Cref{fact:f3,fact:f4,fact:f5}. This inequality, as well as 
	inequalities \eqref{ineq:3} and \eqref{ineq:4} below will be formally stated 
	and proved in subsequent sections. To check \Cref{fact:f3}, we set 
	$\alpha_\varepsilon = {\sqrt{\varepsilon(1-\varepsilon)}}/{(1-2\varepsilon)}$ 
	and note that\footnote{See \Cref{ssec:6.1} for more detailed explanations.}
	\begin{align}
			\|\hat{\bmu}^k - {\bmu}^*\|_2 = \|\bar{\bX}_{\hat{\bw}^k} - {\bmu}^*\|_2
			&\le \alpha_{\varepsilon}\,
				G(\hat{\bw}^k,\hat\bmu^{k-1})^{1/2} + \Xi\\
			&\le \alpha_{\varepsilon}\,
				G(\bw^*,\hat\bmu^{k-1})^{1/2} + \Xi\\
			&\le \alpha_{\varepsilon}\,
				\big(G(\bw^*,\bar\bX_{\bw^*}) + \|\bar\bX_{\bw^*}-\hat\bmu^{k-1}\|_2^2
				\big)^{1/2} + \Xi\\
			&\le \alpha_{\varepsilon}\,\big(
				G(\bw^*,\bmu^*)  + \|\bar\bX_{\bw^*}-\hat\bmu^{k-1}\|_2^2\big)^{1/2} + \Xi\\
			&\le \alpha_{\varepsilon}\,\|\hat\bmu^{k-1}-\bmu^*\|_2 +
			\tilde\Xi,\label{eq:3}
	\end{align}
	where $\tilde\Xi = \alpha_{\varepsilon}\,\big(G(\bw^*,\bmu^*)^{1/2}  + 
	\|\bar\bzeta_{\bw^*}\|_2\big) + \Xi$. Unfolding this recursion, we get
	\footnote{Here and in the sequel $\alpha_{\varepsilon}^K$ stands for 
	$K$-th power of $\alpha_{\varepsilon}$.}
	\begin{align}\label{main:ineq2}
			\|\hat{\bmu}_n^{\textup{IR}} - {\bmu}^*\|_2  = \|\hat{\bmu}^K - {\bmu}^*\|_2 
			&\le \alpha_{\varepsilon}^K\,\|\hat\bmu^{0} - \bmu^*\|_2 +
			\frac{\tilde\Xi}{1-\alpha_\varepsilon}.
	\end{align}
	The geometric median $\hat\bmu^{0}  = \GMmu$ having a breakdown point equal to $1/2$,
	we infer from the last display that the error of the iteratively reweighted estimator
	remains bounded after altering $\varepsilon$-fraction of data points provided that
	$\alpha_\varepsilon <1 $. This implies that the breakdown point is at 
	least equal to the solution of the equation $\sqrt{\varepsilon(1-\varepsilon)}
	=1-2\varepsilon$, which yields $\varepsilon^*\ge (5-\sqrt{5})/10$. Moreover, if $\varepsilon\in[ (5-\sqrt{5})/10, 1/2]$, then the number
	of iterations $K$ equals zero and the iteratively reweighted mean coincides with
	the geometric median. Therefore, its 
	breakdown point is $1/2$.

	\begin{center}
		\begin{minipage}{0.90\linewidth}
			\begin{fact}{}{f4}			
				The estimator $\hat{\bmu}_n^{\textup{IR}}$ is 
				nearly minimax rate optimal, in the sense that its 
				worst-case risk is bounded by $C\|\bSigma\|^{1/2}
				_{\textup{op}}\big(\sqrt{{\rSigma}/{n}} + \varepsilon\sqrt{\log(1/\varepsilon)}\big)$, where
				$C$ is a universal constant.
			\end{fact}	
		\end{minipage}
	\end{center}

	Without loss of generality, we assume that $\|\bSigma\|_{\textup{op}}=1$
	so that $\rSigma = \tr(\bSigma)$. We can always reduce the initial problem to this case
	by considering scaled data points $\bX_i/\|\bSigma\|_{\textup{op}}^{1/2}$ 
	instead of $\bX_i$. Combining \eqref{main:ineq2} and the triangle inequality, 
	we get
	\begin{align}\label{ineq:GMR}
			\|\hat{\bmu}_n^{\textup{IR}} - {\bmu}^*\|_{\mathbb L_2}  		
			&\le \alpha_{\varepsilon}^K\,\|\GMmu - \bmu^*\|_{\mathbb L_2} +
			\frac{\|\tilde\Xi\|_{\mathbb L_2}}{1-\alpha_\varepsilon}.
	\end{align}
	It is not hard to check that 
	$\|\GMmu - \bmu^*\|_{\mathbb L_2}\le 2\sqrt{\rSigma}/(1-2\varepsilon)$, 
	see \Cref{lem:8} below. Furthermore, the choice of $K$ in \eqref{K} entails $2\alpha_\varepsilon^K\sqrt{\rSigma} \le \varepsilon(1-2\varepsilon)$. 
	This implies that
	\begin{align}
			\|\hat{\bmu}_n^{\textup{IR}} - {\bmu}^*\|_{\mathbb L_2}  		
			&\le  \varepsilon + \frac{\|\tilde\Xi\|_{\mathbb L_2}}{1 - 
			\alpha_\varepsilon}.
	\end{align}
	The last two building blocks of the proof are the following\footnote{Inequality \eqref{ineq:3} is 
	\citep[Th. 4]{KoltchLounici}, while \eqref{ineq:4} is the claim of \Cref{proposition2} below.} inequalities: 
	\begin{align}
		&\bfE[G(\bw^*,\bmu^*)] \le C\big(1 +
		\sqrt{\rSigma/n}\big)\sqrt{\rSigma/n},\label{ineq:3}\\
		&\|\Xi\|_{\mathbb L_2} \le \sqrt{\rSigma/n}(1+ C\sqrt{\varepsilon}) +
		C\sqrt{\varepsilon}(\rSigma/n)^{1/4} + 
		C\varepsilon\sqrt{\log(1/\varepsilon)},\label{ineq:4}
	\end{align}
	where $C>0$ is a universal constant. In what follows, the value of $C$
	may change from one line to the other. We have
	\begin{align}
		\|\tilde\Xi\|_{\mathbb L_2} &\le \alpha_{\varepsilon}\big(\|G(\bw^*,\bmu^*)^{1/2}
		\|_{\mathbb L_2}  + \|\bar\bzeta_{\bw^*}\|_{\mathbb L_2}\big) + \|\Xi\|_{\mathbb L_2}\\
		&\le  C\sqrt{\varepsilon}\,\big(\bfE^{1/2}[G(\bw^*,\bmu^*)]  + 
		\sqrt{\rSigma/n}\big) + \|\Xi\|_{\mathbb L_2}^2\\
		&\le  C\varepsilon\,\big((\rSigma/n)^{1/4} + \sqrt{\rSigma/n} \big) + 
		\|\Xi\|_{\mathbb L_2}\\
		&\le \sqrt{\rSigma/n}(1+ C\sqrt{\varepsilon}) +
		C\sqrt{\varepsilon}(\rSigma/n)^{1/4} + C\varepsilon\sqrt{\log(1/\varepsilon)}\\
		&\le C\sqrt{\rSigma/n} + C\varepsilon\sqrt{\log(1/\varepsilon)}.\label{ineq:5}
	\end{align}
	Returning to \eqref{main:ineq2} and combining it with \eqref{ineq:5}, we 
	get the claim of \Cref{fact:f4} for every $\varepsilon\le \varepsilon_0$, 
	where $\varepsilon_0$ is any positive number strictly smaller than 
	$(5-\sqrt{5})/10$. This also proves the second claim of \Cref{fact:f3}.

	\begin{center}
	\begin{minipage}{0.90\linewidth}
	\begin{fact}{}{f5}
		In the setting  $\varepsilon  = \varepsilon_n\to 0$ so  that 
		$\varepsilon^2\log(1/\varepsilon) = o_n(\rSigma/n)$ when $n\to\infty$, 
		the estimator $\hat{\bmu}_n^{\textup{IR}}$ is asymptotically efficient.
	\end{fact}
	\end{minipage}
	\end{center}

	The proof of this fact follows from \eqref{main:ineq2} and \eqref{ineq:4}. 
	Indeed, if $\varepsilon^2\log(1/\varepsilon) = o_n(\rSigma/n)$, \eqref{ineq:4} 
	implies that
	\begin{align}
	\|\tilde\Xi\|_{\mathbb L_2}^2 &\le \frac{\rSigma}{n}\big(1+ o_n(1)\big).
	\end{align}
	Injecting this bound in \eqref{main:ineq2} and using the fact that $\varepsilon$
	tends to zero, we get the claim of \Cref{fact:f5}.

	\section{Relation to prior work and discussion}\label{sec:discuss}
	
	Robust estimation of a mean is a statistical problem
	studied by many authors since at least sixty years. 
	It is impossible to give an overview of all existing 
	results and we will not try to do it here. The interested
	reader may refer to the books \citep{maronna2006robust} 
	and \citep{Huber2009}. We will rather focus here on some 
	recent results that are the most closely related to the present
	work. Let us just recall that \cite{Huber2009} enumerates
	three desirable properties of a statistical procedure: 
	efficiency, stability and breakdown. We showed here that
	iteratively reweighted mean estimator possesses these features and, 
	in addition, is equivariant and computationally tractable.
	
	To the best of our knowledge, the form $\sqrt{p/n} + 
	\varepsilon$ of the minimax risk in the Gaussian mean 
	estimation problem has been first obtained by \cite{chen2018}. 
	They proved that this rate holds with high probability 
	for the Tukey median, which is known to be computationally 
	intractable in the high-dimensional setting. 
	The first nearly-rate-optimal and computationally tractable
	estimators have been proposed by \cite{LaiRV16} and
	\cite{DiakonikolasKK016}\footnote{See \citep{DiakonikolasKKL19} 
	for the extended version}. The methods analyzed in these
	papers are different, but they share the same idea: If for 
	a subsample of points the empirical covariance matrix is
	sufficiently close to the theoretical one, then the 
	arithmetic mean of this subsample is a good estimator
	of the theoretical mean. Our method is based on this 
	idea as well, which is mathematically formalized in
	\eqref{main:ineq}, see also \Cref{proposition1} below. 
	
	Further improvements in running times---up to obtaining a
	linear in $np$ computational comp\-lexity in the case of a 
	constant $\varepsilon$---are presented in \citep{Cheng19}. 
	Some lower bounds suggesting that the log-factor in the term 
	$\varepsilon\sqrt{\log(1/\varepsilon)}$ cannot be removed 
	from the rate of computationally tractable estimators 
	are established in \citep{DiakonikolasKS17}. In a slightly 
	weaker model of corruption, \cite{DiakonikolasKK018} 
	propose an iterative filtering algorithm that achieves 
	the optimal rate $\varepsilon$ without the extra factor
	$\sqrt{\log(1/\varepsilon)}$. On a related note, \citep{coldal} 
	shows that in a weaker contamination model termed as parametric
	contamination, the carefully trimmed sample mean can achieve a better
	rate than that of the coordinatewise/geometric median. 
	
	An overview of the recent advances on robust estimation with a 
	focus on computational aspects can be found in \citep{Diak_review}. 
	Extensions of these methods to the sparse mean estimation are 
	developed in \citep{BalakrishnanDLS17, DiakonikolasKKP19}. 
	All these results are proved to hold on an event with 
	a prescribed probability, see \citep{bateni2020minimax} 
	for a relation between results in expectation and 
	those with high probability, as well as for the 
	definitions of various types of contamination. 
	
	The proposed estimator shares some features
	with the adaptive weights smoothing \citep{PolzehlSpok}. 
	Adaptive weights smoothing (AWS) iteratively updates the weights
	assigned to observations, similarly to Algorithm~\ref{algo:a1}.
	The main difference is that the weights in AWS are not measuring
	the outlyingness but the relevance for interpolating a function
	at a given point. There are also many other statistical problems
	in which robust estimation has been recently revisited from
	the point of view of minimax rates. This includes scale
	and covariance matrix estimation \citep{chen2018,
	comminges2018adaptive}, matrix completion \citep{Klopp2017}, 
	multivariate regression \citep{Gao2020,Thompson,geoffrey2019erm}, 
	classification \citep{cannings2020classification,Bradic}, subspace 
	clustering \citep{Soltanolkotabi}, community detection \citep{Cai}, etc. 
	Properties of robust $M$-estimators in high-dimensional settings 
	are studied in \citep{Loh,Elsener}.  
	There is also an increasing body of literature on the robustness 
	to heavy tailed	distributions \citep{Lugosi1,Lugosi2,Devroye, 
	Lecue2019,lecu2017robust,Minsker} and the computationally tractable 
	methods in this context \citep{Hopkins1,flammarion,DongH019,
	depersin2019robust}. 

	A potentially useful observation, from 
	a computational standpoint, is that it is sufficient 
	to solve the optimization problem in \Cref{iter:1} 
	up to an error proportional to $\sqrt{\rSigma/n} 
	+ \sqrt{\varepsilon}$. Indeed, one can easily repeat
	all the steps in \eqref{eq:3} to check that this
	optimization error does not alter the order of magnitude 
	of the statistical error.

	\section{Formal statement of main building blocks}\label{sec:formal}
	
	The first building block, inequality \eqref{main:ineq}, 
	used in Section \ref{sec:method} to analyze the risk of 
	$\hat{\bmu}^{\textup{IR}}_n$, upper bounds the error 
	of estimating the mean by the error of estimating
	the covariance matrix. In order to formally state 
	the result, we need some additional notations.
	
	Let $\bw\in\bs\Delta^{n-1}$ be a vector of weights 
	and let $I$ be a subset of $\{1,\ldots,n\}$. We use 
	the notation $\bw_{I}$ for the vector obtained from 
	$\bw$ by zeroing all the entries having indices outside 
	$I$. Considering $\bw$ as a probability on $\{1,\ldots, 
	n\}$, we define $\bw_{|I}$ as the corresponding 
	conditional probability on $I$ that is 
	\begin{align}
	\bw_{|I}\in \bs\Delta^{n-1},\qquad
	(\bw_{|I})_i = (w_i/ \|\bw_{I}\|_1) 
	\mathds 1(i\in I).
	\end{align}
	We will make repeated use of the notation
	\begin{align}
			\bar{\bs X}_{\bw} = \sum_{i=1}^n w_i \bs X_i,\qquad     
			\bar{\bxi}_{\bw_{|I}} = \sum_{i\in I} 
			(\bw_{|I})_i \bxi_i,\qquad
			\bar{\bzeta}_{\bw_{|I}} = \sum_{i\in I} 
			(\bw_{|I})_i \bzeta_i.
	\end{align}

	\begin{proposition}\label{proposition1}
	Let $\bzeta_1,\ldots,\bzeta_n$ be a set of vectors such 
	that $\bzeta_i = \bs X_i-\bmu^*$ for every $i\in \mathcal I$, 
	where $\mathcal I$ is a subset of $\{1,\ldots,n\}$. For every 
	weight vector $\bw\in \bs \Delta^{n-1}$ such that 
	$\sum_{i\not\in \mathcal I} w_i\le \varepsilon_{w}\le 1/2$ and 
	for every $p\times p$ matrix $\bSigma$, it holds  
	\begin{align}
	\|\bar{\bs X}_{\bw} - {\bmu}^*\|_2 
		&\le \frac{\sqrt{\varepsilon_{w}}}{1-\varepsilon_{w}}\,
			\lambda_{\max,+}^{1/2}\Big(\sum_{i=1}^n  w_i(\bs X_i - 
			\bar{\bs X}_{\bw})^{\otimes 2} 
			- \bSigma\Big) + R(\bzeta,\bw,\mathcal I),
	\end{align}
	with the remainder term
	\begin{align}
	R(\bzeta,\bw,\mathcal I) 
		= 2\sqrt{\|\bSigma\|_{\rm op}}\varepsilon_{w}+\sqrt{2\varepsilon_{w}}\,
		\lambda_{\max,+}^{1/2}
		\bigg(\sum_{i\in \mathcal I} (\bw_{|\mathcal I})_i (\bSigma - 
		\bzeta_i\bzeta_i^\top)\bigg) + (1+\sqrt{2\varepsilon_{w}}) 
		\|\bar{\bzeta}_{\bw_{|\mathcal I}}\|_2.
	\end{align}
	\end{proposition}

    The proof of this result is postponed to the last 
    section. In simple words, the claim of proposition 
    is that the estimation error of the weighted mean 
    $\bar\bX_{\bw}$ is, up to a remainder term, governed 
    by the quantity $G(\bw,\bar{\bX}_{\bw})^{1/2}$. It turns out 
    that the remainder term is bounded by a small quantity 
    uniformly in $\bw$ and $\mathcal I$, provided that 
    these two satisfy suitable conditions. For $\mathcal I$,
    it is enough to constrain the cardinality of its complement 
    $\mathcal I^c = \mathcal O$. For $\bw$, it appears to be 
    sufficient to assume that its sup-norm is small. In that 
    respect, the following lemma plays a key role in the 
    proof.
    
    \begin{lemma}\label{lemma0}
	For any integer $\ell>0$, let $\mathcal W_{n,\ell}$ be 
	the set of all $\bw\in\bDelta^{n-1}$ such that $\max_i 
	w_i\le 1/\ell$. The following facts hold: 
	\vspace{-10pt}
	\begin{enumerate}
		\item[\textup{   i)}] For every $J\subset \{1,\ldots,n\}$ such 
		that $|J|\ge \ell$, the uniform weight vector $\bu^{J}\in 
		\mathcal W_{n,\ell}$.
		\item[\textup{  ii)}] The set $\mathcal W_{n,\ell}$ is 
		the convex hull of the uniform weight vectors $\{\bu^{J} : 
		|J| = \ell\}$.
		\item[\textup{ iii)}] For every convex mapping 
		$G:\bs\Delta^{n-1}\to\RR$,  we have 
		\begin{align}
			\sup_{\bw\in \mathcal W_{n,\ell}} G(\bw) = 
			\max_{|J| =\ell} G(\bu^J),
		\end{align}
		where the last maximum is over all subsets $J$ of cardinality 
		$\ell$ of the set $\{1,\ldots,n\}$.
		\item[\textup{ iv)}] If $\bw\in\mathcal W_{n,\ell}$ then for any 
		$I$ such that $|I|\ge \ell'>n-\ell$, we have 
		$\bw_{|I}\in \mathcal W_{n,\ell+\ell'-n}$.
	\end{enumerate} 
	\vspace{-10pt}
	\end{lemma}

    Let us denote by $\mathcal W_n(\varepsilon)$ the set 
    $\mathcal W_{n,n(1-\varepsilon)}$. This is exactly the 
    feasible set in the optimization problem defining the 
    iterations of Algorithm~\ref{algo:a1}. It is clear that 
    for $\bw\in \mathcal W_n(\varepsilon)$ and for 
    $| I^c|\le n\varepsilon$, we have $\sum_{i\not\in 
     I} w_i\le \varepsilon/(1-\varepsilon)$. We now 
    infer from \Cref{proposition1} and \eqref{eq:minG} that 
    \begin{align}
		\|\bar{\bs X}_{\bw} - {\bmu}^*\|_2 
		&\le \alpha_\varepsilon\,\lambda_{\max,+}^{1/2}
		\Big(\sum_{i=1}^n  w_i(\bs X_i - \bar{\bs X}_{\bw})
		^{\otimes 2} - \bSigma\Big) + \Xi\\
		& = \alpha_\varepsilon\, \inf_{\bmu\in\mathbb R^p}G(\bw,
		\bmu)^{1/2} + \Xi,
		\label{eq:fund3}
	\end{align}
	with $\Xi$ being the largest value of $R(\bzeta,\bw, I)$ 
	over all possible weights $\bw\in \mathcal W_n(\varepsilon)$ 
	and subsets $I\subset\{1,\ldots,n\}$ satisfying $| 
	I^c|\le n\varepsilon$. The second building block, formally 
	stated in the next proposition, provides a suitable upper
	bound for the random variable $\Xi$. 
	
	\begin{proposition}\label{proposition2}
	Let $R(\bzeta,\bw, I)$ be defined in \Cref{proposition1} and 
	$\bzeta_1,\ldots,\bzeta_n$ be i.i.d. centered Gaussian random 
	vectors with covariance matrix $\bSigma$ satisfying 
	$\lambda_{\max}(\bSigma) = 1$. If $\varepsilon\le 0.28$, 
	then the random variable 
	\begin{align}
		\Xi = \sup_{\bw\in \mathcal W_n(\varepsilon)} 
		\max_{|I|\ge n(1-\varepsilon)} R(\bzeta,\bw, I)
	\end{align}
	satisfies, for a universal constant $C>0$, the inequalities 
	\begin{align}
		\|\Xi\|_{\mathbb L_2} &\le \sqrt{\rSigma/n}(1+ C
		\sqrt{\varepsilon}) + C\sqrt{\varepsilon}(\rSigma/n)^{1/4} 
		+ C\varepsilon\sqrt{\log(1/\varepsilon)} \label{ineq:xi1}\\
		\|\Xi\|_{\mathbb L_2} &\le \sqrt{p/n}(1+ 16\sqrt{\varepsilon})
		+ 5\sqrt{3\varepsilon}\,(p/n)^{1/4} + 
		32\varepsilon\sqrt{\log(2/\varepsilon)},\label{ineq:xi2}
	\end{align}
	where for the second inequality we assumed that $p\ge 2$ and
	$n\ge p\vee 4$.
	\end{proposition}
	
	The second inequality is weaker than the first one, since obviously $\rSigma \le p$. However, the advantage of the second inequality
	is that it comes with explicit constants and shows that these constants are not excessively large.

    To close this section, we state a theorem that rephrases 
    \Cref{fact:f4} in a way that might be more convenient for 
    future references. Its proof is omitted, since it follows 
    the lines of the proof of \Cref{fact:f4} presented above.
    
    \begin{theorem}\label{th:1}
    Let $\hat\bmu_n^{\textup{IR}}$ be the iteratively reweighted 
    mean defined in~\Cref{IR:def} and Algorithm~\ref{algo:a1}. 
    There is a universal constant $C>0$ such that for any 
    $n,p\ge 1 $ and for every $\varepsilon < (5-\sqrt{5})/10$, 
    we have
    \begin{align}
        \sup_{\bmu^*\in\RR^p}
        \sup_{\bfPn\in\textup{GAC}(\bmu^*,\bSigma,\varepsilon)} 
        \bfE^{1/2}[\|\hat\bmu_n^{\textup{IR}}-\bmu^*\|_2^2] 
        \le \frac{C\|\bSigma\|^{1/2}_{\textup{op}}}{1 - 
        2\varepsilon - \sqrt{\varepsilon(1-\varepsilon)}}\,
        \big(\sqrt{\rSigma/n} + \varepsilon\sqrt{\log(1/
        \varepsilon)}\big),
    \end{align}
    where $\bfPn\in\textup{GAC}(\bmu^*,\bSigma,\varepsilon)$ means that 
    the data points $\bX_i$ are Gaussian with adversarial contamination, 
    see \Cref{def:1}. If, in addition, $p\ge 2$ and $n\ge p\vee 10$, then 
    \begin{align}
        \sup_{\bmu^*\in\RR^p}
        \sup_{\bfPn\in\textup{GAC}(\bmu^*,\bSigma,\varepsilon)} 
        \bfE^{1/2}[\|\hat\bmu_n^{\textup{IR}}-\bmu^*\|_2^2] 
        \le \frac{10\|\bSigma\|^{1/2}_{\textup{op}}}{1-2\varepsilon - \sqrt{\varepsilon(1-\varepsilon)}}\,
        \big(\sqrt{p/n} + \varepsilon\sqrt{\log(1/\varepsilon)}\big).
    \end{align}
    \end{theorem}
    
    To the best of our knowledge, this is the first result
    in the literature that provides an upper bound on the
    expected error of an outlier-robust estimator, which is 
    of nearly optimal rate. 
    
    \section{Sub-Gaussian distributions, 
    high-probability bounds and adaptation}\label{sec:extensions}
    
    Risk bounds stated in \Cref{fact:f4} and \Cref{fact:f5} 
    and formalized in \Cref{th:1} hold for the expected error
    under the condition that the reference distribution is 
    Gaussian. Furthermore, the proposed procedure relies on the
    knowledge of both the contamination rate $\varepsilon$ and 
    the covariance matrix $\bfSigma$. The goal of this section 
    is to show how some of these restriction can be alleviated. 
    
    \subsection{High-probability risk bound for a sub-Gaussian 
    reference distribution} 
    As expected, the risk bounds established in previous sections
    can be extended to the case of sub-Gaussian distributions.
    Furthermore, risk bounds holding with high-probability can be
    proved using the same techniques as those employed for
    proving the in-expectation bounds. In  order to be more
    precise, we state in this subsection the high-probability 
    counterpart of the second claim of \Cref{th:1}. The price
    to pay for covering the more general sub-Gaussian case is
    that the constant in the right hand side of the inequality 
    is no longer explicit.
    
    Recall that a zero-mean random vector $\bxi$ is called 
    sub-Gaussian with parameter $\tau>0$ (also known as the 
    variance proxy), if
	\begin{align}
	    \mathbf E[e^{\bv^\top\bxi}] \le 
	    e^{\nicefrac{\tau}2 \|\bv\|_2^2},\qquad 
	    \forall \bv\in\mathbb R^p.
	\end{align}
	We write $\bxi \sim SG_p(\tau)$. If $\bxi$ is standard 
	Gaussian then it is sub-Gaussian with parameter $1$. 
	Similarly, if $\bxi$ is centered and belongs almost surely 
	to the unit ball, then $\bxi$ is sub-Gaussian with 
	parameter 1. Let us describe now the set of data-generating 
	distributions that we consider in this section. 
	
	\begin{definition}\label{def:SGAC}
	We say that the joint distribution $\bfPn$ of the random 
	vectors $\bX_1,\ldots,\bX_n\in\mathbb{R}^p$ belongs to
	the sub-Gaussian model with adversarial contamination
	with mean $\bmu^*$, covariance matrix $\bSigma$ and 
	contamination rate $\varepsilon$, if there are independent
	random vectors $\bxi_i\sim SG_p(\tau)$ such that 
	\begin{align}
	    \Big|\Big\{i:=1,\ldots,n:\bX_i \neq \bmu^*+ \bSigma^{1/2}
	    \bxi_i\Big\}\Big|\le \varepsilon n.
	\end{align}
	We then write $\bfPn\in \textup{SGAC}(\bmu^*,\bSigma, 
	\varepsilon)$.
	\end{definition}
	
	It is clear that a Gaussian model with adversarial contamination
	defined in \Cref{def:1} is a particular case of the sub-Gaussian
	model with adversarial contamination. In other terms, the set
	\textup{SGAC}$(\bmu^*,\bSigma,\varepsilon)$ is strictly larger
	than the set \textup{GAC}$(\bmu^*,\bSigma,\varepsilon)$. 
	Nevertheless, as shows the result below, the risk bounds 
	established for the iteratively reweighted mean algorithms 
	remain valid uniformly over this extended class \textup{SGAC}$(\bmu^*,\bSigma,\varepsilon)$. 
	
    \begin{theorem}\label{th:2}
    Let $\hat\bmu_n^{\textup{IR}}$ be the iteratively reweighted 
    mean defined in~\Cref{IR:def} and in Algorithm~\ref{algo:a1}. 
    Let $\delta\in(4e^{-n},1)$ be a tolerance level. 
    There exists a constant $A_5$ depending only on the variance 
    proxy $\tau$ such that if $n\ge p\ge 2$ and $\varepsilon < 
    (5-\sqrt{5})/10$, then for every $\bmu^*\in\RR^p$ and every $\bfPn\in\textup{SGAC}(\bmu^*,\bSigma,\varepsilon)$, we have
    \begin{align}
        \bfP\Bigg(\|\hat\bmu_n^{\textup{IR}}-\bmu^*\|_2^2 
        \le \frac{A_5\|\bSigma\|^{1/2}_{\textup{op}}}{1-2
        \varepsilon - \sqrt{\varepsilon(1-\varepsilon)}}\,
        \bigg(\sqrt{\frac{p+\log(4/\delta)}{n}} + \varepsilon\sqrt{\log(1/\varepsilon)}\bigg)\Bigg) 
        \le 1 - 4\delta.
    \end{align}
    \end{theorem}
    
    The proof of this theorem is postponed to the supplementary 
    material. Let us just mention that \cite[Section 1.2]{Chengetal19}
    claim that the rate $\sqrt{p/n} + \varepsilon\sqrt{\log(1/
    \varepsilon)}$ is optimal for sub-Gaussian distributions, meaning 
    that, unlike the Gaussian case, the $\sqrt{\log(1/\varepsilon)}$ 
    factor cannot be removed. A formal proof of this fact can be found in the last remark of Section 2 in \citep{LugMen20}. 
    
    \def\ellmax{\ell_{\max}}
    \subsection{Adaptation to unknown contamination rate $\varepsilon$}\label{ssec:Lepski} 
    An appealing feature of the risk bounds that hold with 
    high probability is that they allow us to apply Lepski's 
    method \citep{Lepski, LepskiSpok} for obtaining an adaptive 
    estimator with respect to $\varepsilon$. The obtained 
    adaptive estimator enjoys all the five properties enumerated 
	in \Cref{sec:method} except the asymptotic efficiency, 
	since the adaptation results in an inflation of the risk 
	bound by a factor 3. The precise description of the algorithm, 
	already used in the framework of robust estimation by 
	\cite{coldal}, is presented below. We will denote by
	$\mathbb B(\bmu,r)$ the ball with center $\bmu$ and radius
	$r$ in the Euclidean space $\mathbb R^p$.
	
	\begin{definition}\label{def:AIR}
	We choose a geometric grid $\varepsilon_\ell = a^\ell
	\varepsilon_0$,	$\ell = 1,2,\ldots,\ellmax$, of possible 
	values of the contamination rate. Here, $a\in(0,1)$ is a real 
	number, $\varepsilon_0 = (5-\sqrt{5})/10$ and
	$\ellmax = [0.5\log_a (p/n)]$. For each $\ell=1,\ldots,
	\ellmax$, we denote by $\hat\bmu_n^{\textup{IR}}(
	\varepsilon_\ell)$  the iteratively reweighted mean computed
	for $\varepsilon = \varepsilon_\ell$, see 
	Algorithm~\ref{algo:a1}, and we set
	\begin{align}
	    R_\delta(z) = \frac{A_5\|\bSigma\|^{1/2}_{\textup{op}}}
	    {1-2z - \sqrt{z(1-z)}}\,
        \bigg(\sqrt{\frac{p+\log(4\ellmax/\delta)}{n}} + 
        z\sqrt{\log(1/z)}\bigg), \qquad z\in[0, \varepsilon_0),
	\end{align}
	where $\delta\in(0,1)$ is a tolerance level and $A_5$
	is the constant from \Cref{th:2}. The adaptively chosen
	iteratively reweighted mean estimator $\hat\bmu_n^{\textup{AIR}}$
	is defined by $\hat\bmu_n^{\textup{AIR}} = \hat\bmu_n^{\textup{IR}}
	(\varepsilon_{\hat\ell})$ where
	\begin{align}
	    \hat\ell = \max\Big\{\ell\le \ellmax : \bigcap_{j=1}^\ell 
	    \mathbb B\big(\hat\bmu_n^{\textup{IR}}
	(\varepsilon_j);R_\delta(\varepsilon_j)\big)\neq \varnothing\Big\}.
	\end{align}
	\end{definition}
	
	The estimator $\hat\bmu_n^{\textup{AIR}}$ can be computed without 
	the knowledge of the true contamination rate $\varepsilon$. 
	Furthermore, its computational complexity nearly of the same
	order as the complexity of computing a single instance of
	the iteratively reweighted mean as defined by Algorithm~\ref{algo:a1}.
	Indeed, to compute $\hat\bmu_n^{\textup{AIR}}$, one needs to
	apply Algorithm~\ref{algo:a1} at most $\ellmax = [0.5
	\log_a (p/n)]$ times, and to solve a second-order cone program
	for checking whether the intersection of a small number of
	balls is empty. The next theorem, proved in the supplementary 
	material, shows that the estimation error of this estimator 
	$\hat\bmu_n^{\textup{AIR}}$ is of the optimal rate, up to a
	logarithmic factor. To the best of our knowledge, this is the
	first result of this kind in the literature.

	\begin{theorem}\label{th:3}
	Let $\hat\bmu_n^{\textup{AIR}}$ be the estimator 
    defined in~\Cref{def:AIR}. Let $\delta\in(4e^{-n},1)$ be a 
    tolerance level. Let $n\ge p\ge 2$ and $\varepsilon \le 
    (5-\sqrt{5})a/10$, where $a\in(0,1)$ is the parameter used 
    in \Cref{def:AIR}. Then, for every $\bmu^*\in\RR^p$ and every $\bfPn\in\textup{SGAC}(\bmu^*,\bSigma,\varepsilon)$, we have
    \begin{align}
        \bfP\Bigg(\|\hat\bmu_n^{\textup{AIR}}-\bmu^*\|_2^2 
        \le \frac{3A_5\|\bSigma\|^{1/2}_{\textup{op}}}{a-2
        \varepsilon - \sqrt{\varepsilon(a-\varepsilon)}}\,
        \bigg(\sqrt{\frac{p+\log(4\ellmax/\delta)}{n}} + \varepsilon\sqrt{\log(1/\varepsilon)}\bigg)\Bigg) 
        \le 1 - 4\delta.
    \end{align}
	\end{theorem}
	
	The breakdown point of the adaptive estimator $\hat\bmu_n^
	{\textup{AIR}}$ inferred from the last theorem is slightly 
	smaller than the one of $\hat\bmu_n^{\textup{IR}}$. Indeed,
	there is a factor $a<1$ between these two quantities. 
	Note that one can choose $a$ to be very close to one. 
	The only drawback of choosing $a$ too close to one is 
	the higher computational complexity of the resulting 
	estimator. 
	
		\begin{figure*}
	\begin{center}
	\begin{minipage}{0.8\linewidth}
			\begin{alg}{Iteratively reweighted mean estimator (known 
			$\bs\varepsilon$, unknown $\bSigma$)}{a2}		
				\begin{algorithmic}
				\small
				\STATE {\bfseries Input:} data $\bX_1,\ldots,\bX_n\in\RR^p$,  
				contamination rate $\varepsilon$ 
				\STATE {\bfseries Output:} parameter estimate $\hat\bmu_n^{\textup{IR}}$
				\STATE {\bfseries Initialize:} compute $\hat\bmu^0$ as a minimizer of 
				$\sum_{i=1}^n \|\bX_i-\bmu\|_2$\\[7pt]
				\STATE  Set $K = 0\vee\Big\lceil \frac{\log(4p) - 2\log(\varepsilon({1-2\varepsilon}))}
				{2\log(1-2\varepsilon) - \log \varepsilon -\log(1-\varepsilon)}
				\Big\rceil$.\\[7pt]
				\STATE {{\bf For} $k=1:K$}
				\STATE \qquad {Compute current weights:}\\
				\qquad\quad $\ds	\bw \in \mathop{\textup{arg\,min}}_{(n-n\varepsilon)\|\bw\|_\infty\le 1}
					\lambda_{\max}\bigg(\sum_{i=1}^n  w_i (\bs X_i-\hat\bmu^{k-1})
					^{\otimes 2}\bigg).		
				$
				\STATE \qquad{Update the estimator: $\hat{\bmu}^k = \sum_{i=1}^n w_i\bX_i$}.
				\STATE {\bfseries EndFor}
				\STATE {{\bf Return} $\hat\bmu^{K}$}.
				\end{algorithmic}		
			\end{alg}
		\end{minipage}
	\end{center}
\end{figure*}%

    \subsection{Extension to unknown covariance $\bfSigma$}
    
    The iteratively reweighted mean estimator, as defined
    in Algorithm~\ref{algo:a1}, requires the knowledge of
    the covariance matrix $\bSigma$. Let us briefly discuss
    what happens when this matrix is unknown, by considering
    two qualitatively different situations. 
    
    The first situation is when the covariance matrix is 
    isotropic, $\bSigma = \sigma^2\bfI_p$, with unknown 
    $\sigma>0$. One can easily check that all the claims
    of \Cref{sec:method} and \Cref{sec:formal} hold true
    for a slight modification of $\hat\mu_n^{\textup{IR}}$
    obtained by minimizing $\lambda_{\max}(\sum_i w_i(\bX_i
    -\hat\mu^{k-1})^{\otimes 2}-\bSigma)$ instead of 
    $G(\bw,\hat\bmu^{k-1})$ in \eqref{iter:1}. For isotropic
    matrices $\bSigma$, we have
    \begin{align}
        \text{arg}\min_{\bw}\lambda_{\max}\bigg
        (\sum_i w_i(\bX_i-\hat\mu^{k-1})^{\otimes 2}-\sigma^2
        \bfI_p\bigg) = 
        \text{arg}\min_{\bw}\lambda_{\max}\bigg
        (\sum_i w_i(\bX_i-\hat\mu^{k-1})^{\otimes 2}\bigg),
    \end{align}
    which means that the resulting value is independent of 
    $\bSigma$. Therefore, in this first situation, the 
    modified estimator defined in Algorithm~\ref{algo:a2} 
    satisfies inequalities of \Cref{th:1} and \Cref{th:2}. 
    
    The second situation is when $\bSigma$ is unknown and
    arbitrary. In this case, to the best of our knowledge, 
    there is no known computationally tractable estimator 
    of $\bmu^*$ achieving a rate faster than $\sqrt{p/n} +  
    \sqrt{\varepsilon}$. It turns out that a slightly modified
    version of the iteratively reweighted mean estimator
    defined in Algorithm~\ref{algo:a2} achieves this rate
    as well. The formal statement of the result entailing 
    this claim is presented below. 
    \begin{theorem}\label{th:4}
    Let $\hat\bmu_n^{\textup{IR}}$ be the iteratively reweighted 
    mean defined in Algorithm~\ref{algo:a2}. 
    Let $\delta\in(4e^{-n},1)$ be a tolerance level. 
    There exists a constant $A_5$ depending only on the variance 
    proxy $\tau$ such that if $n\ge p\ge 2$ and $\varepsilon < 
    (5-\sqrt{5})/10$, then for every $\bmu^*\in\RR^p$ and every $\bfPn\in\textup{SGAC}(\bmu^*,\bSigma,\varepsilon)$, we have
    \begin{align}
        \bfP\Bigg(\|\hat\bmu_n^{\textup{IR}}-\bmu^*\|_2^2 
        \le \frac{A_5\|\bSigma\|^{1/2}_{\textup{op}}}{1-2
        \varepsilon - \sqrt{\varepsilon(1-\varepsilon)}}\,
        \bigg(\sqrt{\frac{p+\log(4/\delta)}{n}} + 3\sqrt{\varepsilon}\bigg)\Bigg) \le 1 - 4\delta.
    \end{align}
    \end{theorem}
    
    The proof of this theorem is deferred to the supplementary
    material. One can use \Cref{th:4} to construct an adaptive
    (with respect to $\varepsilon$) estimator of $\bmu^*$ in
    the case of unknown $\bSigma$ using Lepski's method detailed
    in \Cref{ssec:Lepski}. For this construction, it suffices 
    to know an upper bound $\sigma_{\max}$ on the operator 
    norm $\|\bSigma\|^{1/2}_{\textup{op}}$. The resulting 
    estimator has an error of the order $\sigma_{\max}(
    \sqrt{p/n} + \sqrt{\varepsilon})$. 
    
    One can also consider intermediate cases, in which the
    covariance matrix is not arbitrary but has a more general
    form than the simple isotropic one. In such a situation, 
    it might be of interest to extend the method proposed in 
    Algorithm~\ref{algo:a1} by using an initial estimator of 
    $\bSigma$ and	by updating its 
    value at each step. Indeed, when a weight vector is computed, 
    it can be used for updating not only the mean but also the 
	covariance matrix. The study of this estimator is left for
	future research.
    
    \section{Empirical results}\label{sec:empiric}

    This section showcases empirical results obtained by 
    applying the iteratively reweighted mean estimator 
    described in Algorithm~\ref{algo:a1} to synthetic 
    data sets. We stress right away that there are multiple 
    ways of solving the optimization problem involved in 
    Algorithm~\ref{algo:a1}, and the implementation we used 
    in our experiments is not the most efficient one. As 
    already mentioned, the aforementioned optimization 
    problem can be seen as a semi-definite program 
    and out-of-shelf algorithm can be applied to solve it.  
    We implemented this approach in R using the MOSEK 
    solver. All the results reported in this section 
    are obtained using this implementation. We are 
    currently working on an improved implementation using 
    the dual sub-gradient algorithm of \citep{cox2014dual}.

    We applied Algorithm \ref{algo:a1} to $\bX_1, \dots, 
    \bX_n$ from $\bfPn \in \textup{GAC}(\bmu^*,\bSigma,
    \varepsilon)$ with various types of contamination 
    schemes. In the numerical experiments below, we 
    illustrate (a) the evolution of the estimation error 
    along the iterations, (b) the properties related to 
    the theoretical breakdown point and (c) the performance 
    of the estimator obtained from Algorithm \ref{algo:a1} 
    as compared to some simple estimators of the mean and 
    to the oracle. In this section, the error of estimation
    is understood as the Euclidean distance between the
    estimated mean and its true value.

    Notice that, due to the equivariance stated in 
    \Cref{fact:f2}, it is sufficient to take as the true 
    target mean vector $\bmu^*$ the zero vector 
    $\mathbf{0}_p$ and as $\bSigma$ any  diagonal matrix 
    with nonnegative entries.  We consider the following 
    two schemes of outlier generation:
    \begin{itemize}
    \item {\textbf{Contamination by ``smallest'' eigenvector:}} Sample $n$ i.i.d. observations $\bY_1, \dots, \bY_n$ drawn from  $\mathcal{N}(\mathbf{0}_p, \bfI_p)$ and compute the smallest eigenvalue $\lambda_p$ and the corresponding eigenvector $\bv_p$ of the sample covariance matrix, defined as
    \begin{align}
        \widehat{\bSigma}_n = \frac 1 n \sum_{i=1}^n (\bY_i - 
        \overline{\bY}_n)^{\otimes 2},
    \end{align}
    where $\overline{\bY}_n \stackrel{\text{def}}{=} (\bY_1 + \dots + \bY_n) / n$. Choose the  ${n \varepsilon}$ observations from $\bY_1, \dots, \bY_n$ that have the highest (in absolute value) correlation coefficient with $\bv_p$ and replace them by a vector proportional to $\bv_p$ with proportionality coefficient equal to $\sqrt{p}$.
    \item \textbf{Uniform outliers:} Sample $n$ observations according to this model 
    \begin{align}
        \bY_i = \btheta_i + \bxi_i, \text{ with } \bxi_i \stackrel{\text{i.i.d.}}{\sim} \mathcal{N}(\mathbf{0}_p, \bfI_p) \text{ for } ~ i = 1, \dots, n,
    \end{align}
    where $\btheta_i$s are all-zero vectors for $n(1-\varepsilon)$ observations $i \in \mathcal I = \{1,\ldots,n (1 - \varepsilon)\}$, while for the indices $i \not\in \mathcal{I}$ we have $\| \btheta_i \|_2 \neq 0$. We take the values of $\{\btheta_i\}_{i\in\mathcal{O}}$ to be i.i.d. from uniform distribution, \textit{i.e.}, for $i \not\in \mathcal{I}$ we take $\{\btheta_i^j\}_{j=1}^p \stackrel{\text{i.i.d.}}{\sim}\mathcal{U}[a, b]$. We took different values of $a$ and $b$ in different experiments reported below. 
\end{itemize}

\subsection{Improvement along iterations}

In this experiment we show the improvement of the estimation error along the iterations. The data for this experiment were generated according to the second contamination scheme mentioned above with $a=0.5$ and $b=2$. In \Cref{fig:boxplot}, we drew the logarithm of the estimation error (i.e., the risk between $\hat\bmu_n$ and $\bmu^*$) as a function of the iteration number. The results are obtained by averaging over 50 independent repetitions.  We observe that the error decreases very fast during the first iteration and remains almost constant during the rest of time. 
As a matter of fact, in order to speed-up the procedure, we can 
stop the iterations if the current weights $\bw$ satisfy
$\lambda_{\max}(\sum_i w_i (\bX_i -\bar\bX_{\bw})^{\otimes 2}
-\bSigma)\le \varepsilon$. It is easy to check that this 
modified estimator still possesses all the desirable properties 
described in previous sections. 

\begin{figure}
\begin{minipage}[b]{0.46\textwidth}
    \centering
    \includegraphics[origin=c, width=1.1\textwidth]{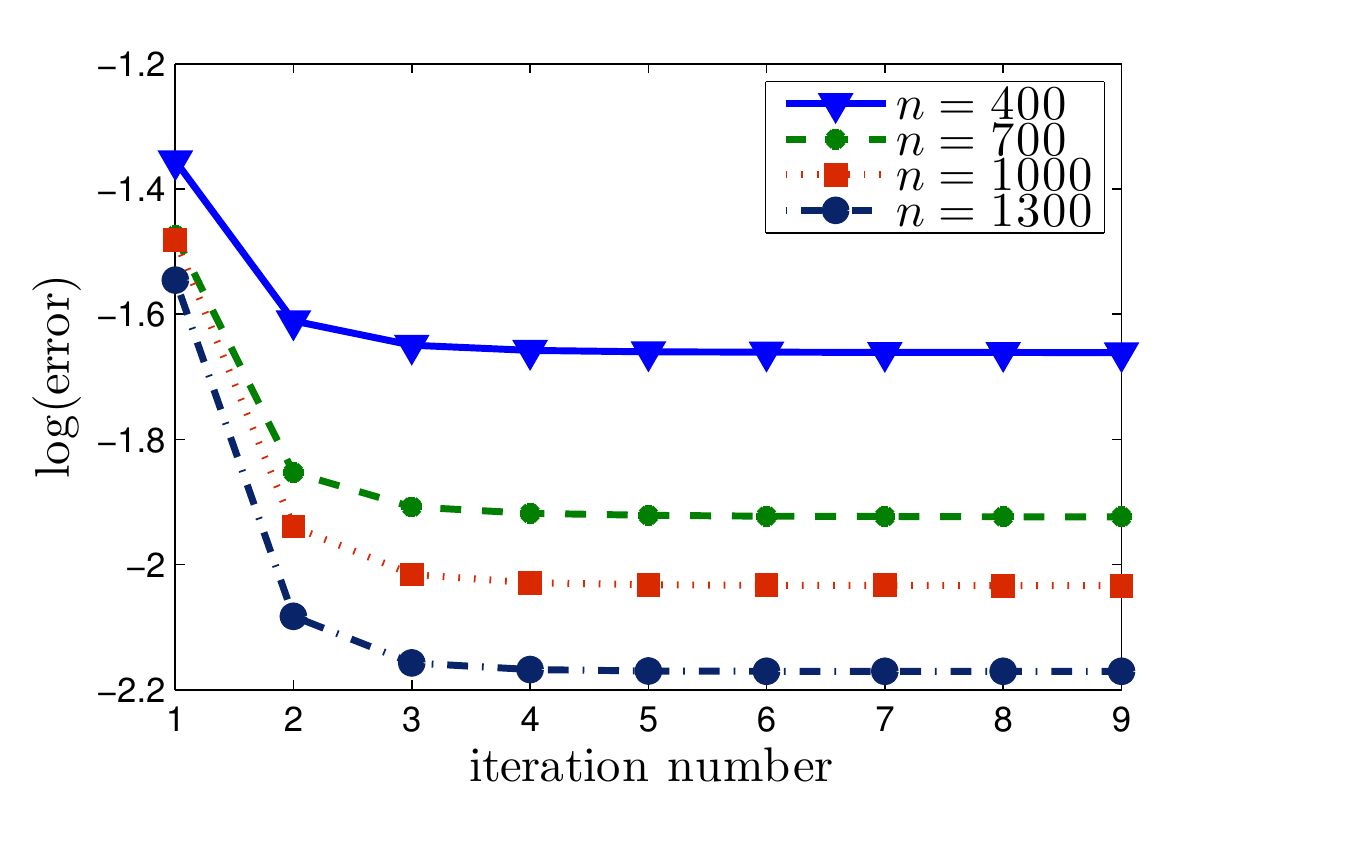}
    \vspace{-1cm}
    \caption{The decay of the error along the iterations for $p = 9$, $\varepsilon=0.2$ and different values of $n$. The contamination scheme is ``uniform outliers'' with $(a,b) = (0.5,2)$. }
    \label{fig:boxplot} 
\end{minipage}
\hfill
\begin{minipage}[b]{0.46\textwidth}
    \centering
    \includegraphics[origin=c, scale=1, width=1.1\textwidth]{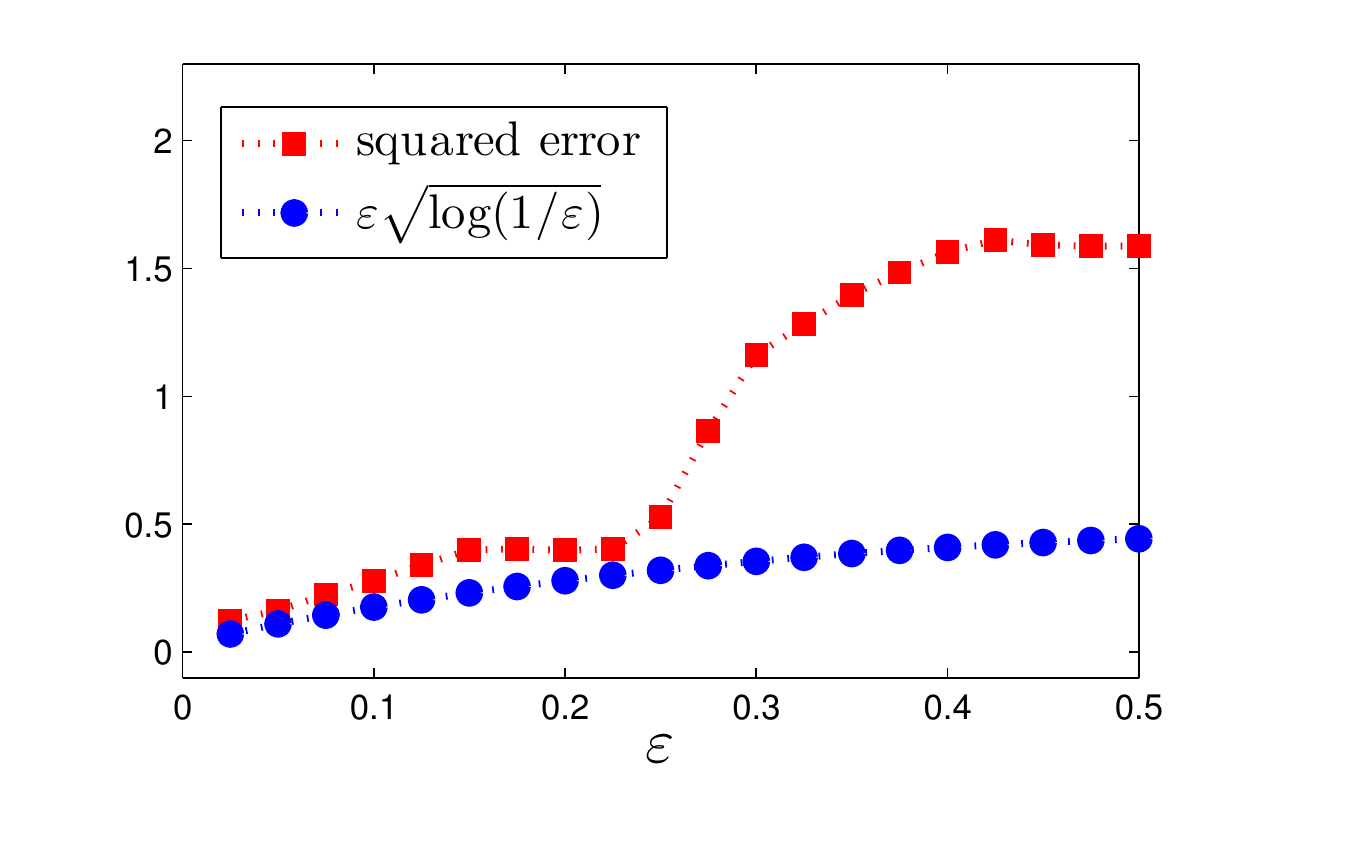}
    \vspace{-1cm}
    \caption{The error as a function of the contamination rate $\varepsilon$ for $n=500$ and $p=5$. The contamination scheme is ``smallest" eigenvector. The number of iterations for $\varepsilon > \varepsilon^*$ was set to $30$.
    }
    \label{fig:breakdown}
\end{minipage}
\end{figure}

\subsection{Breakdown point}
The goal of this experiment is to check empirically the validity of the breakdown point  $\varepsilon^* \stackrel{\text{def}}{=} (5 - \sqrt{5}) / 10\approx 0.28$. To this end, we chose to contaminate the standard normal vectors by {\it{``smallest'' eigenvector}} scheme. Note that if the outliers are well separated from the inliers, like in the uniform outliers scheme, then Algorithm~\ref{algo:a1} detects well these outliers by pushing their weights to $0$ even when $\varepsilon$ is large (larger than $0.28$).   

For $n=500$ and $p=5$, we conducted 50 independent repetitions of the experiment and plotted the error averaged over these 50 repetitions in \Cref{fig:breakdown}. 
For $\varepsilon > 0.28$, we manually set the number of iterations to\footnote{Since the number of iteration $K$ is well-defined for $\varepsilon \le \varepsilon^*_r \approx 0.28$.} 30. It is interesting to observe that there is a clear change in the mean error occurring at a value close to 
$0.28$. This empirical result allows us to conjecture that the breakdown point of the presented estimator is indeed close to $0.28$.

\subsection{Comparison with other estimators}
In this last experiment we wanted to provide a visual illustration of the performance of the proposed estimator  $\hat{\bmu}^{\text{IR}}_n$ as compared to some simple competitors: the sample mean, the coordinatewise median, the geometric median and the oracle obtained by averaging all the inliers. The obtained errors, averaged over 50 independent repetitions, are depicted in \Cref{fig:compare1} and \Cref{fig:compare2}.  The former corresponds to $n=500$, $p=20$ and varying $\varepsilon$, while the parameters of the latter are $p=10$, $\varepsilon = 0.1$ and $n \in [10, 1000]$. In the legend of these figures, \texttt{Estimated Mu} refers to $\hat{\bmu}^{\text{IR}}_n$, the output of Algorithm \ref{algo:a1}, \texttt{Sample Mean} and \texttt{Sample Median} correspond respectively to the sample mean and sample coordinatewise median,  \texttt{Geometric Median} refers to the estimator defined in \eqref{GM:def} and \texttt{Oracle} refers to the sample mean of inliers only. 

The plots show that the iteratively reweighted mean estimators has an error which is nearly as small as the error of the oracle. These errors are way smaller than the errors of the other estimators included in this experiment, which is in line with all the existing theoretical results.  

\begin{figure}
\begin{minipage}[b]{0.46\textwidth}
    \centering
    \includegraphics[scale=0.3, angle=270,width=\textwidth] {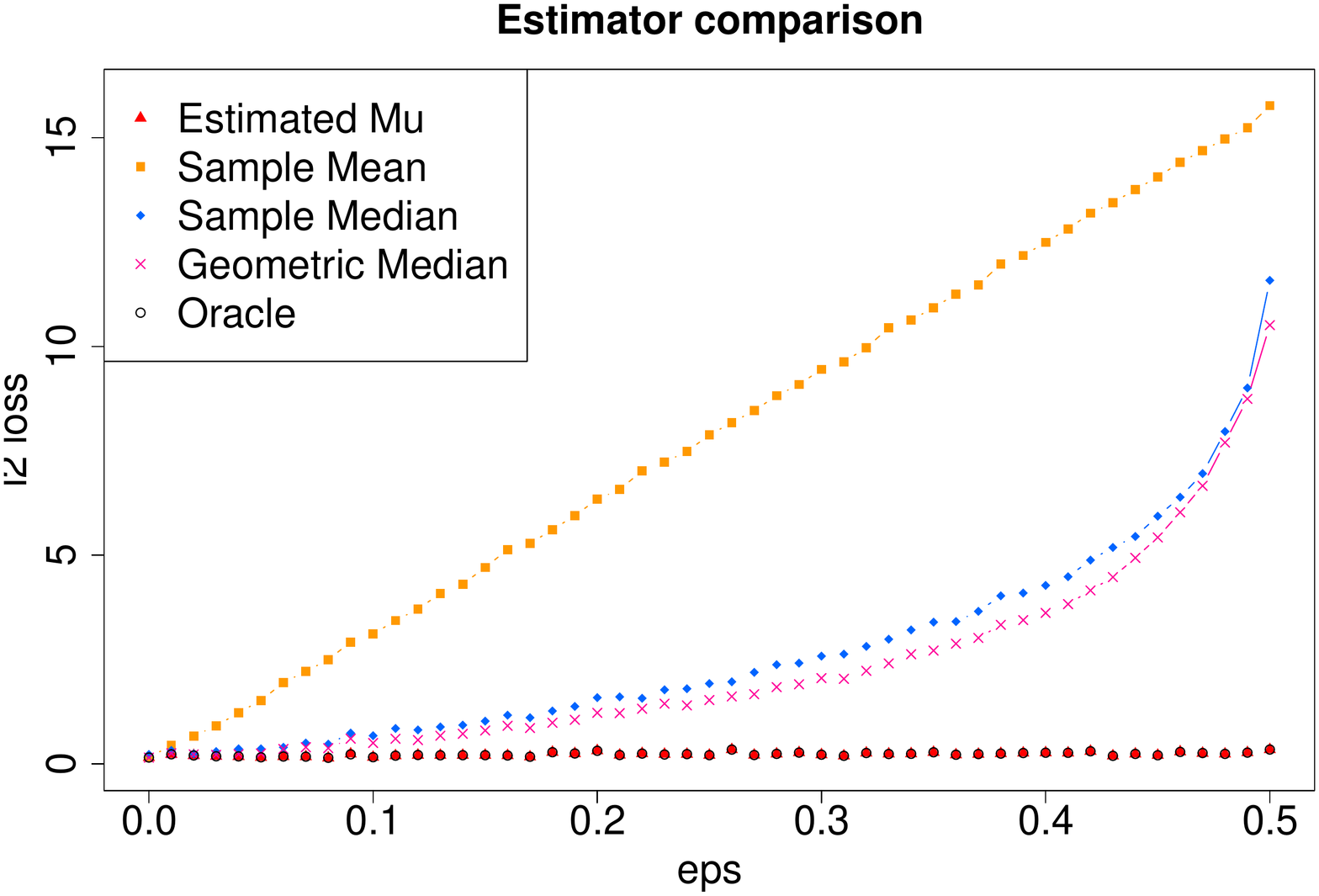}
    \caption{$\ell_2$-loss for different estimators when $n=500, p=20$ and the \textbf{uniform outliers} scheme with  $a=4$ and $b= 10$.
    \label{fig:compare1}} 
\end{minipage}
\hfill
\begin{minipage}[b]{0.46\textwidth}
    \centering
    \includegraphics[scale=0.3, angle=270, width=\textwidth]{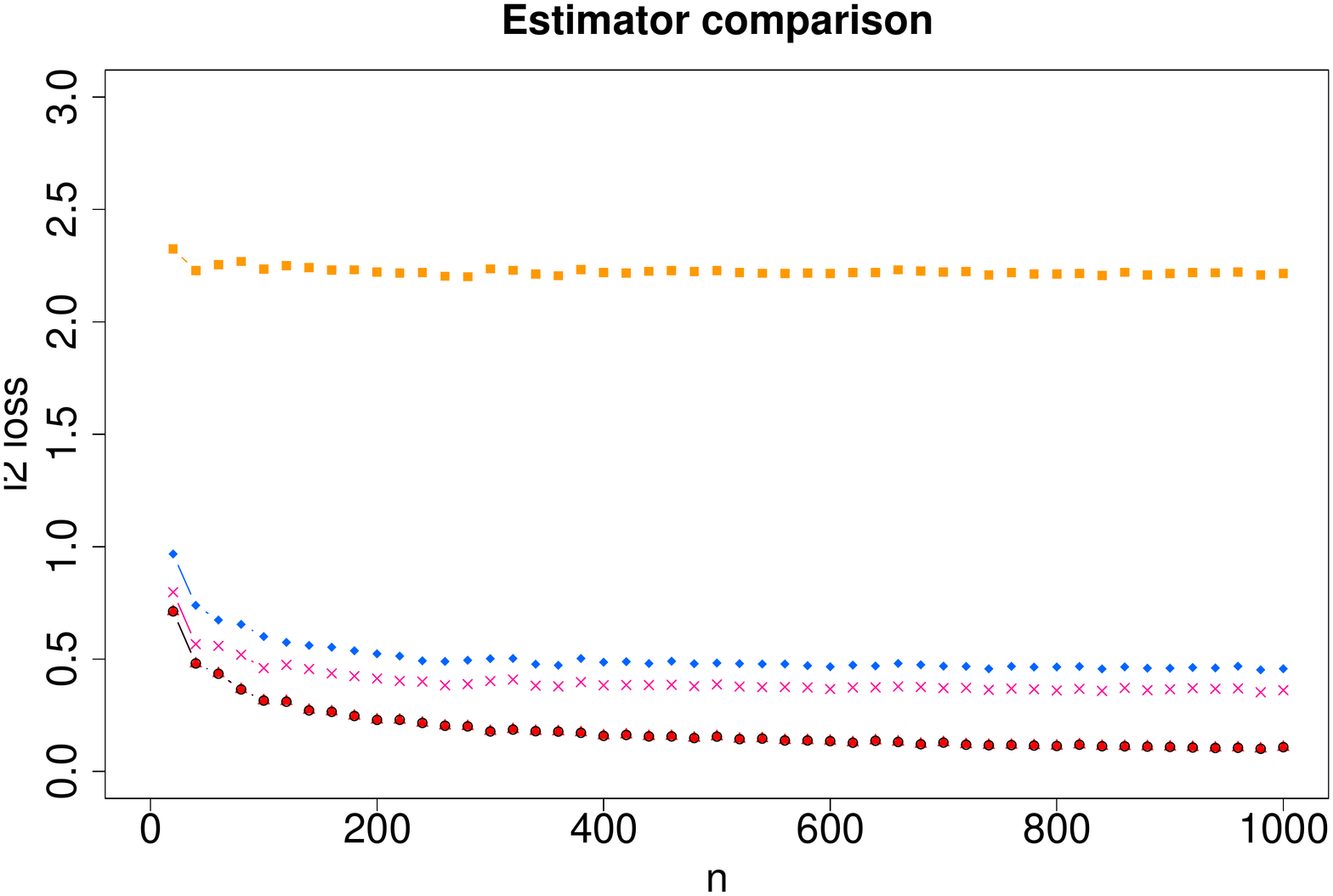}
    \caption{$\ell_2$-loss for different estimators when $p=10$, $\varepsilon=0.1$ and  the \textbf{uniform outliers} scheme with  $a=4$ and $b= 10$.
    \label{fig:compare2}}
\end{minipage}
\end{figure}



	\section{Postponed proofs}\label{sec:proofs}

	We collected in this section all the technical proofs postponed from 
	the \Cref{sec:method} and \Cref{sec:formal}. Throughout this section, we will always assume
	that $\lambda_{\max}(\bSigma) = \|\bSigma\|_{\textup{op}} = 1$. As we
	already mentioned above, the general case can be reduced to this one
	by dividing all the data vectors by $\|\bSigma\|_{\textup{op}}^{1/2}$. 
	The proofs in this section are presented according to the order of the 
	appearance of the corresponding claims in the previous sections. 
	Since the proof of \Cref{th:1} relies on several lemmas and
	propositions, we provide in \Cref{fig:diagram} a diagram
	showing the relations between these results. 
	
	\subsection{Additional details on \eqref{eq:3}}\label{ssec:6.1}
    
	One can check that
	\begin{align}
		\sum_{i=1}^n  w_i^* (\bs X_i-\hat{\bmu}^{k-1})^{\otimes 2} &= 
		\sum_{i=1}^n w_i^*(\bs X_i-\bar{\bs X}_{\bw^*})^{\otimes 2} 
		+ (\bar{\bs X}_{\bw^*}-\hat{\bmu}^{k-1})^{\otimes 2}\\
		&\preceq \sum_{i=1}^n w_i^*(\bs X_i-\bar{\bs X}_{\bw^*})^{\otimes 2}
		+ \big\|\bar{\bs X}_{\bw^*}-\hat{\bmu}^{k-1}\big\|_2^2\,\bfI_p\\
		&\preceq \sum_{i=1}^n w_i^*(\bs X_i-\bmu^*)^{\otimes 2}
		 + \big\|\bar{\bs X}_{\bw^*}-\hat{\bmu}^{k-1}\big\|_2^2\,\bfI_p.
	\end{align}	
	This readily yields
	\begin{align}
			G(\bw^*,\hat{\bmu}^{k-1}) &\le \lambda_{\max,+}\bigg(\sum_{i=1}^n 
			w_i^*(\bs X_i-\bmu^*)^{\otimes 2} - \bSigma\bigg) + 
			\big\|\bar{\bs X}_{\bw^*}-\hat{\bmu}^{k-1}\big\|_2^2\\
			&\le G(\bw^*, \bmu^*)  + 
			\big(\big\|\bar{\bzeta}_{\bw^*}\big\|_2 + 
			\big\|\bmu^*-\hat{\bmu}^{k-1}\big\|_2\big)^2.
	\end{align}
	To get the last line of \eqref{eq:3}, it suffices to apply the
	elementary consequence of the Minkowskii inequality $\sqrt{a+ (b+c)^2}
	\le \sqrt{a+b^2} + c$, for every $a,b,c>0$.


	\subsection{Rough bound on the error of the geometric median}

	This subsection is devoted to the proof of an error estimate of the
	geometric median. This estimate is rather crude, in terms of its dependence
	on the sample size, but it is sufficient for our purposes. As a matter of fact,
	it also shows that the breakdown point of the geometric median is equal to 
	$1/2$. 
	
	\begin{lemma}\label{lem:8}
	For every $\varepsilon\le 1/2$, the geometric median satisfies the inequality
	\begin{align}
		\|\GMmu - \bmu^*\|_2 &\le 	
		\frac{2}{n(1-2\varepsilon)}\sum_{i=1}^n\|\bzeta_i\|_2.\label{GM:in1}
	\end{align}
	Furthermore, its expected error satisfies
	\begin{align}
		\|\GMmu - \bmu^*\|_{\mathbb L_2} &\le 	
		 \frac{2\rSigma^{1/2}}{1-2\varepsilon} .\label{GM:in2}
	\end{align}
	\end{lemma}
	\begin{proof}
	Recall that the geometric median of $\bX_1,\ldots,\bX_n$ is defined by 
	\begin{align}
	\GMmu \in\textup{arg}\min_{\bmu\in\RR^p} \sum_{i=1}^n\|\bX_i-\bmu\|_2.
	\end{align}
	It is clear that
	\begin{align}
		\sum_{i=1}^n\|\bX_i-\GMmu\|_2 &\le \sum_{i=1}^n\|\bX_i-\bmu^*\|_2
	\end{align}
	Without loss of generality, we assume that $\bmu^*= \mathbf 0$. We also assume 
	that $n\varepsilon$ is an integer. On the one hand, 
	we have the simple bound
	\begin{align}
		n(1-\varepsilon)\|\GMmu\|_2 &\le 
		\sum_{i\in\mathcal I}\big(\|\bX_i-\GMmu\|_2 + \|\bzeta_i\|_2\big)\\ 
		&\le 
		\sum_{i=1}^n\|\bX_i-\GMmu\|_2 + \sum_{i\in\mathcal I}\|\bzeta_i\|_2
		-\sum_{i\in\mathcal O}\|\bX_i-\GMmu\|_2\\ 
		&\le 
		\sum_{i=1}^n\|\bX_i-\bmu^*\|_2 + \sum_{i\in\mathcal I}\|\bzeta_i\|_2
		-\sum_{i\in\mathcal O}\|\bX_i-\GMmu\|_2\\ 
		&\le 
		2\sum_{i\in\mathcal I}\|\bzeta_i\|_2
		+\sum_{i\in\mathcal O}\big(\|\bX_i\|_2-\|\bX_i-\GMmu\|_2\big)\\ 
		&\le 
		2\sum_{i\in\mathcal I}\|\bzeta_i\|_2
		+n\varepsilon \|\GMmu\|_2,
	\end{align}
	where the first and the last inequalities follow from triangle inequality.
	From the last display, we infer that
	\begin{align}
		\|\GMmu\|_{\mathbb L_2} &\le \frac{2}{n(1-2\varepsilon)} 
		\sum_{i=1}^n 	\|\bzeta_i\|_{\mathbb L_2}
		\le \frac{2}{1-2\varepsilon}  \|\bzeta_1\|_{\mathbb L_2} 
		\le \frac{2\rSigma^{1/2}}{1-2\varepsilon} 
	\end{align}
	and we get the	claim of the lemma.
	\end{proof}

\begin{figure}    
    \begin{tikzpicture}
  [node distance=.8cm,
  start chain=going above,]

    \node (lem6) [punktchain] 
    {\Cref{lem:smaxbis}:\\ Centered moments and 
    deviations of Gaussian matrices};
    
    \begin{scope}[start branch=hoejre,]
    \node[punktchain,on chain=going left] (lem5) {Lemmas \ref{lem:vershynin},\ref{lem:smax}:\\ 
    Moments and deviations of singular values of
    Gaussian matrices};
    \end{scope}
    
    \begin{scope}[start branch=hoejre,]
    \node[punktchain,on chain=going right] (lem1) {\Cref{lemma0}:\\ Properties of the weights from the feasible set};
    \end{scope}
    
    \node (lem7) [punktchain] 
    {\Cref{lemma3}:\\ Operator-norm error of weighted 
    covariance of Gaussians};
    
    \begin{scope}[start branch=hoejre,]
    \node[punktchain, on chain=going right] (lem3) {\Cref{lemma2}:\\ $\mathbb L_2$-error of weighted averages of Gaussian vectors};
    \end{scope}
    
    \node (prop1) [punktchain] {\Cref{proposition1}:\\ Deterministic bound involving the weighted covariance};
    \begin{scope}[start branch=hoejre,]
      \node[punktchain, on chain=going left] (lem2) {\Cref{lem:8}:\\ Assessing the Error of Initial Estimator\\ 
      (geom.\ median)};
    \end{scope}
    \begin{scope}[start branch=hoejre,]
      \node[punktchain, on chain=going right] (prop2) {\Cref{proposition2}:\\ Uniform bounds on the moments of Stochastic Terms};
    \end{scope}
  \node[punktchain, join] (thm1) {\Cref{th:1}: Bound on Expected Error};
   \draw[blue,|-,|-,->, thick,] (lem2.north) |-+(0,1em)-| (thm1.south);
  \draw[blue,|-,|-,->, thick,] (prop2.north) |-+(0,1.1em)-| (thm1.south);
  \draw[blue,|-,|-, thick,] (lem7.north) |-+(0,1em)-| (lem3.north);
  \draw[blue,-,->, thick,] 
    let \p1=(lem7.north), \p2=(lem3.north), \p3 = (prop1.north), \p4 = (prop2.north) in
    ($(\x1+6.5em,\y1+1em)$) to ($(\x1+6.5em,\y3+1em)$);
  \draw[blue,-,->, thick,] 
    let \p1=(lem7.north), \p2=(prop2.west) in
    ($(\x1+6.5em,\y2)$) to ($(\x2,\y2)$);
  \draw[blue,-,->, thick,] (lem5.east) to (lem6.west);
  \draw[blue,|-,|-, thick,] (lem6.north) |-+(0,1em)-| (lem1.north);
  \draw[blue,-, thick,] 
    let \p1=(lem6.north), \p2=(lem3.west) in
    ($(\x1+6.5em,\y1+1em)$) to ($(\x1+6.5em,\y2)$);
  \draw[blue,-,->, thick,] (lem7.east) to (lem3.west);
  \draw[blue,-,->, thick,] (lem3.west) to (lem7.east);
\end{tikzpicture}
\caption{A diagram showing the relation between different 
lemmas and propositions used in the proof of \Cref{th:1}.}
\label{fig:diagram}
\end{figure}
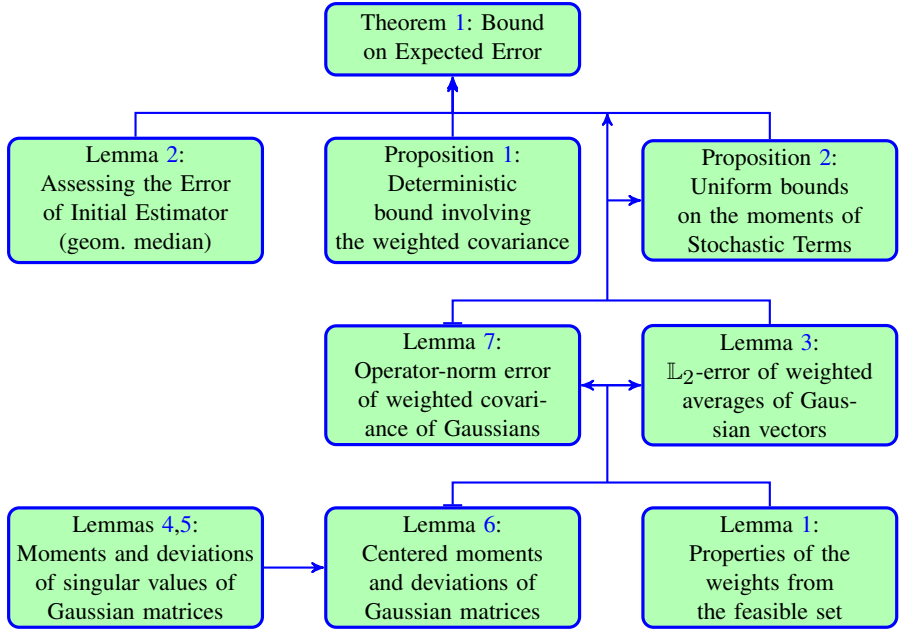    
	\subsection{Proof of \Cref{proposition1}}
	To ease notation throughout this proof, we write $\bar\varepsilon$ instead of 
	$\varepsilon_{w}$. Simple algebra yields
	\begin{align}
		\bar{\bs X}_{\bw} - \bmu^* - \bar{\bzeta}_{\bw_{\mathcal I}}  
		    &= \bar{\bs X}_{\bw_{\mathcal I}} + \bar{\bs X}_{\bw_{\mathcal I^c}} 
		        - \bmu^* - \bar{\bzeta}_{\bw_{\mathcal I}}  \\ 
			& = (1-\bar\varepsilon)\bmu^* + \bar{\bzeta}_{\bw_{\mathcal I}} + 
			    \bar{\bs X}_{\bw_{\mathcal I^c}} - \bmu^* - 
			    \bar{\bzeta}_{\bw_{\mathcal I}}\\
			& = \bar{\bs X}_{\bw_{\mathcal I^c}} - \bar\varepsilon\bmu^*\\
			& = \bar{\bs X}_{\bw_{\mathcal I^c}} - \bar\varepsilon\bar{\bs X}_{\bw} 
			    + \bar\varepsilon(\bar{\bs X}_{\bw} - \bmu^* ).
	\end{align}
	This implies that $(1-\bar\varepsilon)(\bar{\bs X}_{\bw} - \bmu^*) - 
	\bar{\bzeta}_{\bw_{\mathcal I}}  = \bar{\bs X}_{\bw_{\mathcal I^c}} 
	- \bar\varepsilon\bar{\bs X}_{\bw}$,  which is equivalent to 
	\begin{align}
			\bar{\bs X}_{\bw} - \bmu^* - \bar{\bzeta}_{\bw_{|\mathcal I}}  
			= \frac{\bar{\bs X}_{\bw_{\mathcal I^c}} - \bar\varepsilon\bar{\bs X}_{\bw}}
			{1-\bar\varepsilon}.\label{eq:p11a}
	\end{align}
	Therefore, we have 
	\begin{align}
		\big\|\bar{\bs X}_{\bw} - \bmu^* - 
		\bar{\bzeta}_{\bw_{|\mathcal I}}\big\|_2  
		& = \frac{\big\|\bar{\bs X}_{\bw_{\mathcal I^c}} - 
		\bar\varepsilon\bar{\bs X}_{\bw}\big\|_2}{1-\bar\varepsilon}\\
		& = \frac1{1-\bar\varepsilon} \bigg\|\sum_{i\in\mathcal I^c} 
		w_i(\bs X_i - \bar{\bs X}_{\bw})\bigg\|_2\\
		& = \frac1{1-\bar\varepsilon} \sup_{\bv\in\mathbb S^{p-1}} \sum_{i\in\mathcal I^c}
		w_i\,\bv^\top(\bs X_i - \bar{\bs X}_{\bw}).
			\label{eq:p11b}
	\end{align}
	Using the Cauchy-Schwarz inequality as well as the notations 
	$\mathbf M_i = (\bs X_i-\bar{\bs X}_{\bw})^{\otimes 2}$ and 
	$\mathbf M = \sum_{i=1}^n w_i\mathbf M_i$, we get
	\begin{align}
			\bigg\{\sum_{i\in\mathcal I^c} w_i\,\bv^\top(\bs X_i-\bar{\bs X}_{\bw})\bigg\}^2
			&\le \sum_{i\in\mathcal I^c} w_i\sum_{i\in\mathcal I^c} w_i|
				\bv^\top(\bs X_i-\bar{\bs X}_{\bw})|^2
			=\bar\varepsilon\, \bv^\top\sum_{i\in\mathcal I^c} w_i\mathbf{M}_i\bv\\
			&=\bar\varepsilon\,\bv^\top\sum_{i=1}^n w_i\mathbf{M}_i\bv-
			\bv^\top\sum_{i\in\mathcal I} w_i\mathbf{M}_i\bv\\
			&={\bar\varepsilon}\,\bv^\top\big(\mathbf{M} - \bSigma\big)\bv +
			{\bar\varepsilon}\bv^\top\Big(\bSigma-\sum_{i\in\mathcal I} w_i
			\mathbf{M}_i\Big)\bv\\
			&\le \bar\varepsilon\,\lambda_{\max}\big( \mathbf{M}-\bSigma\big)
			+ \bar\varepsilon\,\Big\{\bv^\top\bSigma\bv - 
			\bv^\top\sum_{i\in\mathcal I} w_i\mathbf{M}_i\bv\Big\}.
			\label{eq:p12}
	\end{align}
	Finally, for any unit vector $\bv$,
	\begin{align}
			\bv^\top\sum_{i\in\mathcal I} w_i\mathbf{M}_i \bv &= (1-\bar\varepsilon)
			\,\bv^\top\sum_{i\in\mathcal I} (\bw_{|\mathcal I})_i(\bs X_i-\bar{\bs X}_{\bw})^{\otimes 2}\bv\\
			&\ge(1-\bar\varepsilon)\,\bv^\top\sum_{i\in\mathcal I} (\bw_{|\mathcal I})_i
			(\bs X_i-\bar{\bs X}_{\bw_{|\mathcal I}})^{\otimes 2}\bv\\
			& = (1-\bar\varepsilon)\,\bv^\top\sum_{i\in\mathcal I} (\bw_{|\mathcal I})_i
			(\bzeta_i-\bar{\bzeta}_{\bw_{|\mathcal I}})^{\otimes 2}\bv\\
			&= (1-\bar\varepsilon)\,\bv^\top\sum_{i\in\mathcal I} (\bw_{|\mathcal I})_i
				\bzeta_i\bzeta_i^\top\bv-(1-\bar\varepsilon)\,
				(\bv^\top\bar{\bzeta}_{\bw_{|\mathcal I}})^2\\
			&\ge (1 -\bar\varepsilon)\bigg(\bv^\top\bSigma\bv - \lambda_{\max}
				\bigg(\sum_{i\in\mathcal I} (\bw_{|\mathcal I})_i (\bSigma-
				\bzeta_i\bzeta_i^\top)\bigg) - \|\bar{\bzeta}_{\bw_{|\mathcal I}}\|_2^2
				\bigg).
			 \label{eq:p13}
	\end{align}
	Combining \eqref{eq:p12} and \eqref{eq:p13}, we get
	\begin{align}
		\bigg\{\sum_{i\in\mathcal I^c} w_i\,\bv^\top(\bs X_i-\bmu)\bigg\}^2
				&\le \bar\varepsilon\,\lambda_{\max}\big( \mathbf{M}-\bSigma\big)
					+ \bar\varepsilon^2\,\lambda_{\max}\big(\bSigma\big) \\
				&\qquad +\bar\varepsilon(1-\bar\varepsilon)\bigg\{  \lambda_{\max}
					\bigg(\sum_{i\in\mathcal I} (\bw_{|\mathcal I})_i(\bSigma - 
					\bzeta_i\bzeta_i^\top)\bigg)  + 
					\|\bar{\bzeta}_{\bw_{|\mathcal I}}\|_2^2
					\bigg\}.\label{eq:p14}
	\end{align}
	In conjunction with \eqref{eq:p11b}, this yields 
	\begin{align}
	    \big\|\bar{\bs X}_{\bw} - \bmu^* - 
			\bar{\bzeta}_{\bw_{|\mathcal I}}\big\|_2
			& \leq \frac{\sqrt{\bar\varepsilon}}{1-\bar\varepsilon} \bigg(
			\lambda_{\max}\big( \mathbf{M}-\bSigma\big)
					+ \bar\varepsilon\,\lambda_{\max}\big(\bSigma\big) \\
			&\qquad +(1-\bar\varepsilon)\bigg\{  \lambda_{\max}
					\bigg(\sum_{i\in\mathcal I} (\bw_{|\mathcal I})_i(\bSigma - 
					\bzeta_i\bzeta_i^\top)\bigg)  + 
					\|\bar{\bzeta}_{\bw_{|\mathcal I}}\|_2^2
					\bigg\}\bigg)^{1/2}.
	\end{align}
	From this relation, using the triangle inequality and the inequality 
	$\sqrt{a_1+\ldots + a_n}\le \sqrt{a_1}+\ldots+\sqrt{a_n}$, we get 
	\begin{align}
	    \big\|\bar{\bs X}_{\bw} - \bmu^*\|_2&\le 
			\|\bar{\bzeta}_{\bw_{|\mathcal I}}\big\|_2
			+ \frac{\sqrt{\bar\varepsilon}}{1-\bar\varepsilon} 
			\lambda_{\max,+}^{1/2}\big( \mathbf{M}-\bSigma\big)
			+ \frac{{\bar\varepsilon}\|\bSigma\|_{\textup{op}}^{1/2}}
			{1-\bar\varepsilon}\\
			&\qquad + \sqrt{\frac{\bar\varepsilon}{1-\bar\varepsilon}}
			\lambda_{\max,+}^{1/2}
			\bigg(\sum_{i\in\mathcal I} (\bw_{|\mathcal I})_i(\bSigma - 
			\bzeta_i\bzeta_i^\top)\bigg)  + \sqrt{\frac{\bar\varepsilon}{1-
			\bar\varepsilon}}\|\bar{\bzeta}_{\bw_{|\mathcal I}}\|_2
	\end{align}
	and, after rearranging the terms, the claim of the proposition follows.

	\subsection{Proof of \Cref{lemma0}} 
	Claim i) of this lemma is straightforward. For ii), we use the fact
	that the compact convex polytope $\mathcal W_{n,\ell}$ is the
	convex hull of its extreme points. The fact that each uniform weight 
	vector $\bu^{J}$ is an extreme point of $\mathcal
	W_{n,\ell}$ is easy to check. Indeed, if for two points 
	$\bw$ and $\bw'$ from $\mathcal W_{n,\ell}$ we have
	$\bu^{J} = 0.5(\bw+\bw')$, then we necessarily have
	$\bw_{J^c} = \bw'_{J^c} = 0$. Therefore, 
	for any $j\in J$,
	\begin{align}
			\frac1{\ell} \ge w_j  = 1 - \sum_{i\in J\setminus\{j\}} w_i \ge 
			1-\big(\ell-1)\times \frac1{\ell}=\frac1{\ell}.
	\end{align}
	This implies that $w_j=\frac1{\ell}\mathds 1(j\in J)$. Hence, 
	$\bw = \bu^{J}$ and the same is true for
	$\bw'$. Hence, $\bu^{J}$ is an extreme 
	point. Let us prove now that all the extreme points of $\mathcal
	W_{n,\ell}$ are of the form $\bu^{J}$ with
	$|J| = \ell$. Let $\bw\in \mathcal W_{n,\ell}$ be such that one 
	of its coordinates is strictly positive and strictly smaller 
	than $1/\ell$. Without loss of generality, we assume that the 
	two smallest nonzero entries of $\bw$ are $w_1$ and $w_2$. 
	We have $0<w_1<1/\ell$ and $0<w_2<1/\ell$. Set
	$\rho = w_1\wedge w_2 \wedge \{1/\ell - w_1\}\wedge 
	\{1/\ell - w_2\}$. For $\bw^+ = \bw + (\rho,-\rho,0,\ldots,0)$ 
	and $\bw^- = \bw - (\rho,-\rho, 0,\ldots,0)$, we have 
	$\bw^+,\bw^-\in\mathcal W_{n,\ell}$
	and $\bw = 0.5(\bw^++\bw^-)$. Therefore, $\bw$ is not 
	an extreme point of $\mathcal W_{n,\ell}$. This completes the
	proof of ii). 

	Claim ii) implies that $\mathcal W_{n,\ell} = \{\sum_{|J|=\ell} \alpha_J \bu^J:
	\bs\alpha \in \bs\Delta^{K-1}\}$ with $K={n\choose \ell}$. Hence, 
	\begin{align}
			\sup_{\bw\in \mathcal W_{n,\ell}} G(\bw) &= 
					\sup_{\bs\alpha\in\bs\Delta^{K-1}} G\bigg(\sum_{|J|=\ell} 
					\alpha_J \bu^J\bigg) \le \sup_{\bs\alpha\in\bs\Delta^{K-1}} 
					\sum_{|J|=\ell} \alpha_J G\big(\bu^J\big)
					\le \max_{|J| =\ell} G(\bu^J)
	\end{align}
	and claim iii) follows. 

	To prove iv), we check that 
	\begin{align}
		\|\bw_{ I}\|_1 = \sum_{i\in  I} w_i  
		    = 1 - \sum_{i\not\in  I} w_i
		    \ge 1-\frac{| I^c|}{\ell} \ge 1- \frac{n-\ell'}{\ell} 
			= \frac{\ell+\ell'-n}{\ell}.
	\end{align}
	This readily yields $(\bw_{|I})_i \le 1/(\ell+\ell'-n)$, which 
	leads to the claim of item iv).

	\subsection{Moments of suprema over $\mathcal W_{n,\ell}$ of 
	weighted averages of Gaussian vectors}

	We recall that $\bxi_i$'s, for
	$i=1,\ldots,n$ are i.i.d. Gaussian vectors with zero mean and identity covariance
	matrix, and $\bzeta_i = \bSigma^{1/2}\bxi_i$. In addition, the covariance 
	matrix $\bSigma$ satisfies $\|\bSigma\|_{\textup{op}}=1$. 

	\begin{lemma}\label{lemma2}
	Let $p\ge 1$, $n\ge 1$, $m\in[2,0.562n]$ and $o\in [0, m]$ be four positive integers. 
	It holds that
	\begin{align}
			\bfE^{1/2}\Bigg[\sup_{\substack{\bw\in \mathcal W_{n,n-m+o} \\ |I| \ge n-o}}
			{\|\bar{\bzeta}_{\bw_{|I}}\|}_2^2\Bigg]  \le {\sqrt{\rSigma/n}}\,
			\bigg(1+7\sqrt{\frac{m}{2n}}\bigg) + \frac{4.6m}{n}\sqrt{\log(\nicefrac{ne}{m})}.
	\end{align}
	\end{lemma}
	\begin{proof}[Proof of \Cref{lemma2}] 
	We have
	\begin{align}
			\sup_{\bw\in \mathcal W_{n,n-m+o}}
			\|\bar{\bzeta}_{\bw_{|I}}\|_2 
					& \stackrel{(a)}{\le} 
							\sup_{\bw\in \mathcal W_{n,n-m}}\|\bar{\bzeta}_{\bw}\|_2
					\stackrel{(b)}{\le}
							\max_{|J| = n-m}
							\big\|\bar{\bzeta}_{\bu^{J}}\big\|_2,
			\label{eq:L21}
	\end{align}
	where $(a)$ follows from claim iv)  of \Cref{lemma0} and $(b)$ is a direct 
	consequence of claim iii) of \Cref{lemma0}. Thus, we get
	\begin{align}
			\sup_{\bw\in \mathcal W_{n,n-m+o}}
			\|\bar{\bzeta}_{\bw_{|I}}\|_2
			&\le  \max_{|J| = n-m} 
							{\big\|\bar{\bzeta}_{\bu^{J}}\big\|}_2
							= \frac1{n-m}\max_{|J| = n-m}  
							\bigg\|\sum_{i\in J} \bzeta_i\bigg\|_2\\
			&\le \frac1{n-m}
							\bigg\|\sum_{i=1}^n \bzeta_i\bigg\|_2 + 
							\frac1{n-m}\max_{|\bar J| = m}  
							\bigg\|\sum_{i\in \bar J} \bzeta_i\bigg\|_2.\label{eq:L22}
	\end{align}
	On the one hand, one readily checks that
	\begin{align}\label{eq:L22a}
			\bfE\bigg[\bigg\|\sum_{i=1}^n \bzeta_i\bigg\|_2^2\bigg] = n\rSigma.
	\end{align}
	On the other hand, it is clear that for every $\bar J$ of cardinality $m$, 
	the random variable $\big\|\sum_{i\in \bar J} \bzeta_i\big\|^2_2$ has the 
	same distribution as $m\sum_{j=1}^p \lambda_j(\bSigma)\xi_j^2$, where
	$\xi_1,\ldots,\xi_p$ are i.i.d. standard Gaussian. Therefore, by the union bound, 
	for every $t\ge 0$, we have
	\begin{align}
			\bfP\bigg(\max_{|\bar J| = m}  \bigg\|\sum_{i\in \bar  J} 
			\bzeta_i\bigg\|^2_2>m \big(2\rSigma + t\big)\bigg) 
			&\le {{n}\choose{m}} \bfP\bigg(\bigg\|\sum_{i=1}^{m} 
			\bzeta_i\bigg\|_2^2 > m \big(2\rSigma + t\big)\bigg)\\
			&= {{n}\choose{m}} \bfP\bigg( \sum_{j=1}^p \lambda_j(\bSigma)\xi_j^2
			> 2\rSigma + t\bigg) 
			\le {{n}\choose{m}} e^{-t/3},
	\end{align}
	where the last line follows from a well-known bound on the tails of the 
	generalized $\chi^2$-distribution, see for instance \citep[Lemma 8]{comminges2012}.

	Therefore, setting $Z =\frac1m\max_{|J| = m}  \big\|\sum_{i\in J} 
	\bzeta_i\big\|^2_2 - 2\rSigma$ and using the well-known 
	identity $\bfE[Z] \le \bfE[Z_+] = \int_0^\infty
	\bfP(Z\ge t)\,dt$, we get 
	\begin{align}
			\bfE\bigg[\max_{|\bar J| = m} \frac1{m} \bigg\|
			\sum_{i\in\bar J} \bzeta_i\bigg\|^2_2\bigg] 
			&\le 2\rSigma + \int_0^\infty 
					1\wedge {{n}\choose{m}} e^{-t/3}\,dt\\
			&= 2\rSigma + 3 \log {{n}\choose{m}}
					+ {{n}\choose{m}}\int_{3\log {{n}\choose{m}}}^\infty 
					e^{-t/3}\,dt\\
			& \le 2\rSigma + 3m\log(\nicefrac{ne}{m}) + 3\\
			& \le 2\rSigma + 3.9m\log(\nicefrac{ne}{m}), \label{eq:L23}
	\end{align}
	where the last two steps follow from the inequality $\log {n\choose{m}}
	\le m\log (ne/m)$ and the fact that $m\ge 2$, $m\log(\nicefrac{ne}{m}) \ge 
	n \inf_{x\in [2/n,1]} x(1-\log x) \ge 2(1-\log (2/n))\ge 10/3$. Combining 
	\eqref{eq:L22}, \eqref{eq:L22a} and \eqref{eq:L23}, we arrive at 
	\begin{align}
			\bigg(\bfE\bigg[\sup_{\substack{\bw\in \mathcal W_{n,n-m+o} 
					\\ |I| \ge n-o}}
			{\|\bar{\bzeta}_{\bw_{|I}}\|}_2^2\bigg]\bigg)^{1/2} 
			&\le \frac{\sqrt{n\rSigma}}{n-m} + \frac{\sqrt{m}}{
			n-m}\times\big(2\rSigma + 3.9m 
			\log(\nicefrac{ne}{m})\big)^{1/2}\\
			&\le \sqrt{\rSigma/n}\Big(1+ \frac{m}{n-m} + \frac{\sqrt{2mn}}{n-m}\Big) 
			+ \frac{4.6m}{n}\sqrt{\log(\nicefrac{ne}{m})}.
	\end{align}
	Finally, note that for $\alpha = m/n\le 0.562$, we have
	\begin{align}
	\frac{m}{n-m} + \frac{\sqrt{2mn}}{n-m} &= \sqrt{\frac{m}{2n}}\bigg(
			\frac{\sqrt{2mn}}{n-m} + \frac{2n}{n-m}\bigg)\\
			& = \sqrt{\frac{m}{2n}}\bigg(
			\frac{\sqrt{2\alpha}}{1-\alpha} + \frac{2}{1-\alpha}\bigg)
			\le 7\sqrt{\frac{m}{2n}}.
	\end{align}
	This completes the proof of the lemma.
	\end{proof}


	\subsection{Moments and deviations of singular values of Gaussian matrices}

	Let $\bzeta_1,\ldots,\bzeta_n$ be i.i.d. random vectors drawn from 
	$\mathcal N_p(0,\bSigma)$ distribution, where $\bSigma$ is a 
	$p\times p$ covariance matrix. We denote by $\bzeta_{1:n}$ the 
	$p\times n$ random matrix obtained by concatenating the vectors 
	$\bzeta_i$. Recall that $s_{\min}(\bzeta_{1:n}) = 
	\lambda_{\min}^{1/2}(\bzeta_{1:n}\bzeta_{1:n}^\top)$ 
	and $s_{\max}(\bzeta_{1:n}) = \lambda_{\max}^{1/2}
	(\bzeta_{1:n}\bzeta_{1:n}^\top)$ are the smallest and the 
	largest singular values of the matrix $\bzeta_{1:n}$. 

	\begin{lemma}[\cite{vershynin_2012}, Theorem 5.32 and 
	Corollary 5.35]\label{lem:vershynin}
	Let $\lambda_{\max}(\bSigma)=1$. For every $t>0$ and for every 
	pair of positive integers $n$ and $p$, 
	we have
	\begin{align}
		\bfE\big[s_{\max}(\bzeta_{1:n})\big] & \le \sqrt{n} + \sqrt{\rSigma},\quad
		&\bfP\big(s_{\max}(\bzeta_{1:n})  \ge \sqrt{n} + \sqrt{\rSigma} + 
		t\big)  \le e^{-t^2/2}.
	\end{align}
	If, in addition, $\bSigma$ is the identity matrix, then
	\begin{align}
		\bfE\big[s_{\min}(\bzeta_{1:n})\big] & \ge {\big(\sqrt{n}
		-\sqrt{p}\big)}_+,\quad
		&\bfP\big(s_{\min}(\bzeta_{1:n})  \le \sqrt{n} - \sqrt{p} 
		- t\big)  \le e^{-t^2/2}.
	\end{align}
	\end{lemma}
	The corresponding results in \citep{vershynin_2012} treat only 
	the case of identity covariance matrix $\bSigma=\mathbf I_p$, 
	however the proof presented therein carries with almost no 
	change over the case of arbitrary covariance matrix. These 
	bounds allow us to establish the following inequalities. 

	\begin{lemma} \label{lem:smax}
	For a subset $\bar J$ of $\{1,\ldots,n\}$, we denote by $\bzeta_{\bar J}$ the 
	$p\times |\bar J|$ matrix obtained by concatenating the vectors 
	$\{\bzeta_i: i\in\bar J\}$. Let the covariance matrix $\bSigma$ be such that 
	$\lambda_{\max}(\bSigma) =1$. For every pair of integers  $n,p\ge 1$ and for 
	every integer $m\in[1,n]$, we have
	\begin{align}
		\bfE\big[s_{\max}^2(\bzeta_{1:n})\big] &\le  \big(\sqrt{\rSigma} + 
		\sqrt{n}\big)^2 +4,\label{esp:1}\\
		\bfE\big[(s_{\max}(\bzeta_{1:n})^2-n)_+\big]			
			&\le 6\sqrt{n\rSigma} + 4\rSigma,\qquad \forall n\ge 8,\label{esp:1b}\\
		\bfE\big[(n-s_{\min}(\bxi_{1:n})^2)_+\big]			
			&\le 6\sqrt{np},\qquad \forall n\ge 8,\label{esp:1c}\\
		\bfE\Big[\max_{|\bar J| = m} s_{\max}^2(\bzeta_{\bar J})\Big] &\le
				\big(\sqrt{\rSigma} + \sqrt{m} +  1.81\sqrt{m \log(\nicefrac{ne}{m})}\big)^2
				+4\label{esp:2}.
	\end{align}	
	\end{lemma} 
	\begin{proof}
	The bias-variance decomposition, in conjunction with \Cref{lem:vershynin}, yields
	\begin{align}
		\bfE\big[s_{\max}^2(\bzeta_{1:n})\big] &=  
		\textbf{Var}\big(s_{\max}(\bzeta_{1:n})\big)
			+\bfE\big[s_{\max}(\bzeta_{1:n})\big]^2\\
		&\le  \textbf{Var}\big(s_{\max}(\bzeta_{1:n})\big)+
			\big(\sqrt{n} + \sqrt{\rSigma}\big)^2.
	\end{align}
	Applying the well-known fact $\bfE[Z^2] = \int_0^\infty \bfP(Z^2\ge t)\,dt$ 
	to the random variable $Z= s_{\max}(\bzeta_{1:n}) - \bfE[s_{\max}(\bzeta_{1:n})]$ 
	and using the Gaussian concentration inequality, we get
	\begin{align}
		\textbf{Var}\big(s_{\max}(\bzeta_{1:n})\big) &= \int_0^\infty \mathbf{P}
		\Big(\big|s_{\max}(\bzeta_{1:n}) - \bfE[s_{\max}(\bzeta_{1:n})]
		\big|\ge \sqrt{t}\Big)\,dt \le 2\int_0^\infty e^{-t/2}\,dt = 4.
	\end{align}
	This completes the proof of \eqref{esp:1}. 

	For every random variable $Z$ and every constant $a>0$, we have
	\begin{align}
		\bfE\big[((a + Z)^2 - n)_+\big]
			&= \bfE\big[(a^2+2aZ +Z^2-n)_+\big] \\
			&\le \bfE\big[(a^2+2aZ-n)_+\big]  + \bfE[Z^2]\\
			&\le |a^2-n| + 2a\bfE[Z_+]  + \bfE[Z^2].
	\end{align}
	Taking $a = \bfE[s_{\max}(\bzeta_{1:n})]$ and $Z = s_{\max}(\bzeta_{1:n})
	-\bfE[s_{\max}(\bzeta_{1:n})]$, we get
	\begin{align}
	\bfE\big[(s_{\max}(\bzeta_{1:n})^2-n)_+\big]
			&\le |(\sqrt{n} + \sqrt{\rSigma})^2-n| + (\sqrt{n} + \sqrt{\rSigma})
			\sqrt{2\pi}  + 4\\
			&= 2\sqrt{n\rSigma} +\rSigma + (\sqrt{n} + \sqrt{\rSigma})
			\sqrt{2\pi}  + 4\\
			&\le 6\sqrt{n\rSigma} + 4\rSigma,\qquad \forall n\ge 8.
	\end{align}
	Similarly, we have
	\begin{align}
		\bfE\big[( n - (a + Z)^2 )_+\big]
			&= \bfE\big[(n- a^2 - 2aZ  - Z^2)_+\big] 
			\le (n- a^2)_+ +2a\bfE\big[ Z_-\big].
	\end{align}
	Taking $a = \bfE[s_{\min}(\bxi_{1:n})]$ and $Z = s_{\min}(\bxi_{1:n})
	-\bfE[s_{\min}(\bxi_{1:n})]$, we get
	\begin{align}
		\bfE\big[(n-s_{\min}(\bxi_{1:n})^2)_+\big] &\le 
		(n- \bfE[s_{\min}(\bxi_{1:n})]^2)_+ +\bfE[s_{\min}(\bxi_{1:n})]
		\sqrt{2\pi}\\
		&\le (n- (\sqrt{n} - \sqrt{p})^2)_++
		(\sqrt{n} + \sqrt{p})\sqrt{2\pi}\\
		&\le 2\sqrt{np} +(\sqrt{n} + \sqrt{p})\sqrt{2\pi}\le 
		6\sqrt{np}.
	\end{align}
	Thus, we have checked \eqref{esp:1b} and \eqref{esp:1c}. 
	
	In view of \Cref{lem:vershynin}, for every $t>0$,
	\begin{align}
		\bfP\Big(s_{\max}\big(\bzeta_{\bar J}\big) 	
		< \sqrt{|\vbox to 7pt{\hbox{$\bar J$}}|} + \sqrt{\rSigma}+ t\Big)
		\ge 1-e^{-t^2/2}.\label{eq:L33b}
	\end{align}
	Let $I_0\subset [n]$ be any set of cardinality $m$. Combining 
	the relation $\bfE[Z] = \int_0^\infty \bfP(Z\ge s)\,ds$, the union bound, 
	and \eqref{eq:L33b}, we arrive at
	\begin{align}
	\bfE\Big[\max_{|\bar J| = m} s_{\max}(\bzeta_{\bar J})\Big]
			&\le \sqrt{m}+\sqrt{\rSigma}+ \bfE\Big[\max_{|\bar J| = m} 
			\big(s_{\max}(\bzeta_{\bar J}) - \sqrt{m} - \sqrt{\rSigma}\big)_+\Big]\\
			& = \sqrt{m} + \sqrt{\rSigma}+ \int_0^\infty \bfP\Big(\max_{|\bar J| = m} 
			s_{\max}(\bzeta_{\bar J})\ge \sqrt{m} +\sqrt{\rSigma} + t \Big)\,dt\\
			&\le \sqrt{m}+\sqrt{\rSigma}+ \int_0^\infty 1\bigwedge {n\choose{m}}
			\bfP\Big(s_{\max}(\bzeta_{I_0}) \ge \sqrt{m} +\sqrt{\rSigma} + t \Big)\,dt\\		
			&\le \sqrt{m} + \sqrt{\rSigma} +  
			\int_{0}^\infty \bigg\{1\bigwedge {n \choose {m}}e^{-t^2/2} 
			\bigg\}\,dt.\label{eq:L35}
	\end{align}	
	From now on, we assume that $m\in [2,n-1]$, which implies that $ {n \choose {m}} \ge n$.
	Let $t_0$ be the value of $t$ for which the two terms in the last minimum
	coincide, that is
	\begin{align}
			{n \choose{m}}e^{-t_0^2/2} = 1
			\quad&\Longleftrightarrow\quad
			t_0^2 =  2\log{n \choose{m}}
			\quad\Longrightarrow\quad
			\begin{cases}
			t_0^2 	\le 2m \log(\nicefrac{ne}{m}),\\
			t_0^2 	\ge 2 \log(\nicefrac{n(n-1)}{2}),
			\end{cases}
			\label{eq:L36}
	\end{align}
	We have, for $m\ge 1$,
	\begin{align}
			\int_{0}^\infty \bigg\{1\bigwedge {n \choose{m}}e^{-t^2/2} \bigg\}\,dt
			& = t_0 + {n \choose{m}}\int_{t_0}^\infty e^{-t^2/2}\,dt
			 \le {t_0} +  {n \choose{m}} \frac{e^{-t^2_0/2}}{t_0}\\
			& \le {t_0} +  \frac{1}{t_0}
			\le 1.81\sqrt{m \log(\nicefrac{ne}{m})}.\label{eq:L37a'}
	\end{align}
	Inequalities \eqref{eq:L35} and \eqref{eq:L37a'} yield
	\begin{align}
		\bfE\Big[\max_{|\bar J| = m}s_{\max}\big(\bzeta_{\bar J}\big)\Big]
			&\le \sqrt{m} + \sqrt{\rSigma} +  1.81\sqrt{m \log(\nicefrac{ne}{m})}.		
	\end{align}
	On the other hand, the mapping $\bzeta_{1:n}\mapsto F(\bzeta_{1:n}) : = 
	\max_{|\bar J| = m}s_{\max}\big(\bzeta_{\bar J}\big)$ being 1-Lipschitz, the 
	Gaussian concentration inequality leads to
	\begin{align}
		\textbf{Var} \Big(\max_{|\bar J| = m}s_{\max}\big(\bzeta_{\bar J}\big)\Big)
			& = \int_0^\infty \textbf{P} \Big(\big|F(\bzeta_{1:n})-\bfE
			[F(\bzeta_{1:n})]\big|\ge\sqrt{t}\Big)\,dt\le 2\int_0^\infty e^{-t/2}\,dt 
			= 4.
	\end{align}
	This completes the proof of the lemma.
	\end{proof}

	\begin{lemma} \label{lem:smaxbis}
	There is a constant $A_1>0$ such that for every pair of integers  $n\ge 8$ 
	and $p\ge 1$ and for every covariance matrix $\bSigma$ such that 
	$\lambda_{\max}(\bSigma) = 1$, we have
	\begin{align}
		\bfE\Big[\|\bzeta_{1:n}\bzeta_{1:n}^\top-n\bSigma\|_{\textup{op}}\Big] 
		&\le  A_1\big(\sqrt{n} 
			+ \sqrt{\rSigma}\big)\sqrt{\rSigma},\label{espbis:1}\\
		\bfE\Big[\lambda_{\max,+}(\bzeta_{1:n}\bzeta_{1:n}^\top-n\bSigma)\Big] 
		&\le  6\sqrt{np} + 4p,\label{espbis:2}\\
		\bfE\Big[\lambda_{\max,+}(n\bSigma - \bzeta_{1:n}\bzeta_{1:n}^\top)\Big] 
		&\le  6\sqrt{np},\label{espbis:3}
	\end{align}	
	where the last inequality is valid under the additional assumption $p\le n$. 
	Furthermore, there is a constant $A_2>0$ such that 
	\begin{align}
		\bfP\bigg(\|\bzeta_{1:n}\bzeta_{1:n}^\top-n\bSigma\|_{\textup{op}} -
		\bfE\Big[\|\bzeta_{1:n}\bzeta_{1:n}^\top-n\bSigma\|_{\textup{op}}\Big] 
		\ge A_2\Big(\sqrt{tn} + \sqrt{t\rSigma}+t\Big)\bigg)\le e^{-t},\quad \forall t\ge 1.
	\end{align}
	\end{lemma} 
	\begin{proof}
	Inequality \eqref{espbis:1} and the last claim of the lemma are respectively 
	Theorems 4 and 5 in \citep{KoltchLounici}. Let us prove the two other claims. 
	Since $\bzeta_i = \bSigma^{1/2}\bxi_i$ where $\bxi_i$'s are i.i.d. $\mathcal N(0,\bfI_p)$, 
	we have
	\begin{align}
		\lambda_{\max,+}(\bzeta_{1:n}\bzeta_{1:n}^\top-n\bSigma)
		&\le \lambda_{\max,+}(\bxi_{1:n}\bxi_{1:n}^\top-n\bfI_p)
		 = \big(s_{\max}(\bxi_{1:n})^2 - n\big)_+.
	\end{align}
	Inequality \eqref{espbis:2} now follows from \eqref{esp:1b}	applied in the
	case of an identity covariance matrix so that $\rSigma = p$. Similarly, 
	\eqref{espbis:3}  follows from \eqref{esp:1c} using the same argument.
	\end{proof}

	\subsection{Moments of suprema over $\mathcal W_{n,\ell}$ of weighted 
	centered Wishart matrices}
	
	\begin{lemma}\label{lemma3}
	Let $p\ge 2$, $n\ge 4\vee p$, $m\in[2,0.6n]$ and $o\le m$ be four integers. 
	It holds that
	\begin{align}
			\bfE&\Bigg[\sup_{\substack{\bw\in \mathcal W_{n,n-m+o} \\ |I| \ge n-o}} 
			\lambda_{\max,+}\Big(\bSigma - \sum_{i=1}^n (\bw_{|I})_i\bzeta_i\bzeta_i^\top\Big
			)\Bigg] \le 25\sqrt{p/n} + 33(m/n) \log({n}/{m}).
	\end{align}
	Furthermore, for any $p\ge 1$, $n\ge 1$, $m\in[2,0.6n]$ and $o\le m$,
	\begin{align}
			\bfE&\bigg[\sup_{\substack{\bw\in \mathcal W_{n,n-m+o} \\ |I| \ge n-o}}  
			\Big\|\bSigma - \sum_{i=1}^n (\bw_{|I})_i\bzeta_i\bzeta_i^\top
			\Big\|_{\textup{op}}\bigg] \le (5.1A_1 + 2.5 A_2)\sqrt{\frac{\rSigma}{n}} 
			+ 7.5A_2 \frac{m \log(\nicefrac{ne}{m})}{n},			
	\end{align}
	where $A_1$ and $A_2$ are the same constants as in \Cref{lem:smaxbis}.
	\end{lemma}

	\begin{proof}[Proof of \Cref{lemma3}]
	For every $I\subset [n]$ such that $|I| \ge n-o$, we have
	\begin{align}
			\sup_{\bw\in \mathcal W_{n,n-m+o}} 
			\lambda_{\max,+}\Big(\bSigma - \sum_{i=1}^n (\bw_{|I})_i
			\bzeta_i\bzeta_i^\top\Big)
			&\stackrel{(a)}{\le}
			\sup_{\bw\in \mathcal W_{n,n-m}} \lambda_{\max,+}\Big(\bSigma
			-\sum_{i=1}^nw_i \bzeta_i\bzeta_i^\top \Big)\\
			&\stackrel{(b)}{\le}\sup_{|J| = n-m}\lambda_{\max,+}\Big( \sum_{i=1}^n
			u^J_i(\bSigma -\bzeta_i\bzeta_i^\top)\Big).\label{eq:L31}
	\end{align}
	In the above inequalities, $(a)$ follows from claim iv)
	of \Cref{lemma0}, while  $(b)$ is a direct consequence
	of \Cref{lemma0}, claim iii).  Thus, we can infer that for every 
	$\bw\in \mathcal W_{n,n-m+o}$,
	\begin{align}
		(n-m) \lambda_{\max,+}&\Big(\bSigma - \sum_{i=1}^n (\bw_{|I})_i
			\bzeta_i\bzeta_i^\top\Big)\le \max_{|J| = n-m}\lambda_{\max,+}
			\Big(\sum_{i\in J}(\bSigma-\bzeta_i\bzeta_i^\top)\Big)\\
		&\le  \lambda_{\max,+}\bigg(\sum_{i=1}^n(\bSigma-\bzeta_i\bzeta_i^\top
			)\bigg) + \max_{|\bar J| = m} \lambda_{\max,+}\bigg(\sum_{i\in \bar J}(
			\bzeta_i\bzeta_i^\top - \bSigma)\bigg)\\
		&\le \lambda_{\max,+}\bigg(\sum_{i=1}^n(\bSigma-\bzeta_i\bzeta_i^\top
					)\bigg) + \max_{|\bar J| = m} \bigg\{\lambda_{\max}
			\bigg(\sum_{i\in \bar J}\bxi_i\bxi_i^\top \bigg)-m\bigg\}_+,
			\label{eq:L32}
	\end{align}
	where in the last line we have used the fact that for any symmetric 
	matrix $\mathbf B$, we have $\lambda_{\max,+}(\bSigma^{1/2}\mathbf B
	\bSigma^{1/2}) \le \|\bSigma^{1/2}\|^2_{\textup{op}}\lambda_{\max,+}
	(\mathbf B)= \lambda_{\max,+}(\mathbf B)$. To analyze the last term 
	of the previous display, we note that 
	\begin{align}
		\lambda_{\max}\bigg(\sum_{i\in \bar J}\bxi_i\bxi_i^\top	\bigg)
			&= \lambda_{\max}\big(\bxi_{\bar J}\bxi_{\bar J}^\top\big) = 
			s_{\max}^2\big(\bxi_{\bar J}\big).\label{eq:L33}
	\end{align}
	On the one hand, inequality \eqref{esp:2} of \Cref{lem:smax} yields
	\begin{align}
		\bfE\bigg[\max_{|\bar J| = m} \bigg\{\lambda_{\max}
			\bigg(\sum_{i\in \bar J}\bxi_i\bxi_i^\top \bigg)-m\bigg\}_+\bigg]
			&\le \bfE\Big[\max_{|\bar J| = m}s_{\max}^2\big(\bxi_{\bar J}\big)\Big]\\
			&\le \big(\sqrt{p} + \sqrt{m} +  1.81\sqrt{m \log(\nicefrac{ne}{m})}
				\big)^2+4\\
			&\le p+ 10m\log(\nicefrac{ne}{m}) + 5.62\sqrt{pm\log(\nicefrac{ne}{m})}\\
			&\le 4p + 13m\log(\nicefrac{ne}{m}).
		\label{eq:L37c}
	\end{align}
	On the other hand, for $n\ge p$, according to \eqref{espbis:3}, 
	\begin{align}
		\lambda_{\max,+}\bigg(\sum_{i=1}^n(\bSigma-\bzeta_i\bzeta_i^\top)\bigg)
			\le 6\sqrt{np}.\label{eq:L38}
	\end{align}
	Combining \eqref{eq:L32}, \eqref{eq:L37c} and  \eqref{eq:L38}, we arrive at 
	\begin{align}
		\bfE\bigg[\sup_{\substack{\bw\in \mathcal W_{n,n-m+o} \\ |I| \ge n-o}} 
		\lambda_{\max,+}\Big(\bSigma - \sum_{i=1}^n
			(\bw_{|I})_i\bxi_i\bxi_i^\top\Big)\bigg]
		&\le \frac{6\sqrt{np} + 4p+  13m \log(\nicefrac{ne}{m})}{n-m}\\
		&\le \frac{10\sqrt{np} +  13m \log(\nicefrac{ne}{m})}{0.4n}\\
		&\le 25\sqrt{\nicefrac{p}{n}}+33(\nicefrac{m}{n})\log(\nicefrac{ne}{m}),
	\end{align}	
	and the first claim of the lemma follows.


	To prove the last claim, we repeat the arguments in \eqref{eq:L32} to check that
	for every weight vector $\bw\in \mathcal W_{n,n-m+o}$,
	\begin{align}
			(n-m) \Big\|
			\bSigma - \sum_{i=1}^n (\bw_{|I})_i\bzeta_i\bzeta_i^\top\Big\|_{\textup{op}}
			&\le  \bigg\|\sum_{i=1}^n(\bSigma-\bzeta_i\bzeta_i^\top)\bigg\|_{\textup{op}} + 
				\max_{|\bar J| = m} \bigg\|\sum_{i\in \bar J}(\bzeta_i\bzeta_i^\top - 
				\bSigma)\bigg\|_{\textup{op}}\\
			&=\big\|\bzeta_{1:n}\bzeta_{1:n}^\top-n\bSigma\big\|_{\textup{op}}
				+\max_{|\bar J| = m} \big\|\bzeta_{\bar J}\bzeta_{\bar J}^\top - 
				m\bSigma\big\|_{\textup{op}}.\label{eq:L81}
	\end{align}
	In view of \Cref{lem:smaxbis}, for every $t\ge 1$, we have 
	\begin{align}
		\bfP\bigg(\|\bzeta_{\bar J}\bzeta_{\bar J}^\top-m\bSigma\|_{\textup{op}} 	
		\ge 2A_1\sqrt{m\rSigma} + A_2\Big(\sqrt{t m} + \sqrt{t \rSigma}+ 
		t\Big)\bigg)\le e^{-t}.
	\end{align}
	Since $t\ge 1$ and $m\ge 2$, the last inequality implies
	\begin{align}
		\bfP\bigg(\|\bzeta_{\bar J}\bzeta_{\bar J}^\top-m\bSigma\|_{\textup{op}} 	
		\ge 2A_1\sqrt{m\rSigma} + A_2\big(m+\rSigma+1.5t\big)\bigg)\le e^{-t}.
	\end{align}
	To ease notation, let us set ${\sf a} =2A_1\sqrt{m\rSigma} + A_2(m +\rSigma)$ 
	and $\mathsf b = 1.5A_2$. Using the union bound, we 
	arrive at
	\begin{align}
		\bfE\bigg[\max_{|\bar J| = m} \big\|\bzeta_{\bar J}\bzeta_{\bar J}^\top - 
		m\bSigma\big\|_{\textup{op}} \bigg] 
		&= \int_0^\infty  \bfP\bigg(\max_{|\bar J| = m} \big\|\bzeta_{\bar J}
		\bzeta_{\bar J}^\top - m\bSigma\big\|_{\textup{op}} \ge s\bigg)\,ds\\
		&= \mathsf a + \mathsf b\int_0^\infty  \bfP\bigg(\max_{|\bar J| = m} 
		\big\|\bzeta_{\bar J}\bzeta_{\bar J}^\top - m\bSigma\big\|_{\textup{op}} 
		\ge \mathsf a + \mathsf b t\bigg)\,dt\\
		&\le \mathsf a + \mathsf b\int_0^\infty  1\wedge {n\choose m}\bfP
		\bigg(\big\|\bzeta_{1:m}\bzeta_{1:m}^\top - m\bSigma\big\|_{\textup{op}} 
		\ge \mathsf a + \mathsf b t\bigg)\,dt\\
		&\le \mathsf a +  \mathsf b\int_0^\infty  1\wedge {n\choose m}e^{-t}\,dt.
	\end{align}
	Splitting the last integral integral into two parts, corresponding to the intervals
	$[0,\log{n\choose m}]$ and $[\log {n\choose m},+\infty)$, we obtain
	\begin{align}
		\bfE\bigg[\max_{|\bar J| = m} \big\|\bzeta_{\bar J}\bzeta_{\bar J}^\top - 
		m\bSigma\big\|_{\textup{op}} \bigg] 
		&\le \mathsf a +  \mathsf b \log {n\choose m} + \mathsf b\\
		&\le 2A_1\sqrt{m\rSigma} + A_2\rSigma + 3A_2 m \log(\nicefrac{ne}{m}),
	\end{align}
	where in the last line we used that 
	$$
	1.5 + 1.5\log^{-1}{n\choose m} + m\log^{-1}{n\choose m}
	\le 3,\qquad \forall m\in[2,n-1],\ 
	\forall n\ge 4.
	$$
	Combining these bounds with \eqref{eq:L81}, we arrive at
	\begin{align}
			\bfE\bigg[\sup_{\mathcal W_{n,n-m+o}}\limits \Big\|\bSigma -& 
			\sum_{i=1}^n (\bw_{|I})_i\bzeta_i\bzeta_i^\top\Big\|_{\textup{op}}\bigg]\\
			&\le  \frac{A_1\sqrt{\rSigma}(\sqrt{n} + 2\sqrt{m}) + (A_1+A_2)\rSigma 
			+ 3A_2 m \log(\nicefrac{ne}{m})}{n-m}\\
			&\le (5.1A_1 + 2.5 A_2)\sqrt{\frac{\rSigma}{n}} +
			+ 7.5A_2 (m/n) \log(\nicefrac{ne}{m}).\label{eq:L85}
	\end{align}
	This completes the proof.
	\end{proof}


	\subsection{Proof of \Cref{proposition2}}
	Throughout this proof,  $\sup_{\bw,I}$ stands for the supremum 
	over all $\bw\in\mathcal W_n(\varepsilon)$ and over all
	$I\subset\{1,\ldots,n\}$ of cardinality larger than or equal to
	$n(1-\varepsilon)$. We recall that, $\Xi = \sup_{\bw,I} R(\bzeta,
	\bw,I)$ where for any subset $I$ of$\{1,\ldots,n\}$,
	\begin{align} 
		R(\bzeta,\bw, I) 
			&= 2\varepsilon_{w}+\sqrt{2\varepsilon_{w}}\,\lambda_{\max,+}^{1/2}
			\bigg(\sum_{i\in I} (\bw_{| I})_i (\bSigma - 
			\bzeta_i\bzeta_i^\top)\bigg) + (1+\sqrt{2\varepsilon_{w}}) 
			\|\bar{\bzeta}_{\bw_{|I}}\|_2.
	\end{align}
	Furthermore, as already mentioned earlier, for every $\bw\in\mathcal W_n
	(\varepsilon)$, $\varepsilon_w\le \varepsilon/(1-\varepsilon)\le 1.5\varepsilon$. 
	This implies that 
	\begin{align} 
		\bfE^{1/2}[\Xi^2]
			&\le 3\varepsilon + (1+\sqrt{3\varepsilon}) 
			\bfE^{1/2}\bigg[\sup_{\bw,I}\|\bar{\bzeta}_{\bw_{|I}}\|_2^2\bigg]\\ 
			&\qquad + \sqrt{3\varepsilon}\,\bfE^{1/2}\bigg[\sup_{\bw,I}
			\lambda_{\max,+}\bigg(\sum_{i\in I} (\bw_{| I})_i (\bSigma - 
			\bzeta_i\bzeta_i^\top)\bigg) \bigg].\label{eq:14}
	\end{align}
	As proved in \Cref{lemma2} (by taking $m=2o$ 
	and $o=n\varepsilon$),
	\begin{align} \label{eq:15}
		\bfE^{1/2}\bigg[\sup_{\bw,I}\|\bar{\bzeta}_{\bw_{|I}}\|_2^2\bigg] \le 
		{\sqrt{\rSigma/n}}\,\big(1+7\sqrt{\varepsilon}\big) + 9.2\varepsilon
		\sqrt{\log(2/\varepsilon)}.
	\end{align}
	In addition, in view of the first claim of \Cref{lemma3} (with 
	$m=2o$ and $o=n\varepsilon$), stated and proved 
	in the last section, we have
	\begin{align} \label{eq:16}
		\bfE\bigg[\sup_{\bw,I}
			\lambda_{\max,+}\bigg(\sum_{i\in I} (\bw_{| I})_i (\bSigma - 
			\bzeta_i\bzeta_i^\top)\bigg) \bigg]
			&\le 25\sqrt{p/n} + 66\varepsilon \log(1/2\varepsilon).
	\end{align}
	Combining \eqref{eq:14}, \eqref{eq:15} and \eqref{eq:16}, we get 
	\begin{align} 
		\bfE^{1/2}[\Xi^2]
			&\le 3\varepsilon + (1+\sqrt{3\varepsilon}) 
			\big(\sqrt{p/n}\,\big(1+7\sqrt{\varepsilon}\big) + 9.2\varepsilon
			\sqrt{\log(2/\varepsilon)}\big)\\
			&\qquad + \sqrt{3\varepsilon}\,\big(25\sqrt{p/n} + 66\varepsilon 
			\log(1/2\varepsilon)\big)^{1/2}\\
			&\le \sqrt{p/n}(1+16\sqrt{\varepsilon}) 
			+ 18\varepsilon\sqrt{\log(2/\varepsilon)} + 
			\sqrt{3\varepsilon}\,\big(25\sqrt{p/n} + 66\varepsilon 
			\log(1/2\varepsilon)\big)^{1/2}.
	\end{align}
	This leads to \eqref{ineq:xi2}. To obtain \eqref{ineq:xi1}, we 
	repeat the same arguments but use the second claim of 
	\Cref{lemma3} instead of the first one.
    
    \textbf{Acknowledgments.}\ 
	The work of AD was partially supported by the grant
	Investissements d'Avenir (ANR-11-IDEX-0003/Labex Ecodec/ANR-11-LABX-0047). The work of AM was supported 
	by the FAST Advance grant.

	{\renewcommand{\addtocontents}[2]{}
	\bibliography{Literature}}

\begin{thebibliography}{56}
\providecommand{\natexlab}[1]{#1}
\providecommand{\url}[1]{\texttt{#1}}
\expandafter\ifx\csname urlstyle\endcsname\relax
  \providecommand{\doi}[1]{doi: #1}\else
  \providecommand{\doi}{doi: \begingroup \urlstyle{rm}\Url}\fi

\bibitem[Balakrishnan et~al.(2017)Balakrishnan, Du, Li, and
  Singh]{BalakrishnanDLS17}
S.~Balakrishnan, S.~S. Du, J.~Li, and A.~Singh.
\newblock Computationally efficient robust sparse estimation in high
  dimensions.
\newblock In \emph{{COLT} 2017}, pages 169--212, 2017.

\bibitem[Bateni and Dalalyan(2020)]{bateni2020minimax}
A.-H. Bateni and A.~S. Dalalyan.
\newblock Confidence regions and minimax rates in outlier-robust estimation on
  the probability simplex.
\newblock \emph{Electron. J. Statist.}, 14\penalty0 (2):\penalty0 2653--2677,
  2020.

\bibitem[Cai and Li(2015)]{Cai}
T.~T. Cai and X.~Li.
\newblock Robust and computationally feasible community detection in the
  presence of arbitrary outlier nodes.
\newblock \emph{Ann. Statist.}, 43\penalty0 (3):\penalty0 1027--1059, 2015.

\bibitem[Cannings et~al.(2020)Cannings, Fan, and
  Samworth]{cannings2020classification}
T.~I. Cannings, Y.~Fan, and R.~J. Samworth.
\newblock Classification with imperfect training labels.
\newblock \emph{Biometrika}, 107\penalty0 (2):\penalty0 311--330, 2020.

\bibitem[Chen et~al.(2016)Chen, Gao, and Ren]{chen2016}
M.~Chen, C.~Gao, and Z.~Ren.
\newblock A general decision theory for {H}uber's $\epsilon$-contamination
  model.
\newblock \emph{Electron. J. Statist.}, 10\penalty0 (2):\penalty0 3752--3774,
  2016.

\bibitem[Chen et~al.(2018)Chen, Gao, and Ren]{chen2018}
M.~Chen, C.~Gao, and Z.~Ren.
\newblock Robust covariance and scatter matrix estimation under {H}uber's
  contamination model.
\newblock \emph{Ann. Statist.}, 46\penalty0 (5):\penalty0 1932--1960, 10 2018.

\bibitem[Cheng et~al.(2019{\natexlab{a}})Cheng, Diakonikolas, and Ge]{Cheng19}
Y.~Cheng, I.~Diakonikolas, and R.~Ge.
\newblock High-dimensional robust mean estimation in nearly-linear time.
\newblock In \emph{{SODA} 2019}, pages 2755--2771, 2019{\natexlab{a}}.

\bibitem[Cheng et~al.(2019{\natexlab{b}})Cheng, Diakonikolas, and
  Ge]{Chengetal19}
Y.~Cheng, I.~Diakonikolas, and R.~Ge.
\newblock High-dimensional robust mean estimation in nearly-linear time.
\newblock In T.~M. Chan, editor, \emph{{SODA} 2019, San Diego}, pages
  2755--2771. {SIAM}, 2019{\natexlab{b}}.

\bibitem[Cherapanamjeri et~al.(2019)Cherapanamjeri, Flammarion, and
  Bartlett]{flammarion}
Y.~Cherapanamjeri, N.~Flammarion, and P.~L. Bartlett.
\newblock Fast mean estimation with sub-gaussian rates.
\newblock In \emph{Proceedings of COLT}, volume~99 of \emph{Proceedings of
  Machine Learning Research}, pages 786--806, Phoenix, USA, 25--28 Jun 2019.
  PMLR.

\bibitem[Chinot(2020)]{geoffrey2019erm}
G.~Chinot.
\newblock Erm and rerm are optimal estimators for regression problems when
  malicious outliers corrupt the labels.
\newblock \emph{Electron. J. Statist.}, 14\penalty0 (2):\penalty0 3563--3605,
  2020.

\bibitem[Collier and Dalalyan(2019)]{coldal}
O.~Collier and A.~S. Dalalyan.
\newblock Multidimensional linear functional estimation in sparse gaussian
  models and robust estimation of the mean.
\newblock \emph{Electron. J. Statist.}, 13\penalty0 (2):\penalty0 2830--2864,
  2019.

\bibitem[Comminges and Dalalyan(2012)]{comminges2012}
L.~Comminges and A.~S. Dalalyan.
\newblock Tight conditions for consistency of variable selection in the context
  of high dimensionality.
\newblock \emph{Ann. Statist.}, 40\penalty0 (5):\penalty0 2667--2696, 2012.

\bibitem[Comminges et~al.(2020)Comminges, Collier, Ndaoud, and
  Tsybakov]{comminges2018adaptive}
L.~Comminges, O.~Collier, M.~Ndaoud, and A.~B. Tsybakov.
\newblock Adaptive robust estimation in sparse vector model, 2020.

\bibitem[Cox et~al.(2014)Cox, Juditsky, and Nemirovski]{cox2014dual}
B.~Cox, A.~Juditsky, and A.~Nemirovski.
\newblock Dual subgradient algorithms for large-scale nonsmooth learning
  problems.
\newblock \emph{Mathematical Programming}, 148\penalty0 (1-2):\penalty0
  143--180, 2014.

\bibitem[Dalalyan and Thompson(2019)]{Thompson}
A.~Dalalyan and P.~Thompson.
\newblock Outlier-robust estimation of a sparse linear model using
  {$\ell_1$}-penalized {H}uber's {M}-estimator.
\newblock In \emph{NeurIPS 32}, pages 13188--13198. 2019.

\bibitem[Depersin and Lecu{é}(2019)]{depersin2019robust}
J.~Depersin and G.~Lecu{é}.
\newblock Robust subgaussian estimation of a mean vector in nearly linear time.
\newblock \emph{arXiv}, abs/1906.03058, 2019.

\bibitem[Devroye et~al.(2016)Devroye, Lerasle, Lugosi, and Oliveira]{Devroye}
L.~Devroye, M.~Lerasle, G.~Lugosi, and R.~I. Oliveira.
\newblock Sub-{G}aussian mean estimators.
\newblock \emph{Ann. Statist.}, 44\penalty0 (6):\penalty0 2695--2725, 2016.

\bibitem[Diakonikolas and Kane(2019)]{Diak_review}
I.~Diakonikolas and D.~M. Kane.
\newblock Recent advances in algorithmic high-dimensional robust statistics.
\newblock \emph{CoRR}, abs/1911.05911, 2019.

\bibitem[Diakonikolas et~al.(2016)Diakonikolas, Kamath, Kane, Li, Moitra, and
  Stewart]{DiakonikolasKK016}
I.~Diakonikolas, G.~Kamath, D.~M. Kane, J.~Li, A.~Moitra, and A.~Stewart.
\newblock Robust estimators in high dimensions without the computational
  intractability.
\newblock In \emph{{FOCS} 2016}, pages 655--664, 2016.

\bibitem[Diakonikolas et~al.(2017)Diakonikolas, Kane, and
  Stewart]{DiakonikolasKS17}
I.~Diakonikolas, D.~M. Kane, and A.~Stewart.
\newblock Statistical query lower bounds for robust estimation of
  high-dimensional gaussians and gaussian mixtures.
\newblock In \emph{{FOCS} 2017}, pages 73--84, 2017.

\bibitem[Diakonikolas et~al.(2018)Diakonikolas, Kamath, Kane, Li, Moitra, and
  Stewart]{DiakonikolasKK018}
I.~Diakonikolas, G.~Kamath, D.~M. Kane, J.~Li, A.~Moitra, and A.~Stewart.
\newblock Robustly learning a gaussian: Getting optimal error, efficiently.
\newblock In \emph{{SODA} 2018}, pages 2683--2702, 2018.

\bibitem[Diakonikolas et~al.(2019{\natexlab{a}})Diakonikolas, Kamath, Kane, Li,
  Moitra, and Stewart]{DiakonikolasKKL19}
I.~Diakonikolas, G.~Kamath, D.~Kane, J.~Li, A.~Moitra, and A.~Stewart.
\newblock Robust estimators in high-dimensions without the computational
  intractability.
\newblock \emph{{SIAM} J. Comput.}, 48\penalty0 (2):\penalty0 742--864,
  2019{\natexlab{a}}.

\bibitem[Diakonikolas et~al.(2019{\natexlab{b}})Diakonikolas, Kane, Karmalkar,
  Price, and Stewart]{DiakonikolasKKP19}
I.~Diakonikolas, D.~Kane, S.~Karmalkar, E.~Price, and A.~Stewart.
\newblock Outlier-robust high-dimensional sparse estimation via iterative
  filtering.
\newblock In \emph{NeurIPS 2019}, pages 10688--10699, 2019{\natexlab{b}}.

\bibitem[Dong et~al.(2019)Dong, Hopkins, and Li]{DongH019}
Y.~Dong, S.~B. Hopkins, and J.~Li.
\newblock Quantum entropy scoring for fast robust mean estimation and improved
  outlier detection.
\newblock In \emph{NeurIPS 2019}, pages 6065--6075, 2019.

\bibitem[Donoho(1982)]{Donoho1}
D.~Donoho.
\newblock \emph{Breakdown properties of multivariate location estimators}.
\newblock Phd thesis, Harvard University, 1982.

\bibitem[Donoho and Huber(1983)]{DonHuber}
D.~Donoho and P.~J. Huber.
\newblock The notion of breakdown point.
\newblock In \emph{A {F}estschrift for {E}rich {L}. {L}ehmann}, Wadsworth
  Statist./Probab. Ser., pages 157--184. Wadsworth, Belmont, CA, 1983.

\bibitem[Donoho and Gasko(1992)]{DonGasko}
D.~L. Donoho and M.~Gasko.
\newblock Breakdown properties of location estimates based on halfspace depth
  and projected outlyingness.
\newblock \emph{Ann. Statist.}, 20\penalty0 (4):\penalty0 1803--1827, 1992.

\bibitem[Elsener and van~de Geer(2018)]{Elsener}
A.~Elsener and S.~van~de Geer.
\newblock Robust low-rank matrix estimation.
\newblock \emph{Ann. Statist.}, 46\penalty0 (6B):\penalty0 3481--3509, 2018.

\bibitem[Gao(2020)]{Gao2020}
C.~Gao.
\newblock Robust regression via mutivariate regression depth.
\newblock \emph{Bernoulli}, 26\penalty0 (2):\penalty0 1139--1170, 2020.

\bibitem[Hampel(1968)]{Hampel}
F.~R. Hampel.
\newblock \emph{Contributions to the theory of robust estimation}.
\newblock PhD thesis, University of California, Berkeley, 1968.

\bibitem[Hopkins(2018)]{Hopkins1}
S.~B. Hopkins.
\newblock Sub-gaussian mean estimation in polynomial time.
\newblock \emph{CoRR}, abs/1809.07425, 2018.

\bibitem[Huber and Ronchetti(2009)]{Huber2009}
P.~J. Huber and E.~M. Ronchetti.
\newblock \emph{Robust Statistics, Second Edition}.
\newblock Wiley Series in Probability and Statistics. Wiley, 2009.

\bibitem[Klopp et~al.(2017)Klopp, Lounici, and Tsybakov]{Klopp2017}
O.~Klopp, K.~Lounici, and A.~B. Tsybakov.
\newblock Robust matrix completion.
\newblock \emph{Probability Theory and Related Fields}, 169\penalty0
  (1):\penalty0 523--564, Oct 2017.

\bibitem[Koltchinskii and Lounici(2017)]{KoltchLounici}
V.~Koltchinskii and K.~Lounici.
\newblock Concentration inequalities and moment bounds for sample covariance
  operators.
\newblock \emph{Bernoulli}, 23\penalty0 (1):\penalty0 110--133, 02 2017.

\bibitem[Lai et~al.(2016)Lai, Rao, and Vempala]{LaiRV16}
K.~A. Lai, A.~B. Rao, and S.~S. Vempala.
\newblock Agnostic estimation of mean and covariance.
\newblock In \emph{{FOCS} 2016}, pages 665--674, 2016.

\bibitem[Lecu\'{e} and Lerasle(2019)]{Lecue2019}
G.~Lecu\'{e} and M.~Lerasle.
\newblock Learning from {MOM}'s principles: {L}e {C}am's approach.
\newblock \emph{Stochastic Process. Appl.}, 129\penalty0 (11):\penalty0
  4385--4410, 2019.

\bibitem[Lecué and Lerasle(2020)]{lecu2017robust}
G.~Lecué and M.~Lerasle.
\newblock {Robust machine learning by median-of-means: Theory and practice}.
\newblock \emph{The Annals of Statistics}, 48\penalty0 (2):\penalty0 906 --
  931, 2020.

\bibitem[Lepski and Spokoiny(1997)]{LepskiSpok}
O.~V. Lepski and V.~G. Spokoiny.
\newblock Optimal pointwise adaptive methods in nonparametric estimation.
\newblock \emph{The Annals of Statistics}, 25\penalty0 (6):\penalty0
  2512--2546, 1997.

\bibitem[Lepskii(1992)]{Lepski}
O.~V. Lepskii.
\newblock Asymptotically minimax adaptive estimation. i: Upper bounds.
  optimally adaptive estimates.
\newblock \emph{Theory Probab. Appl.}, 36\penalty0 (4):\penalty0 682--697,
  1992.

\bibitem[Li and Bradic(2018)]{Bradic}
A.~H. Li and J.~Bradic.
\newblock Boosting in the presence of outliers: adaptive classification with
  nonconvex loss functions.
\newblock \emph{J. Amer. Statist. Assoc.}, 113\penalty0 (522):\penalty0
  660--674, 2018.

\bibitem[Loh(2017)]{Loh}
P.-L. Loh.
\newblock Statistical consistency and asymptotic normality for high-dimensional
  robust {$M$}-estimators.
\newblock \emph{Ann. Statist.}, 45\penalty0 (2):\penalty0 866--896, 2017.

\bibitem[Lopuha\"{a} and Rousseeuw(1991)]{Lopuha}
H.~P. Lopuha\"{a} and P.~J. Rousseeuw.
\newblock Breakdown points of affine equivariant estimators of multivariate
  location and covariance matrices.
\newblock \emph{Ann. Statist.}, 19\penalty0 (1):\penalty0 229--248, 1991.

\bibitem[Lugosi and Mendelson(2019)]{Lugosi2}
G.~Lugosi and S.~Mendelson.
\newblock Near-optimal mean estimators with respect to general norms.
\newblock \emph{Probab. Theory Related Fields}, 175\penalty0 (3-4):\penalty0
  957--973, 2019.

\bibitem[Lugosi and Mendelson(2020)]{Lugosi1}
G.~Lugosi and S.~Mendelson.
\newblock Risk minimization by median-of-means tournaments.
\newblock \emph{J. Eur. Math. Soc. (JEMS)}, 22\penalty0 (3):\penalty0 925--965,
  2020.

\bibitem[Lugosi and Mendelson(2021)]{LugMen20}
G.~Lugosi and S.~Mendelson.
\newblock {Robust multivariate mean estimation: The optimality of trimmed
  mean}.
\newblock \emph{Ann. Statist.}, 49\penalty0 (1):\penalty0 393 -- 410, 2021.

\bibitem[Maronna et~al.(2006)Maronna, Martin, and Yohai]{maronna2006robust}
R.~Maronna, D.~Martin, and V.~Yohai.
\newblock \emph{Robust Statistics: Theory and Methods}.
\newblock Wiley Series in Probability and Statistics. Wiley, 2006.

\bibitem[Minsker(2018)]{Minsker}
S.~Minsker.
\newblock Sub-{G}aussian estimators of the mean of a random matrix with
  heavy-tailed entries.
\newblock \emph{Ann. Statist.}, 46\penalty0 (6A):\penalty0 2871--2903, 2018.

\bibitem[Polzehl and Spokoiny(2000)]{PolzehlSpok}
J.~Polzehl and V.~G. Spokoiny.
\newblock Adaptive weights smoothing with applications to image restoration.
\newblock \emph{Journal of the Royal Statistical Society: Series B (Statistical
  Methodology)}, 62\penalty0 (2):\penalty0 335--354, 2000.

\bibitem[Rigollet and H{\"u}tter(2019)]{RigolletHutter}
P.~Rigollet and J.-C. H{\"u}tter.
\newblock High dimensional statistics, November 2019.

\bibitem[Rousseeuw(1985)]{Rousseeuw85}
P.~Rousseeuw.
\newblock Multivariate estimation with high breakdown point.
\newblock In \emph{Mathematical statistics and applications, {V}ol. {B} ({B}ad
  {T}atzmannsdorf, 1983)}, pages 283--297. Reidel, Dordrecht, 1985.

\bibitem[Rousseeuw and Hubert(2013)]{Rousseeuw}
P.~Rousseeuw and M.~Hubert.
\newblock High-breakdown estimators of multivariate location and scatter.
\newblock In \emph{Robustness and complex data structures}, pages 49--66.
  Springer, Heidelberg, 2013.

\bibitem[Rousseeuw(1984)]{Rousseeuw84}
P.~J. Rousseeuw.
\newblock Least median of squares regression.
\newblock \emph{J. Amer. Statist. Assoc.}, 79\penalty0 (388):\penalty0
  871--880, 1984.

\bibitem[Soltanolkotabi and Cand\`{e}s(2012)]{Soltanolkotabi}
M.~Soltanolkotabi and E.~J. Cand\`{e}s.
\newblock A geometric analysis of subspace clustering with outliers.
\newblock \emph{Ann. Statist.}, 40\penalty0 (4):\penalty0 2195--2238, 2012.

\bibitem[Stahel(1981)]{Stahel}
W.~Stahel.
\newblock \emph{Robuste Sch{\"a}tzungen: infinitesimale Optimalit{\"a}t und
  Sch{\"a}tzungen von Kovarianzmatrizen}.
\newblock Phd thesis, ETH Zurich, 1981.

\bibitem[Vershynin(2012)]{vershynin_2012}
R.~Vershynin.
\newblock Introduction to the non-asymptotic analysis of random matrices.
\newblock In \emph{Compressed sensing}, pages 210--268. Cambridge Univ. Press,
  Cambridge, 2012.

\bibitem[Zhu et~al.(2020)Zhu, Jiao, and Steinhardt]{zhu2020tukey}
B.~Zhu, J.~Jiao, and J.~Steinhardt.
\newblock When does the tukey median work?
\newblock In \emph{IEEE International Symposium on Information Theory (ISIT)},
  2020.

\end{thebibliography}

    \newpage{}
    
    \addtocontents{toc}{\setcounter{tocdepth}{-1}}
    \appendix
    \section{High-probability bounds in the sub-Gaussian case}
    
    This section is devoted to the proof of \Cref{th:2}, which 
    provides a high-probability upper bound on the error of the
    iteratively reweighted mean estimator in the sub-Gaussian 
    case. We start with some technical lemmas, and postpone the
    full proof of \Cref{th:2} to the end of the present section.
    
    Let us remind some notation. We assume here that for some 
    covariance matrix $\bSigma$, we have $\bzeta_i = \bSigma^{1/2}
    \bxi_i$, for $i=1,\ldots,n$, where $\bxi_i$'s are independent 
    zero-mean with identity covariance matrix. In addition, we 
    assume that $\bxi_i$s are sub-Gaussian vectors with parameter 
    $\tau>0$, that is
	\begin{align}
	    \mathbf E[e^{\bv^\top\bxi_i}] \le 
	    e^{\nicefrac{\tau}2 \|\bv\|_2^2},\qquad 
	    \forall \bv\in\mathbb R^p.
	\end{align}
	We write $\bxi_i\stackrel{\textup{ind}}{\sim} SG_p(\tau)$. Recall 
	that if $\bxi_i$ is standard Gaussian then it is sub-Gaussian with
	parameter $1$. It is a well-known fact (see, for instance, \citep[Theorem 1.19]{RigolletHutter}) 
	that if $\bxi\sim SG_p(\tau)$
	then for all $\delta\in(0,1)$, it holds
	\begin{align}\label{normSG}
	    \mathbf P\Big(\|\bxi\|_2\ge 4\sqrt{p\tau} + 
	    \sqrt{8\tau \log(1/\delta)}\Big)\le \delta. 
	\end{align}
	In what follows, we assume that the covariance matrix $\bSigma$
	satisfies $\|\bSigma\|_{\textup{op}}=1$. 
	

    \begin{lemma}\label{lem:9}
    Let $\bar J \subset \{1,\ldots,n\}$ be a subset of cardinality $m$. 
    For every $\delta\in(0,1)$, it holds that
    \begin{align}
        \mathbf P\bigg(\bigg\|\sum_{j\in\bar J} \bzeta_j\bigg\|_2
        \le \sqrt{m\tau}\big(4\sqrt{p} + \sqrt{8\log(1/\delta)}
        \big)\bigg)\ge 1-\delta.
    \end{align}
    \end{lemma}
    
    \begin{proof}[Proof of \Cref{lem:9}]
    Without loss of generality, we assume that $\bar J = \{1,\ldots,m\}$.
    On the one hand,  $\|\bSigma\|_{\textup{op}}=1$ implies that 
    \begin{align}
        \bigg\|\sum_{i=1}^m \bzeta_i\bigg\|_2 \le 
        \bigg\|\sum_{i=1}^m \bxi_i\bigg\|_2.
    \end{align}
    On the other hand, $\bxi_1+\ldots+\bxi_m\sim SG_p(m\tau)$. 
    Applying inequality \eqref{normSG} to this random variable,
    we get the claim of the lemma. 
    \end{proof}
    
	\begin{lemma}\label{lemma2bis}
	Let $p\ge 1$, $n\ge 1$, $m\in[2,0.562n]$ and $o\in [0, m]$ 
	be four positive integers. For every $\delta\in(0,1)$, with
	probability at least $1-\delta$, we have
	\begin{align}
		\sup_{\substack{\bw\in \mathcal W_{n,n-m+o} \\ 
		|I| \ge n-o}}{\|\bar{\bzeta}_{\bw_{|I}}\|}_2  
		\le 16\sqrt{\frac{\tau p}{n}} + 
		6.5\frac{m\sqrt{\tau\log(ne/m)}}{n} + 
		8\sqrt{\frac{2\tau\log(2/\delta)}{n}}.
	\end{align}
	\end{lemma}
    \begin{proof}[Proof of \Cref{lemma2bis}] 
	We have
	\begin{align}
		\sup_{\bw\in \mathcal W_{n,n-m+o}}\|\bar{\bzeta}_{\bw_{|I}}\|_2 
			& \stackrel{(a)}{\le} \sup_{\bw\in \mathcal W_{n,n-m}}
			    \|\bar{\bzeta}_{\bw}\|_2\stackrel{(b)}{\le}
			    \sup_{|J| = n-m}\big\|\bar{\bzeta}_{\bu^{J}}\big\|_2,
			\label{eq:L21}
	\end{align}
	where $(a)$ follows from claim iv)  of \Cref{lemma0} and $(b)$ is 
	a direct consequence of claim iii) of \Cref{lemma0}. Thus, we get
	\begin{align}
		\sup_{\bw\in \mathcal W_{n,n-m+o}}
		\|\bar{\bzeta}_{\bw_{|I}}\|_2
		&\le  \max_{|J| = n-m} 
			{\big\|\bar{\bzeta}_{\bu^{J}}\big\|}_2
			= \frac1{n-m}\max_{|J| = n-m}  
			\bigg\|\sum_{i\in J} \bzeta_i\bigg\|_2\\
		&\le \frac1{n-m}
			\bigg\|\sum_{i=1}^n \bzeta_i\bigg\|_2 + 
			\frac1{n-m}\max_{|\bar J| = m}  
			\bigg\|\sum_{i\in \bar J} \bzeta_i\bigg\|_2.\label{eq:L22bis}
	\end{align}
    By the union bound, for every $t\ge 0$, we have
	\begin{align}
		\bfP\bigg(\max_{|\bar J| = m}  \bigg\|\sum_{i\in \bar  J} 
			\bzeta_i\bigg\|_2> \sqrt{m\tau}\big(4\sqrt{p} + 2t\big)\bigg) 
		&\le {{n}\choose{m}} \max_{|\bar J| = m} \bfP\bigg( \bigg\| 
			\sum_{i\in\bar J} \bzeta_i\bigg\|_2 > \sqrt{m\tau}
			\big(4\sqrt{p} + 2t\big) \bigg)\\
		&\le {{n}\choose{m}} e^{-t^2/2},
	\end{align}
	where the last line follows from \Cref{lem:9}. Hence, with probability
	at least $1-\delta/2$, we have
	\begin{align}\label{lem9:eq1}
	    \max_{|\bar J| = m}  \bigg\|\sum_{i\in \bar  J} 
			\bzeta_i\bigg\|_2 \le \sqrt{m\tau}\big(4\sqrt{p} + 
			\sqrt{8m\log(ne/m)} +\sqrt{8\log(2/\delta)}\big),
	\end{align}
	where we have used the inequality $\log {n\choose{m}}\le m 
	\log (ne/m)$. Using once again \Cref{lem:8}, we check that
	with probability at least $1-\delta/2$,
	\begin{align}\label{lem9:eq2}
	    \bigg\|\sum_{i=1}^n \bzeta_i\bigg\|_2 
	    \le \sqrt{n\tau}\big(4\sqrt{p} + \sqrt{8\log(2/\delta)}\big).
	\end{align}
	Combining \eqref{eq:L22bis}, \eqref{lem9:eq1} and \eqref{lem9:eq2},
	and setting $\alpha=m/n\le 0.562$, we arrive at
	\begin{align}
	    \sup_{\bw\in \mathcal W_{n,n-m+o}}
		\|\bar{\bzeta}_{\bw_{|I}}\|_2 &\le 
		\sqrt{\frac{\tau}{n}}\bigg(\frac{4\sqrt{p}+\sqrt{8\log(2/\delta)}}{1
		-\sqrt{\alpha}}\bigg) + \frac{m\sqrt{8\tau\log(ne/m)}}{n(1-\alpha)}\\
		&\le 
		16\sqrt{\frac{\tau p}{n}} + 8\sqrt{\frac{2\tau\log(2/\delta)}{n}} + 6.5\frac{m\sqrt{\tau\log(ne/m)}}{n}
	\end{align}
	with probability at least $1-\delta$. This completes the proof of the lemma.
	\end{proof}
	
	We recall that $\bzeta_{1:n}$ is the 
	$p\times n$ random matrix obtained by concatenating the vectors 
	$\bzeta_i$. We also remind that  $s_{\min}(\bzeta_{1:n}) = 
	\lambda_{\min}^{1/2}(\bzeta_{1:n}\bzeta_{1:n}^\top)$ 
	and $s_{\max}(\bzeta_{1:n}) = \lambda_{\max}^{1/2}
	(\bzeta_{1:n}\bzeta_{1:n}^\top)$ are the smallest and the 
	largest singular values of the matrix $\bzeta_{1:n}$. 

	\begin{lemma}[\cite{vershynin_2012}, Theorem 5.39]\label{lem:vershynin:bis}
	There is a constant $C_\tau$ depending only on the variance 
	proxy $\tau$ such that for every $t>0$ and for every pair of 
	positive integers $n$ and $p$, 
	we have
	\begin{align}
		\bfP\big(s_{\min}(\bxi_{1:n})  \le \sqrt{n} - C_\tau\sqrt{p} 
		- C_\tau t\big)  \le e^{-t^2},\\
		\bfP\big(s_{\max}(\bxi_{1:n})  \ge \sqrt{n} + C_\tau\sqrt{p} 
		+ C_\tau t\big)  \le e^{-t^2}.
	\end{align}
	\end{lemma}

	\begin{lemma} 
	For a subset $\bar J$ of $\{1,\ldots,n\}$, we denote by 
	$\bzeta_{\bar J}$ the $p\times |\bar J|$ matrix obtained by
	concatenating the vectors $\{\bzeta_i: i\in\bar J\}$. For every 
	pair of integers  $n,p\ge 1$ and for every integer $m\in[1,n]$,
	we have
	\begin{align}
		\bfP\Big(\max_{|\bar J| = m} s_{\max}(\bzeta_{\bar J}) &\le
				\sqrt{m} +  C_\tau\big(\sqrt{p} + \sqrt{m \log(\nicefrac{ne}{m})}
				+ t\big) \Big) \ge 1 - e^{-t^2},\label{esp:3bis}
	\end{align}	
	where $C_\tau$ is the same constant as in \Cref{lem:vershynin:bis}.
	\end{lemma} 
	\begin{proof}
	Using the union bound, we get
	\begin{align}
		\bfP\Big(\max_{|\bar J| = m} s_{\max}(\bxi_{\bar J}) &\ge
			\sqrt{m} +  C_\tau\big(\sqrt{p} + \sqrt{m \log( \nicefrac{ne}{m})} + t\big) \Big)\\ 
		&\le {n\choose{m}}\max_{|\bar J| = m} 
			\bfP\Big( s_{\max}(\bxi_{\bar J}) \ge \sqrt{m} +
			C_\tau\big(\sqrt{p} + \sqrt{m \log(\nicefrac{ne}{m})}+ 
			t\big) \Big)\\
		&\stackrel{(1)}{\le} {n\choose{m}} \exp\big(-\{
		\sqrt{m \log(\nicefrac{ne}{m})} +t\}^2\big)
		\stackrel{(2)}{\le} e^{-t^2},
	\end{align}	
	where in (1) above we have used the second inequality from
	\Cref{lem:vershynin:bis}, while (2) is a consequence of
	the inequality $\log{n\choose{m}}\le m\log(ne/m)$. 
	\end{proof}
    
	
	\begin{lemma}[\cite{KoltchLounici}, Theorem 9] \label{lem:smax:bis}
	There is a constant $A_3>0$ depending only on the variance proxy $\tau$ such that for every pair of integers
	$n\ge 1$ and $p\ge 1$, we have
	\begin{align}
		\bfP\bigg(\|\bzeta_{1:n}\bzeta_{1:n}^\top-n\bSigma\|_{\textup{op}}
		\ge A_3\Big(\sqrt{(p+t)n} + p + t\Big)\bigg)\le e^{-t},\quad \forall t\ge 1.
	\end{align}
	\end{lemma}

	\begin{lemma}\label{lemma3:bis}
	Let $t>0$. For any $p\ge 1$, $n\ge 1$, $m\in[2,0.6n]$ and $o\le m$, with probability at least $1-2e^{-t}$,
	\begin{align}
		\sup_{\substack{\bw\in \mathcal W_{n,n-m+o} \\ |I| \ge n-o}}
		\Big\|\bSigma & - \sum_{i=1}^n (\bw_{|I})_i\bzeta_i\bzeta_i^\top
		\Big\|_{\textup{op}} \le  5A_3 \bigg( \sqrt{ \frac{p+t}{n}} + 
			\frac{p + t}{n} + \frac{m\log(ne/m)}{n}\bigg),
	\end{align}
	where $A_3$ is the same constant as in \Cref{lem:smax:bis}.
	\end{lemma}

	\begin{proof}[Proof of \Cref{lemma3:bis}]
	For every $I\subset [n]$ such that $|I| \ge n-o$, we check that
	for every weight vector $\bw\in \mathcal W_{n,n-m+o}$,
	\begin{align}
		(n-m) \Big\|\bSigma - \sum_{i=1}^n (\bw_{|I})_i \bzeta_i 
		    \bzeta_i^\top \Big\|_{\textup{op}}
		&\le  \bigg\|\sum_{i=1}^n(\bSigma-\bzeta_i\bzeta_i^\top)
		    \bigg\|_{\textup{op}} + \max_{|\bar J| = m} \bigg\|
		    \sum_{i\in \bar J}(\bzeta_i\bzeta_i^\top - \bSigma) \bigg\|_{\textup{op}}\\
		&=\big\|\bzeta_{1:n}\bzeta_{1:n}^\top-n\bSigma\big\|_{\textup{op}}
			+\max_{|\bar J| = m} \big\|\bzeta_{\bar J}\bzeta_{\bar J
			}^\top - m\bSigma\big\|_{\textup{op}}\\
		&\le\big\|\bxi_{1:n}\bxi_{1:n}^\top-n\bfI_p\big\|_{
		    \textup{op}} +\max_{|\bar J| = m} \big\|\bxi_{
		    \bar J}\bxi_{\bar J}^\top - m \bfI_p \big\|_{\textup{op}}.\label{eq:L81bis}
	\end{align}
	On the one hand, \Cref{lem:smax:bis} implies that the inequality
	\begin{align}
	    \|\bxi_{\bar J}\bxi_{\bar J}^\top-m\bfI\|_{\textup{op}}
	    &\le A_3\big(\sqrt{(p+t)m} + p + t_0\big)\\
	    &\le 0.5A_3\big(m + 3p + 3t_0\big),
	\end{align}
	holds with probability at least $1 - e^{-t_0}$. Choosing $t_0 =
	t + \log {n\choose{m}}$ and using the union bound, 
	the last display yields	that with probability at least $1-e^{-t}$,
	\begin{align}
		\max_{|\bar J| = m} \|\bxi_{\bar J}\bxi_{\bar J}^\top -m\bfI 
		    \|_{\textup{op}}
		&\le  2A_3\Big(p + m\log(ne/m) + t\Big).
		\label{eq:L37c:bis}
	\end{align}
    Using once again \Cref{lem:smax:bis}, we  can check that the 
    inequality
    \begin{align}
		\|\bxi_{1:n}\bxi_{1:n}^\top -n\bfI 
		    \|_{\textup{op}}
		&\le  A_3\Big(\sqrt{n(p+\log(2/\delta))} + p + t\Big)
		\label{eq:L37b:bis}
	\end{align}
    holds with probability at least $1-e^{-t}$. Therefore, combining
    \eqref{eq:L81bis}, \eqref{eq:L37c:bis} and \eqref{eq:L37b:bis}, we 
    get the following inequalities hold with probability at least
    $1-2e^{-t}$:
	\begin{align}
		\sup_{\mathcal W_{n,n-m+o}}\limits \Big\|\bSigma - 
			\sum_{i=1}^n (\bw_{|I})_i\bzeta_i\bzeta_i^\top\Big\|_{\textup{op}}
			&\le  \frac{A_3}{n-m}\Big(\sqrt{n(p+t)} + 2p + 
			2t + 2m\log(ne/m)\Big)\\
			&\le  5A_3\Big(\sqrt{\frac{p+t}{n}} + 
			\frac{p + t}{n} + \frac{m\log(ne/m)}{n}\Big).
			\label{eq:L85bis}
	\end{align}
	This completes the proof.
	\end{proof}
    
    \begin{proposition}\label{proposition3}
	Let $R(\bzeta,\bw, I)$ be defined in \Cref{proposition1} and 
	let $\bzeta_1,\ldots,\bzeta_n$ be independent centered random 
	vectors with covariance matrix $\bSigma$ satisfying 
	$\lambda_{\max}(\bSigma) = 1$. Assume that $\bxi_i = 
	\bSigma^{-1/2}\bzeta_i$ are sub-Gaussian with a variance proxy
	$\tau>0$. Let $\varepsilon\le 0.28$, $p\le n$ and $\delta\in 
	(4e^{-n}, 1)$. The random variable 
	\begin{align}
		\Xi = \sup_{\bw\in \mathcal W_n(\varepsilon)} 
		\max_{|I|\ge n(1-\varepsilon)} R(\bzeta,\bw, I)
	\end{align}
	satisfies, with probability at least $1-\delta$, the inequality 
	\begin{align}
		\Xi &\le A_4\bigg(\sqrt{\frac{p + \log(4/\delta)}{n}} + 
		\varepsilon\sqrt{\log(1/\varepsilon)}\bigg), \label{ineq:xi:bis}
	\end{align}
	where $A_4>0$ is a constant depending only on the variance
	proxy $\tau$.
	\end{proposition}
	
    \begin{proof}[Proof of \Cref{proposition3}]
	We recall that, for any subset $I$ of $\{1,\ldots,n\}$,
	\begin{align} 
		R(\bzeta,\bw, I) 
			&= 2\varepsilon_{w}+\sqrt{2\varepsilon_{w}}\,\lambda_{\max,+}^{1/2}
			\bigg(\sum_{i\in I} (\bw_{| I})_i (\bSigma - 
			\bzeta_i\bzeta_i^\top)\bigg) + (1+\sqrt{2\varepsilon_{w}}) 
			\|\bar{\bzeta}_{\bw_{|I}}\|_2.
	\end{align}
	Furthermore, for every $\bw\in\mathcal W_n
	(\varepsilon)$, $\varepsilon_w\le \varepsilon/(1-\varepsilon)\le 1.5\varepsilon$. 
	This implies that 
	\begin{align} 
		\Xi &\le 3\varepsilon 
			+ \sqrt{3\varepsilon}\,\sup_{\bw,I}
			\lambda_{\max,+}\bigg(\sum_{i\in I} (\bw_{| I})_i (\bSigma - 
			\bzeta_i\bzeta_i^\top)\bigg)^{1/2}
			+ (1+\sqrt{3\varepsilon}) 
			\sup_{\bw,I}\|\bar{\bzeta}_{\bw_{|I}}\|_2,\label{eq:14b}
	\end{align}
	where $\sup_{\bw,I}$ is the supremum over all $\bw\in\mathcal
	W_n(\varepsilon)$ and over all $I\subset\{1,\ldots,n\}$ of 
	cardinality larger than or equal to $n(1-\varepsilon)$. As proved 
	in \Cref{lemma2bis} (by taking $m=2o$ and $o=n\varepsilon$), with
	probability at least $1-2e^{-t}$,
	\begin{align} \label{eq:15b}
		\sup_{\bw,I}\|\bar{\bzeta}_{\bw_{|I}}\|_2 \le 
		16\sqrt{\frac{\tau p}{n}} + 13\varepsilon\sqrt{\tau\log(2/\varepsilon)} + 
		8\sqrt{\frac{2t\tau}{n}}.
	\end{align}
	In addition, in view of \Cref{lemma3:bis} (with $m=2o$ and
	$o=n\varepsilon$), with probability at least $1-2e^{-t}$, we have
	\begin{align} 
		\sup_{\bw,I}\lambda_{\max,+}\bigg(\sum_{i\in I} (\bw_{| I})_i
		    (\bSigma - \bzeta_i\bzeta_i^\top)\bigg)
		&\le 5A_3 \bigg( \sqrt{ \frac{p+t}{n}} + 
			\frac{p + t}{n} + 2\varepsilon\log(2/\varepsilon)\bigg)\\
		&\le 13 A_3 \bigg( 
			\sqrt{\frac{p + t}{n}} + \varepsilon\log(2/\varepsilon)\bigg),
			\label{eq:16b}	
	\end{align}
	where in the last inequality we have used that $p\le n$ and $t\le n$. 
	Combining \eqref{eq:14b}, \eqref{eq:15b} and \eqref{eq:16b}, we 
	obtain that the inequality
	\begin{align} 
		\Xi &\le A_4\bigg(\sqrt{\frac{p + t}{n}} + \varepsilon
		\sqrt{\log(1/\varepsilon)}\bigg)
	\end{align}
	hold true with probability at least $1-4e^{-t}$, for every $t\in[0,n]$, where $A_4$ is a constant depending only
	on the variance proxy $\tau$. Replacing $t$ by $\log(4/\delta)$, 
	we get the claim of the proposition.
	\end{proof}

\begin{proof}[Proof of \Cref{th:2}]
    The first step of the proof consists in establishing 
    a recurrent inequality relating the estimation error 
    of $\hat\bmu^{k}$ to that of $\hat\bmu^{k-1}$. More 
    precisely, we show that this error decreases at a 
    geometric rate, up to a small additive error term. 
    Without loss of generality, we assume that 
    $\|\bSigma\|_{\textup{op}} = 1$.
    
    Let us recall the notation
    \begin{align}
        \mathcal W_n(\varepsilon) = \Big\{\bw\in \bDelta^{n-1}: 
        \max_i w_i\le \frac1{n(1-\varepsilon)}\Big\}
    \end{align}
    and
    \begin{align}
        R(\bzeta,\bw,I) = 2\varepsilon_w + \sqrt{2\varepsilon_w}
        \lambda^{1/2}_{\max,+}\bigg(\sum_{i\in I} (
        \bw_{|I})_i(\bSigma - \bzeta_i\bzeta_i^\top)\bigg)
        +(1+\sqrt{2\varepsilon_w})\|\bar\bzeta_{\bw|I}\|_2.
    \end{align}
    For every $k\ge 1$, in view of \Cref{proposition1} and 
    \eqref{eq:minG}, we have 
    \begin{align}
        \|\hat\bmu^{k}-\bmu^*\|_2 
        &\le \frac{\sqrt{\varepsilon_k}}{
        1-\varepsilon_k} G(\hat\bw^k,\hat\bmu^{k})^{1/2} + \Xi\\
        &\le \frac{\sqrt{\varepsilon_k}}{
        1-\varepsilon_k} G(\hat\bw^k,\hat\bmu^{k-1})^{1/2} + \Xi,\label{mu:k}
    \end{align}
    where 
    \begin{align}\label{epsk}
        \varepsilon_k = \sum_{i\not\in\mathcal I} w_i^{k}\le 
    \frac{\varepsilon n }{n-n\varepsilon} = \frac{\varepsilon}{1-
    \varepsilon}
    \end{align}
    and
    \begin{align}
        \Xi = \max_{|I|\ge n-n\varepsilon} \sup_{\bw\in \mathcal W_n
        (\varepsilon)} R(\bzeta,\bw,I).
    \end{align}
    Inequality \eqref{epsk} and the cat that the function $x\mapsto 
    \sqrt{x}/(1-x)$ is increasing for $x\in[0,1)$ imply that
    \begin{align}
        \frac{\sqrt{\varepsilon_k}}{
        1-\varepsilon_k} \le \frac{\sqrt{\varepsilon(1-\varepsilon)}}{
        1-2\varepsilon} = \alpha_\varepsilon.
    \end{align}
    Using this inequality and the fact that $\hat\bw^k$ is a minimizer
    of $G(\cdot, \hat\bmu^{k-1})$, we infer from 
    \eqref{mu:k} that
    \begin{align}\label{mu:k1}
        \|\hat\bmu^{k}-\bmu^*\|_2 &\le \alpha_\varepsilon
        G(\hat\bw^k,\hat\bmu^{k-1})^{1/2} + \Xi\\
        &\le \alpha_\varepsilon
        G(\bw^*_{\mathcal I},\hat\bmu^{k-1})^{1/2} + \Xi,
    \end{align}
     where $\bw^*_{\mathcal I}$ is the weight vector defined by 
    $w^*_i = \mathds 1(i\in\mathcal I)/|\mathcal I|$.
    One can check that
	\begin{align}
		\sum_{i=1}^n  w_i^* (\bs X_i-\hat{\bmu}^{k-1})^{\otimes 2} &= 
		\sum_{i=1}^n w_i^*(\bs X_i-\bar{\bs X}_{\bw^*_{\mathcal I}})^{\otimes 2} 
		+ (\bar{\bs X}_{\bw^*_{\mathcal I}}-\hat{\bmu}^{k-1})^{\otimes 2}\\
		&\preceq \sum_{i=1}^n w_i^*(\bs X_i-\bar{\bs X}_{\bw^*_{\mathcal I}})^{\otimes 2}
		+ \big\|\bar{\bs X}_{\bw^*_{\mathcal I}}-\hat{\bmu}^{k-1}\big\|_2^2\,\bfI_p\\
		&\preceq \sum_{i=1}^n w_i^*(\bs X_i-\bmu^*)^{\otimes 2}
		 + \big\|\bar{\bs X}_{\bw^*_{\mathcal I}}-\hat{\bmu}^{k-1}\big\|_2^2\,\bfI_p.
	\end{align}	
	This readily yields
	\begin{align}
		G(\bw^*_{\mathcal I},\hat{\bmu}^{k-1}) &\le \lambda_{\max,+}
		    \bigg(\sum_{i=1}^n 
			w_i^*(\bs X_i-\bmu^*)^{\otimes 2} - \bSigma\bigg) + 
			\big\|\bar{\bs X}_{\bw^*_{\mathcal I}}-\hat{\bmu}^{k-1}\big\|_2^2\\
			&\le G(\bw^*_{\mathcal I}, \bmu^*)  + 
			\big(\big\|\bar{\bzeta}_{\bw^*_{\mathcal I}}\big\|_2 + 
			\big\|\bmu^*-\hat{\bmu}^{k-1}\big\|_2\big)^2.
	\end{align}
	Combining with inequality \eqref{mu:k1}, we arrive at
	\begin{align}
        \|\hat\bmu^{k}-\bmu^*\|_2 
            &\le \alpha_\varepsilon G(\bw^*_{\mathcal I}, \bmu^*)^{1/2}  
            + \alpha_\varepsilon\big\|\bar{\bzeta}_{
            \bw^*_{\mathcal I}}\big\|_2 + \alpha_\varepsilon
			\big\|\bmu^*-\hat{\bmu}^{k-1}\big\|_2 + \Xi.
    \end{align}    
    Thus, we have obtained the desired recurrent inequality
    \begin{align}\label{mu:k3}
        \|\hat\bmu^{k}-\bmu^*\|_2 
		&\le \alpha_\varepsilon \big\|\hat{\bmu}^{k-1}-\bmu^*\big\|_2
		+\tilde\Xi,
    \end{align}
    with 
    \begin{align}
        \tilde\Xi  &= \alpha_\varepsilon G(\bw^*_{\mathcal I}, 
        \bmu^*)^{1/2} + \alpha_\varepsilon\big\|\bar{\bzeta}_{
            \bw^*_{\mathcal I}}\big\|_2 + \Xi\\
        &\le \sqrt{5\varepsilon}\,\sup_{\bw,I}
			\bigg\|\sum_{i\in I} (\bw_{| I})_i (\bSigma - 
			\bzeta_i\bzeta_i^\top)\bigg\|_{\rm op}^{1/2} + \sqrt{5\varepsilon}\,\big\|\bar{\bzeta}_{
        \bw^*_{\mathcal I}}\big\|_2 + \Xi.\label{xi:tilde}
    \end{align}
    The last inequality above follows from the inequality
    $\alpha_\varepsilon\le \sqrt{5\varepsilon}$ as soon
    as $\varepsilon\le 0.3$. Unfolding the recurrent inequality
    \eqref{mu:k3}, we get
    \begin{align}\label{mu:K1}
        \|\hat\bmu^{K}-\bmu^*\|_2 
		&\le \alpha_\varepsilon^K \big\|\hat{\bmu}^{0}-\bmu^*\big\|_2
		+\frac{\tilde\Xi}{1-\alpha_\varepsilon}.
    \end{align}
    In view of \Cref{lemma0} and the condition $\varepsilon
    \le 0.3$, we get
    \begin{align}\label{mu:K2}
        \|\hat\bmu^{K}-\bmu^*\|_2 
		&\le \frac{5\alpha_\varepsilon^K}{n} \sum_{i=1}^n 
		\big\|\bzeta_i\big\|_2 +\frac{\tilde\Xi}{1-\alpha_\varepsilon}.
    \end{align}
    Simple algebra yields
    \begin{align}
        \frac1n\sum_{i=1}^n \big\|\bzeta_i\big\|_2
        &\le \bigg(\frac1n\sum_{i=1}^n \big\|\bxi_i\big\|_2^2\bigg)^{1/2}
        \le \sqrt{p} + \bigg(\frac1n\sum_{i=1}^n (\big\|\bxi_i\big\|_2^2-p)\bigg)^{1/2}\\
        &= \sqrt{p} + \bigg(\frac{\tr(\bxi_{1:n}\bxi_{1:n}^\top
        -\bfI_p)}n\bigg)^{1/2}\\
        &\le \sqrt{p} + \sqrt{p/n}\,\big\|\bxi_{1:n}\bxi_{1:n}^\top
        -\bfI_p\big\|_{\rm op}^{1/2}.
    \end{align}
    Thus, in view of \Cref{lem:smaxbis}, with probability
    at least $1-\delta$, 
    \begin{align}
        \frac1n\sum_{i=1}^n \big\|\bzeta_i\big\|_2
        &\le \sqrt{p}\bigg(1 +\sqrt{2A_3} +  \sqrt{\frac{1.5 A_3\log(1/\delta)}{n}}\bigg).
    \end{align}
    Since $\delta\ge 4e^{-n}$ and $K$ is chosen in such a way
    that $\alpha_\varepsilon^K \le 0.5\varepsilon/\sqrt{p}$, 
    we get, on an event of probability at least $1-\delta$,
    \begin{align}
        \|\hat\bmu^{K}-\bmu^*\|_2 
		&\le 0.5 \varepsilon(1+3A_3)
		+\frac{\tilde\Xi}{1-\alpha_\varepsilon}.
    \end{align}
    Combining \eqref{xi:tilde} with \Cref{proposition3},
    \Cref{lem:9} and \Cref{lemma3:bis}, we check that for some
    constant $A_5$ depending only on the variance proxy $\tau$,
    the inequality
    \begin{align}
        \|\hat\bmu^{K}-\bmu^*\|_2 
		&\le \frac{A_5}{1-\alpha_\varepsilon}
		\Bigg(\sqrt{\frac{p+\log(4/\delta)}{n}} + \varepsilon\sqrt{\log(1/\varepsilon)}\Bigg)
    \end{align}
    holds with probability at least $1-4\delta$. 
    This completes the proof of the theorem. 
    \end{proof}

    \section{Proof of adaptation to unknown contamination rate}
    
    This section provides a proof of \Cref{th:3}. 
	Let $\varepsilon$ be the true contamination rate, which means 
	that the distribution $\bfPn\in\textup{SGAC}(\bmu^*,\bSigma,
	\varepsilon)$. Let $\ell^*$ be the largest value of $\ell$
	such that the corresponding element $\varepsilon_\ell$ of the 
	grid is larger than or equal to $\varepsilon$. Since 
	$\varepsilon$ is assumed to be smaller than $\varepsilon_0 a$, 
	the following inequalities hold
	\begin{align}\label{chain:eps}
	    \varepsilon_0 > \varepsilon_{\ell^*} \ge \varepsilon 
	    \ge \varepsilon_{\ell^*}a.
	\end{align}
	Let us introduce the events $\Omega_j = \{\bmu^*\in\mathbb B
	(\hat\bmu_n^{\textup{IR}}(\varepsilon_j),R_\delta(
	\varepsilon_j))\}$. Since $\bfPn\in \textup{SGAC}(\bmu^*,
	\bSigma,\varepsilon)\subseteq \textup{SGAC}(\bmu^*,\bSigma,
	\varepsilon_j)$ for all $j \le \ell^*$. We infer 
	from \Cref{th:2} that $\bfP (\Omega_j)\ge 1-4\delta/
	\ellmax$ for every $j\le \ell^*$. Hence, by	the union 
	bound,  with probability at least $1-4\delta$, we have
	\begin{align}\label{ell*}
	    \bmu^*\in \bigcap_{j=1}^{\ell_*} \mathbb B
	\big(\hat\bmu_n^{\textup{IR}}(\varepsilon_j),R_\delta(
	\varepsilon_j)\big).
	\end{align} 
	From now on, we assume that this event is realized. Clearly,
	this implies that $\hat\ell\ge \ell^*$ and, therefore, 
	\begin{align}
	\mathbb B
	\big(\hat\bmu_n^{\textup{IR}}(\varepsilon_{\ell^*}),R_\delta(
	\varepsilon_{\ell^*})\big) \cap 
	\mathbb B
	\big(\hat\bmu_n^{\textup{AIR}},R_\delta(
	\varepsilon_{\hat\ell})\big)\neq \varnothing\qquad 
	\text{and}\qquad \varepsilon_{\hat\ell}\le \varepsilon_{\ell^*}. 
	\end{align}
	Using the triangle inequality, we get 
	\begin{align}\label{AIR:IR}
	    \|\hat\bmu_n^{\textup{AIR}} - \hat\bmu_n^{\textup{IR}}
	    (\varepsilon_{\ell^*})\|_2 \le 
	    R_\delta(\varepsilon_{\ell^*}) + 2R_\delta
	    (\varepsilon_{\hat\ell})
	    \le 2R_\delta
	    (\varepsilon_{\ell^*}) \le 
	    2R_\delta
	    (\varepsilon/a),
	\end{align}
	where the last inequality follows from \eqref{chain:eps} and the
	fact that $z\mapsto R_\delta(z)$ is an increasing function. 
	Combining \eqref{AIR:IR} and \eqref{ell*}, we get
	\begin{align}
	    \|\hat\bmu_n^{\textup{AIR}} - \bmu^*\|_2 \le 
	    \|\hat\bmu_n^{\textup{AIR}} - \hat\bmu_n^{\textup{IR}}
	    (\varepsilon_{\ell^*})\|_2 + \|\hat\bmu_n^{\textup{IR}}
	    (\varepsilon_{\ell^*})-\bmu^*\|_2 \le 
	    3R_\delta (\varepsilon/a). 
	\end{align}
	This completes the proof of \Cref{th:3}. 

\section{Proof in the case of unknown $\bSigma$}
    
    This section is devoted to the proof of \Cref{th:4}.
    We follow exactly the same steps as in the proof of
    \Cref{th:2}. Without loss of generality, we assume 
    that $\|\bSigma\|_{\textup{op}} \le 1$. For every 
    $k\ge 1$, in view of \Cref{proposition1} and 
    \eqref{eq:minG}, we have 
    \begin{align}
        \|\hat\bmu^{k}-\bmu^*\|_2 
        &\le \frac{\sqrt{\varepsilon_k}}{
        1-\varepsilon_k} G(\hat\bw^k,\hat\bmu^{k})^{1/2} + \Xi
        \le \frac{\sqrt{\varepsilon_k}}{
        1-\varepsilon_k} G(\hat\bw^k,\hat\bmu^{k-1})^{1/2} + \Xi,\label{mu:kbis}
    \end{align}
    where 
    \begin{align}\label{epskbis}
        \varepsilon_k = \sum_{i\not\in\mathcal I} w_i^{k}\le 
    \frac{\varepsilon n }{n-n\varepsilon} = \frac{\varepsilon}
    {1-\varepsilon}
    \end{align}
    and $\Xi = \max_{|I|\ge n-n\varepsilon} \sup_{\bw\in 
    \mathcal W_n(\varepsilon)} R(\bzeta,\bw,I)$.
    
    Inequality \eqref{epskbis} and the cat that the function
    $x\mapsto \sqrt{x}/(1-x)$ is increasing for $x\in[0,1)$ 
    imply that
    \begin{align}
        \frac{\sqrt{\varepsilon_k}}{
        1-\varepsilon_k} \le \frac{\sqrt{\varepsilon(1-\varepsilon)}}{
        1-2\varepsilon} = \alpha_\varepsilon.
    \end{align}
    Using this inequality and the obvious inequality 
    $G(\bw,\bmu)\le \|\sum_{i=1}^n w_i(\bX_i-\bmu)^{\otimes 2}\|^{1/2}_{\textup{op}}$, we infer from
    \eqref{mu:kbis} that
    \begin{align}
        \|\hat\bmu^{k}-\bmu^*\|_2 &\le \alpha_\varepsilon
        G(\hat\bw^k,\hat\bmu^{k-1})^{1/2} + \Xi\\
        &\le \alpha_\varepsilon
        \bigg\|\sum_{i=1}^n \hat w_i^{k}(\bX_i-\hat\bmu^{k-1}
        )^{\otimes 2}\bigg\|_{\textup{op}}^{1/2} + \Xi.
    \end{align}
    Then, using the fact that $\hat\bw^k$ is the minimizer
    of $\bw\mapsto\big\|\sum_{i=1}^n  w_i(\bX_i-\hat\bmu^{k-1}
    )^{\otimes 2}\big\|_{\textup{op}}$, we get
    \begin{align}
        \|\hat\bmu^{k}-\bmu^*\|_2 &\le \alpha_\varepsilon
        G(\hat\bw^k,\hat\bmu^{k-1})^{1/2} + \Xi\\
        &\le \alpha_\varepsilon\bigg\| \sum_{i=1}^n
        w_i^*(\bX_i-\hat\bmu^{k-1})^{\otimes 2} 
        \bigg\|_{\textup{op}}^{1/2} + \Xi,\label{mu:k1bis}
    \end{align}
    where $\bw^*_{\mathcal I}$ is the weight vector defined by 
    $w^*_i = \mathds 1(i\in\mathcal I)/|\mathcal I|$. Simple
    algebra (see the proof of \Cref{th:2} for more details) 
    yields
	\begin{align}
		\bigg\|\sum_{i=1}^n w_i^*(\bX_i-\hat\bmu^{k-1}
        )^{\otimes 2}\bigg\|_{\textup{op}}^{1/2} 
        &\le \bigg\|\sum_{i=1}^n w_i^*(\bX_i-\bmu^*)^{\otimes
        2}\bigg\|_{\textup{op}}^{1/2} + \big\|\bar{
        \bs X}_{\bw^*_{\mathcal I}}-\hat{\bmu}^{k-1}\big\|_2\\
        &\le\bigg\|\sum_{i=1}^n w_i^*(\bX_i-\bmu^*)^{\otimes 2} 
        \bigg\|_{\textup{op}}^{1/2}
        + \big\|\bar{\bzeta}_{\bw^*_{\mathcal I}}\big\|_2 + 
			\big\|\bmu^*-\hat{\bmu}^{k-1}\big\|_2\\
		&\le G(\bw^*_{\mathcal I}, \bmu^*)^{1/2} + \|\bSigma
		    \|_{\textup{op}} +\big\|\bar{\bzeta}_{\bw^*_{
		    \mathcal I}}\big\|_2 + \big\| \bmu^* - \hat{\bmu}^{k-1}
		    \big\|_2.
	\end{align}
	Combining with inequality \eqref{mu:k1bis}, we arrive at
	\begin{align}
        \|\hat\bmu^{k}-\bmu^*\|_2 
            &\le \alpha_\varepsilon G(\bw^*_{\mathcal I}, \bmu^*)^{1/2}  + \alpha_\varepsilon 
            + \alpha_\varepsilon\big\|\bar{\bzeta}_{
            \bw^*_{\mathcal I}}\big\|_2 + \alpha_\varepsilon
			\big\|\bmu^*-\hat{\bmu}^{k-1}\big\|_2 + \Xi.
    \end{align}    
    Thus, we have obtained the desired recurrent inequality
    \begin{align}\label{mu:k3bis}
        \|\hat\bmu^{k}-\bmu^*\|_2 
		&\le \alpha_\varepsilon \big\|\hat{\bmu}^{k-1}-\bmu^*\big\|_2
		+\alpha_\varepsilon + \tilde\Xi,
    \end{align}
    with 
    \begin{align}
        \tilde\Xi  &= \alpha_\varepsilon G(\bw^*_{\mathcal I}, 
        \bmu^*)^{1/2} +  \alpha_\varepsilon 
        \big\|\bar{\bzeta}_{\bw^*_{\mathcal I}}\big\|_2 + \Xi\\
        &\le \sqrt{5\varepsilon}\,\sup_{\bw,I}\bigg\|\sum_{i\in I}
            (\bw_{| I})_i (\bSigma - \bzeta_i\bzeta_i^\top) \bigg\|_{\rm op}^{1/2} 
			+ \sqrt{5\varepsilon}\,\big\|\bar{\bzeta}_{
            \bw^*_{\mathcal I}}\big\|_2 + \Xi.\label{xi:tildebis}
    \end{align}
    The last inequality above follows from the inequality
    $\alpha_\varepsilon\le \sqrt{5\varepsilon}$ as soon
    as $\varepsilon\le 0.3$. Unfolding the recurrent inequality
    \eqref{mu:k3bis}, we get
    \begin{align}\label{mu:K1bis}
        \|\hat\bmu^{K}-\bmu^*\|_2 
		&\le \alpha_\varepsilon^K \big\|\hat{\bmu}^{0}-\bmu^*\big\|_2
		+\frac{\alpha_\varepsilon + \tilde\Xi}{1-\alpha_\varepsilon}.
    \end{align}
    In view of \Cref{lemma0} and the condition $\varepsilon
    \le 0.3$, we get
    \begin{align}\label{mu:K2bis}
        \|\hat\bmu^{K}-\bmu^*\|_2 
		&\le \frac{5\alpha_\varepsilon^K}{n} \sum_{i=1}^n 
		\big\|\bzeta_i\big\|_2 +\frac{\alpha_\varepsilon 
		+\tilde\Xi}{1-\alpha_\varepsilon}.
    \end{align}
    We have already seen that with probability at least 
    $1-4\delta$,
    \begin{align}
        \frac{5\alpha_\varepsilon^K}{n} \sum_{i=1}^n 
		\big\|\bzeta_i\big\|_2 +\frac{\tilde\Xi}{1- 
		\alpha_\varepsilon}
		\le  \frac{A_5}{1-\alpha_\varepsilon}
		\Bigg(\sqrt{\frac{p+\log(4/\delta)}{n}} + \varepsilon\sqrt{\log(1/\varepsilon)}\Bigg),
    \end{align}
    where $A_5>1$ is a constant depending only on $\tau$. 
    In addition, we have $\alpha_\varepsilon \le \sqrt{
    5\varepsilon}$ and $\varepsilon\sqrt{\log(1/\varepsilon)}
    \le 0.6 \sqrt{\varepsilon}$, which lead to 
    $\alpha_\varepsilon + A_5\varepsilon\sqrt{\log(1/
    \varepsilon)}\le 3A_5\sqrt{\varepsilon}$. This completes 
    the proof of the theorem.

	\section{Lower bound for the Gaussian model with
	adversarial contamination}\label{App:D}
	In this section we provide the proof of the lower bound on the 
	expected risk in the setting of GAC model. To this end, we first denote by $\mathcal{M}^\text{HC}_n(\varepsilon, \bmu^*)$ the set of joint probability distributions $\bfP_n$ of the random vectors 
	$\bX_1, \dots, \bX_n$ coming from Huber's contamination model. 
	Recall that Huber's contamination model reads as follows
	\begin{equation}
	    \bX_1, \dots, \bX_n
	    \stackrel{\text{i.i.d.}}{\sim}
        \bfP_{\varepsilon, \bmu^*, \bfQ},
	\end{equation}
	where $\bfP_{\varepsilon, \bmu^*, \bfQ} =  (1- \varepsilon) \bfP_{\bmu^*} + \varepsilon \bfQ$. It is evident that with probability $\varepsilon^n$ all observations $\bX_1, \dots, \bX_n$ can be outliers, \textit{i.e.},  drawn from distribution $\bfQ$. 
	This means that 
	on an event of non-zero probability, it is impossible to have 
	a bounded estimation error. To overcome this difficulty, one can
	focus on {\it{Huber's deterministic contamination (HDC) model}} 
	(see \citep{bateni2020minimax}). In this model, it is assumed that the number of outliers is at most $\lceil n\varepsilon \rceil$. The 
	set of joint probability distributions $\bfP_n$ coming from Huber's deterministic contamination model is denoted by $\mathcal{M}^{\text{HDC}}_n(\varepsilon, \bmu^*)$. We define the worst-case risks for these two models 
	\begin{align}
	    R_{\max}^{\textup{HC}}(\hat\bmu_n, \varepsilon) = \sup_{\bmu^*}\sup_{\bfP_n \in \mathcal{M}^{\textup{HC}}_n(\varepsilon, \bmu^*)} R_{\bfP_n} (\hat\bmu_n, \bmu^*), \\ R_{\max}^{\textup{HDC}}(\hat\bmu_n, \varepsilon) = \sup_{\bmu^*}\sup_{\bfP_n \in \mathcal{M}^{\textup{HDC}}_n(\varepsilon, \bmu^*)} R_{\bfP_n} (\hat\bmu_n, \bmu^*).
	\end{align}
	
	For Huber's contamination model, the following lower bound holds true.
    \begin{theorem}[Theorem 2.2 from \citep{chen2018}]\label{thm:lb-huber} There exists some universal constant $C> 0$ such that 
    \begin{equation}
        \inf_{\hat\bmu_n} \sup_{\bmu^*} \sup_{\mathbf P_n \in \mathcal{M}^{\textup{HC}}_n(\varepsilon, \bmu^*)} \mathbf{P} \bigg( \| \hat\bmu_n - \bmu^* \|_2 \ge C \| \bSigma\|_{\textup{op}}^{1/2}\bigg[\sqrt{\frac{\rSigma}{n}} + \varepsilon \bigg]\bigg) \ge 1/2,
    \end{equation}
    for any $\varepsilon \in [0,1]$.
    \end{theorem}
    
    The combination of the lower bound from Theorem \ref{thm:lb-huber} and 
    the relation between the risks in $\mathcal{M}^{\text{HC}}_n(\varepsilon, \bmu^*)$ and $\mathcal{M}^{\text{HDC}}_n(\varepsilon, \bmu^*)$ (Prop. 1 from \citep{bateni2020minimax}) yields the desired lower bound for the 
    GAC model, as stated in the next proposition.
    
    \begin{proposition}[Lower bound for GAC]
    There exist some universal constants $c, \tilde{c} > 0$ such that 
    if $n > \tilde{c}$, then 
    \begin{equation}
        \inf_{\hat\bmu_n} R_{\max}(\hat\bmu_n, \bSigma, \varepsilon) \ge c \| \bSigma \|_{\textup{op}}^{1/2} \bigg(\sqrt{\frac{\rSigma}{n}} + \varepsilon\bigg). 
    \end{equation}
        
    \end{proposition}
    
    \begin{proof}
    
    First, notice that it is sufficient to prove this bound for $\varepsilon > \sqrt{{\rSigma}/{n}}$. Indeed, if $\varepsilon < \sqrt{{\rSigma}/{n}}$ then one gets the desired lower bound by simply comparing GAC model with the standard outlier free model. More formally, 
    \begin{align}
        \inf_{\hat\bmu_n} R_{\max}(\hat\bmu_n, \bSigma, \varepsilon) \ge \inf_{\hat\bmu_n} R_{\max}^{\textup{OF}}(\hat\bmu_n) = \sqrt{\frac{\tr{\bSigma}}{n}},
    \end{align}
    where $R_{\max}^{\textup{OF}}(\hat\bmu_n)$ is the worst case risk in the classical outlier free setting where the observations are i.i.d. and drawn from $\mathcal{N}_p(\bmu^*, \bSigma)$. 
    
    Since the GAC model is more general than HDC model, then one clearly has 
    \begin{align}\label{eq:ac-hdc}
        \inf_{\hat\bmu_n} R_{\max}(\hat\bmu_n, \bSigma, \varepsilon) \ge \inf_{\hat\bmu_n} R_{\max}^{\textup{HDC}}(\hat\bmu_n, \varepsilon).
    \end{align}
    On the other hand, in view of Eq. (4) from \citep{bateni2020minimax}, 
    for every $r>0$,
    \begin{align}
        R_{\max}^{\textup{HDC}}(\hat\bmu_n, \varepsilon) + r e^{-n\varepsilon/6} \ge \sup_{\mathbf P_n \in \mathcal{M}^{\textup{HC}}_n(0.5\varepsilon, \bmu^*)} r\mathbf{P} \big( \| \hat\bmu_n - \bmu \|_2 \ge r\big). 
    \end{align}
    Using the result of Theorem \ref{thm:lb-huber} and taking first $r = C \| \bSigma\|_{\textup{op}}^{1/2}\big[\sqrt{{\rSigma}/{n}} + \varepsilon\big]$ then infimum of both sides one arrives at 
    \begin{align}\label{eq:lb-proof1}
        \inf_{\hat\bmu_n} R_{\max}^{\textup{HDC}}(\hat\bmu_n, \varepsilon) \ge (C/4) \| \bSigma\|_{\textup{op}}^{1/2}\bigg[\sqrt{\frac{\rSigma}{n}} + \varepsilon\bigg]\big(1 - 2e^{-n\varepsilon/6}\big). 
    \end{align}
    Now, using $\varepsilon > \sqrt{\rSigma/n}$ one gets that $n\varepsilon > \sqrt{\rSigma n} \ge \sqrt{n}$. Therefore, taking $n \ge 36\log^2(4)$, one can check that \eqref{eq:lb-proof1} yields
    \begin{align}
        \inf_{\hat\bmu_n} R_{\max}^{\textup{HDC}}(\hat\bmu_n, 2\varepsilon) \ge (C/8) \| \bSigma\|_{\textup{op}}^{1/2}\bigg[\sqrt{\frac{\rSigma}{n}} + \varepsilon\bigg].
    \end{align}
Then, the final result follows from \eqref{eq:ac-hdc} with constants $c = C/8$ and $\tilde{c} = 36\log^2(4)$.
    
    \end{proof}

	\end{document}